\documentclass{lmcs} 
\pdfoutput=1

\usepackage{lastpage}
\lmcsdoi{17}{2}{4}
\lmcsheading{}{\pageref{LastPage}}{}{}%
{Jul.~09,~2019}{Apr.~14,~2021}{}


\usepackage[utf8]{inputenc}

\usepackage{amsmath}
\usepackage[bbgreekl]{mathbbol}
\usepackage{amssymb}
\usepackage{amsthm}
\usepackage[all]{xy}

\usepackage{amsfonts,amssymb,amsmath,qtree,stmaryrd,textcomp}

\usepackage{enumitem}

\usepackage{etex}
\usepackage{ marvosym }

\usepackage{amscd}
\usepackage[all]{xy}

\usepackage{qtree}

\usepackage{pifont,amsthm,tabularx,graphicx}
\usepackage{bbm}

\usepackage{stmaryrd}

\usepackage{tikz}
\usepackage{tikz-cd}

\usetikzlibrary{positioning}
\usetikzlibrary{shapes}

\usetikzlibrary{arrows}
\usetikzlibrary{calc}
\usepackage{tikz-3dplot}

\usepackage{relsize}

\usepackage{listings}
\usepackage{wrapfig}
\usepackage{thmtools}

\usepackage{epigraph}

\usepackage{framed}

\newcommand{\B}{\boldsymbol}
\newcommand{\C}[1]{\mathcal{#1}}
\newcommand{\D}[1]{\mathbb{#1}}

\newcommand{\Nat}{{\mathbb N}}
\newcommand{\Real}{{\mathbb R}}

\newcommand{\Fii}{{\mathbb F}}

\newcommand{\id}{\mathrm{id}}

\newcommand{\Const}{\mathrm{Const}}

\newcommand{\TOT}{\Leftrightarrow}

\newcommand{\To}{\Rightarrow}
\newcommand{\ot}{\leftarrow}

\newcommand{\CST}{\mathrm{CST}}

\newcommand{\BISH}{\mathrm{BISH}}

\newcommand{\sto}{\rightsquigarrow}

\newcommand{\eto}{\hookrightarrow}

\newcommand{\Fam}{\textnormal{\texttt{Fam}}}

\newcommand{\prb}{\textnormal{\textbf{pr}}}

\newcommand{\BST}{\mathrm{BST}}

\newcommand{\cof}{\textnormal{\texttt{cof}}}

\newcommand{\Disj}{\B ] \B [}

\newcommand{\pr}{\textnormal{\texttt{pr}}}

\newcommand{\Hom}{\textnormal{\texttt{Hom}}}

\newcommand{\MLTT}{\mathrm{MLTT}}

\newcommand{\Map}{\textnormal{\texttt{Map}}}

\newcommand{\Set}{\mathrm{Set}}

\newcommand{\CSFT}{\mathrm{CSFT}}

\newcommand{\two}{\D 2}

\newcommand{\Bic}{\mathrm{Bic}}
\newcommand{\BR}{\Bic(\Real)}
\newcommand{\BS}{\mathrm{BS}}
\newcommand{\Id}{\mathrm{Id}}
\newcommand{\Ind}{\mathrm{Ind}}
\newcommand{\Bis}{\mathrm{\textbf{Bis}}} 

\newcommand{\emb}{\hookrightarrow}

\newcommand{\Mor}{\mathrm{Mor}}

\newcommand{\lt}{\preccurlyeq}
\newcommand{\mt}{\succcurlyeq}

\newcommand{\Lim}{\mathrm{Lim}}

\newcommand{\fXY}{f \colon X \to Y}

\newcommand{\Cf}{\textnormal{\texttt{Cof}}}

\newcommand{\Even}{\textnormal{\texttt{Even}}}

\newcommand{\Odd}{\textnormal{\texttt{Odd}}}

\newcommand{\Spec}{\textnormal{\texttt{Spec}}}

\newcommand{\Cnt}{\textnormal{\texttt{Cont}}}

\newcommand{\eql}{\textnormal{\texttt{eql}}}

\newcommand{\Eql}{\textnormal{\texttt{Eql}}}

\theoremstyle{plain} 

\begin{document}

\title[Direct spectra of Bishop spaces and their limits]{Direct spectra of Bishop spaces and their limits}

\author[I.~Petrakis]{Iosif Petrakis}	
\address{Mathematics Institute, Ludwig-Maximilians-Universit\"{a}t M\"{u}nchen, 
Theresienstr. 39, 80333, Munich}	
\email{petrakis@math.lmu.de}  

%





\begin{abstract}
\noindent We apply fundamental notions of Bishop set theory ($\BST$), an informal theory that complements
Bishop's theory of sets, to the theory of Bishop spaces, a 
function-theoretic approach to constructive topology. Within $\BST$ we develop  
the notions of a direct family of sets, of a direct spectrum of Bishop spaces, of the direct limit of a 
direct spectrum of Bishop spaces, and of the inverse limit of a contravariant direct 
spectrum of Bishop spaces. Within the extension of Bishop's informal system of constructive mathematics $\BISH$ 
with inductive definitions with rules of countably
many premises, we prove the fundamental theorems on the direct and inverse limits of spectra 
of Bishop spaces and the duality principle between them.  
\end{abstract}

\maketitle

\section{Introduction}
\label{S:one}

\subsection{Bishop's set theory}
\label{subsec: bst}

The theory of sets underlying Bishop-style constructive mathematics $(\BISH)$ was only sketched in Chapter 3 of
Bishop's seminal book~\cite{Bi67}. Since Bishop's central aim in~\cite{Bi67} was to show 
that a large part of advanced mathematics can be done within a constructive and 
computational framework that does not contradict the classical practice, the inclusion of a 
detailed account of the set-theoretic foundations of $\BISH$ could possibly be against the effective 
delivery of his message.

The Bishop-Cheng measure theory, developed in~\cite{BC72}, was very
different from the measure theory of~\cite{Bi67}, and the inclusion of an enriched version of the former 
into~\cite{BB85}, the book on $\BISH$ that Bishop co-authored with Bridges later, 
affected the corresponding Chapter 3 in two main respects. First, the inductively defined notion of the set 
of Borel sets generated by a given family of complemented subsets of a set $X$, with respect to a set 
of real-valued functions on $X$, was excluded, as unnecessary, and, second, the operations on the 
complemented subsets of a set $X$ were defined differently, and in accordance to the needs 
of the new measure theory. 
 
Yet, in both books many issues were left untouched, a fact that often was a source of confusion.
In many occasions, especially
in the measure theory of~\cite{BC72} and~\cite{BB85}, the powerset was treated as a set, while in the
measure theory of~\cite{Bi67}, Bishop generally avoided the powerset by using appropriate families of 
subsets instead. In later works of Bridges and Richman, like ~\cite{BR87} and~\cite{MRR88}, the powerset was
clearly used as a set, in contrast though, to the predicative spirit of~\cite{Bi67}. The concept of a family of
sets indexed by a (discrete) set, was asked to be defined in~\cite{Bi67} (Exercise 2, p.~72), and a definition, 
attributed to Richman, was given in~\cite{BB85} (Exercise 2, p.~78). An elaborate study of this concept 
within $\BISH$ is missing though, despite its central character in the measure theory of~\cite{Bi67}, its 
extensive use in the theory of Bishop spaces~\cite{Pe15} and in abstract constructive algebra~\cite{MRR88}.
Actually, in~\cite{MRR88} Richman introduced the more general notion of a family of objects of 
a category indexed by some set, but the categorical component in the resulting mixture of Bishop's set theory and 
category theory was not explained in constructive terms\footnote{As it was done in 
the formulation of category theory in homotopy type theory (see Chapter 9 in~\cite{HoTT13}).}.

Bishop briefly discussed some formal aspects of $\BISH$ in~\cite{Bi70}, where a variant
of G\"odel's $T$ was proposed as a formal system for $\BISH$, while in his unpublished work~\cite{Bi68a} 
he elaborated a version of dependent type theory with one universe, in order to formalise $\BISH$.
The various set-theoretic formalisations of $\BISH$, developed by Myhill~\cite{My75}, Friedman~\cite{Fr77}, 
Aczel~\cite{AR10}, Feferman~\cite{Fe79}, Beeson~\cite{Be81}, and Greenleaf~\cite{Gr81}, were, roughly speaking, 
quite influenced by Zermelo-Fraenkel set theory, and ``top-down'' approaches to $\BISH$, with many ``unexpected'' 
features\footnote{A detailed presentation of the unexpected features of these formalisations is going to be
found in~\cite{Pe20}.} with respect to its practice.

The type-theoretic interpretation of Bishop's set theory into the theory of setoids (see especially the 
work of Palmgren~\cite{Pa05}-\cite{PW14}) has become nowadays the standard way to understand Bishop sets.
The identity type of Martin-L\"of's type theory ($\MLTT$) (see~\cite{ML98}), expresses, in a proof-relevant way,
the existence of the least reflexive relation on a type, a fact with no counterpart in Bishop's set theory.
As a consequence, the free setoid on a type is definable (see~\cite{Pa14}, p. 90), and the presentation axiom
in setoids is provable. Moreover, in $\MLTT$ the families of types over a type $I$ form the type $I \to \C U$, 
which belongs to the successor universe $\C U{'}$ of $\C U$. In Bishop's set theory though, where only one
universe of sets is used, the set-character of the totality of all families of sets indexed by some set $I$
is questionable from the predicative point of view (see our comment after Definition~\ref{def: map}).

\subsection{Bishop Set Theory}
\label{subsec: BST}

Bishop set theory $(\BST)$ is an informal, constructive theory of totalities and assignment routines
that serves as a reconstruction of Bishop's set theory. Its aim is, first, to fill in the ``gaps'' in 
Bishop's account of the set theory underlying $\BISH$, and,
second, to serve as an intermediate step between Bishop's informal set theory
and a \textit{suitable} i.e., an \textit{adequate} and \textit{faithful}, in the sense of Feferman~\cite{Fe79},
formalisation of $\BISH$. To assure faithfulness, we need to use concepts or principles that appear,
explicitly or implicitly, in $\BISH$.

Next we describe briefly the features of $\BST$ that ``complete'' Bishop's theory of sets in~\cite{Pe20}.\\[2mm]
\textbf{1. Explicit use of a universe of sets.} Bishop used a universe of sets only implicitly. E.g., he ``roughly'' 
describes in~\cite{Bi67}, p.~72, a set-indexed family of sets as 
\begin{quote}
$\ldots$ a rule which assigns to each $t$ in a discrete set $T$ a set $\lambda(t)$.
\end{quote}
Every other rule, or assignment routine mentioned by Bishop is from one given totality, the domain of the rule, 
to some other totality, its codomain. The only way to make the rule of a family of sets compatible with this 
pattern is to employ a totality of sets.
In the unpublished manuscript~\cite{Bi68a} Bishop explicitly used a universe in his formulation 
of dependent type theory as a formal system for $\BISH$.
Here we use an open-ended totality $\D V_0$ of sets, which contains the primitive 
set $\Nat$ and all defined (predicative) sets. $\D V_0$ itself is not a set, but a class. It is a notion instrumental to the definition
of dependent operations, and of a set-indexed family of sets.\\[2mm]
\textbf{2. Clear distinction between sets and classes.} A class is a totality defined through a membership condition 
in which a quantification over $\D V_0$ occurs. The powerset $\C P(X)$ of a set $X$, the totality $\C P^{\Disj}(X)$ 
of complemented subsets of a set $X$, and the totality $\C F(X,Y)$ of partial functions from a set $X$ to a set $Y$ 
are characteristic examples of classes. A class is never used here as the domain of an assignment routine, but only as
a codomain of an assignment routine. \\[2mm]
\textbf{3. Explicit use of dependent operations.} The standard view, even among practicioners of Bishop-style 
constructive mathematicians, is that dependency is not necessary to $\BISH$.
Dependent functions though, do appear explicitly in Bishop's definition of the intersection $\bigcap_{t \in T}
\lambda (t)$ of a family $\lambda$ of subsets of some set $X$ indexed by an inhabited set $T$
(see~\cite{Bi67}, p.~65, and ~\cite{BB85}, p.~70). As we try to show in~\cite{Pe20},
the elaboration of dependency within $\BISH$ is only fruitful 
to it. Dependent functions are not only necessary to the definition of products of families 
of sets indexed by an arbitrary set, but in many areas of constructive mathematics. 
As already mentioned, dependency is formulated in Bishop's type theory~\cite{Bi68a}. The somewhat ``silent'' role of 
dependency within Bishop's set theory is replaced by a central role within $\BST$.\\[2mm]
\textbf{4. Elaboration of the theory of families of sets.} With the use of the universe $\D V_0$, of the notion of 
a non-dependent assignment routine $\lambda_0$ from an index-set $I$ to $\D V_0$, and of a certain dependent 
operation $\lambda_1$, we 
define explicitly in Definition~\ref{def: familyofsets} the notion of a family of sets indexed by $I$. 
Although an $I$-family of sets is a certain function-like object, it can also be understood as an object with level 
the level of a set plus one. The corresponding notion of a ``function'' from an $I$-family $\Lambda$
to an $I$-family $M$ is that of a family-map. Operations between sets generate operations between families of 
sets and their family-maps. If the index-set $I$ is a directed set, the corresponding notion of a family of sets 
over it is that of a direct family of sets. 
\textit{Families of subsets} of a 
given set $X$ over an index-set $I$ are special $I$-families that deserve an independent treatment. Families 
of equivalence classes, families of partial functions, families of complemented subsets and direct families 
of subsets are some of the variations of set-indexed families of subsets that are studied in~\cite{Pe20} with 
many applications in Bishop-style constructive mathematics.\\[2mm]
Here we apply the general theory of families of sets, in order \textbf{to develop the theory of spectra of Bishop spaces.} 
A Bishop space is a constructive, function-theoretic alternative to the notion of a topological space. 
A Bishop topology $F$ on a set $X$ is a subset of the set $\D F(X)$ of real-valued functions on $X$
that includes the constant functions and it is closed 
under addition, composition with Bishop continuous functions from $\Real $ to $\Real$, and uniform limits. 
Hence, in contrast to topological spaces, continuity of real-valued functions is a primitive notion and a concept
of open set comes a posteriori. A Bishop topology on a set can be seen as an abstract and 
constructive approach to the ring of continuous functions $C(X)$ of a topological space $X$. 
Associating appropriately a Bishop topology to the set $\lambda_0(i)$ of a family of sets over a set $I$, 
for every $i \in I$, the notion of a spectrum of Bishop spaces is defined. If $I$ is a directed set, we get a 
direct spectrum. The theory of direct spectra of Bishop spaces and their limits is developed here 
in complete analogy to the classical theory of spectra of topological spaces and their limits (see~\cite{Du66}, Appendix
Two). The constructive theory of spectra of other structures, like groups, or rings, or modules, can be developed 
along the same lines.

\subsection{Outline of this paper}
\label{subsec: bst}

We structure this paper\footnote{Here we continue our work in~\cite{Pe19d}, where
we showed the distributivity of dependent sums over dependent products within $\BST$ (there we used the acronym 
$\CSFT$ instead).} as follows:
\begin{enumerate}
\item In Section~\ref{sec: basic} we present the fundamental notions of BST that are used in the subsequent
sections. 
\item In Section~\ref{sec: famsets} we define the families of Bishop sets indexed by some set $I$ and the family-maps
between them. The corresponding $\sum$-and $\prod$-sets are introduced.
\item In Section~\ref{sec: directed} we introduce the notion of a directed set with a modulus of directedness and the 
notion of a cofinal subset of a directed set with a modulus of cofinality. These moduli help us avoid the use of some 
choice-principle in later proofs.
\item In Section~\ref{sec: dirfamsets} we define the families of Bishop sets indexed by some directed set $(I, \lt)$ and the 
family-maps between them. The corresponding $\sum^{\lt}$-and $\prod^{\lt}$-sets are introduced.
\item In Section~\ref{sec: bspaces} we include the basic notions and facts on Bishop spaces that are 
used in the subsequent sections.
\item In Section~\ref{sec: direxternalspectra} we define the notion of a direct spectrum of Bishop spaces, 
the notion of a continuous, 
direct spectrum-map, and we define a canonical Bishop topology on the direct sum of Bishop spaces.
\item In Section~\ref{sec: fameqclass} we define the families of subsets of a given set $X$ 
indexed by some set $I$, and the family-maps between them. We also define sets of subsets of a set $X$ indexed by some set $I$,
and set-indexed families of equivalence classes of an equivalence structure. These notions are used in the definition of 
the direct limit of a (covariant) spectrum of Bishop spaces.
\item In Section~\ref{sec: directlimit} we define the direct limit $\underset{\to} \Lim \C F_i$ of a direct spectrum 
of Bishop spaces, we prove the universal property of the direct limit for $\underset{\to} \Lim \C F_i$
(Proposition~\ref{prp: universaldirect}), the existence of a unique map of the limit spaces from a 
spectrum-map (Theorem~\ref{thm: directspectrummap1}), the cofinality theorem for direct limits 
(Theorem~\ref{thm: cofinal2}), and the existence of a Bishop bijection for the product of spectra of Bishop
spaces (Proposition~\ref{prp: proddirect}).
\item In Section~\ref{sec: inverselimit} we define the inverse limit $\underset{\ot} \Lim \C F_i$ of a contravariant 
direct spectrum 
of Bishop spaces, we prove the universal property of the inverse limit for $\underset{\ot} \Lim \C F_i$
(Proposition~\ref{prp: universalinverse}), the existence of a unique map of the limit spaces from a 
spectrum-map (Theorem~\ref{thm: inverselimitmap}), the cofinality theorem for inverse limits 
(Theorem~\ref{thm: cofinal3}), and the existence of a Bishop morphism for the product of inverse 
spectra of Bishop spaces (Proposition~\ref{prp: prodinverse}). 
\item In Section~\ref{sec: duality} we prove the duality principle between the inverse and direct limits of Bishop spaces
(Theorem~\ref{thm: duality1}). 
\end{enumerate}

The above results form a constructive counterpart to the theory of limits of topological spaces, as this is 
presented in~\cite{Du66}, Appendix Two. As in the classic textbook of
Dugundji, the choice of presentation is \textit{non-categorical} and \textit{purely topological}. 
By the first we mean that 
the language of category theory is avoided, and by the second, that the limit constructions are 
done directly for spaces. Our central aim is to present the basic theory 
of spectra of Bishop spaces \textbf{within $\BST$}. The use of categorical arguments would be consistent 
with our aim, only if the corresponding category theory was formulated within BST. In~\cite{MRR88} Richman 
used the notion of functor, to define the concept of a set-indexed family of sets. 
The problem with Richman's approach was that the category theory involved was not explained in constructive terms, 
and its relation to $\BISH$ was left unclear. Here we avoid the use of categorical concepts in order to establish
our results, as we want to work exclusively within BST. Actually, in~\cite{Pe20}, pp.~80--83 we use the notion of a set-indexed 
family of sets  in order to \textit{define} the notion of a functor between small categories within $\BST$!

The theory of Bishop spaces, that was only sketched by Bishop in~\cite{Bi67}, and revived 
by Bridges in~\cite{Br12}, and Ishihara in~\cite{Is13}, was developed by the author in~\cite{Pe15}-\cite{Pe19c} and 
~\cite{Pe20b},~~\cite{Pe20c}. Since it makes use of inductive definitions with rules of countably many premises,
for their study we work within $\BST^*$, which is $\BST$ extended with such inductive definitions. 
Our concepts and results avoid the use even of countable choice (CC). Although
practicioners of $\BISH$ usually embrace dependent choice, hence CC, using non-sequential or non-choice-based arguments
instead, forces us to formulate ``better''
concepts and find ``better'' proofs. This standpoint was advocated first by Richman in~\cite{Ri01}. There are results 
in~\cite{BB85} that require CC in their proof, but we make no use of them here. E.g., our proofs of the 
cofinality theorems (Theorem~\ref{thm: cofinal2} and Theorem~\ref{thm: cofinal3}) are choice-free as we use a 
proof-relevant definition of a cofinal subset. Notice that the formulation of the universal properties of the 
limits of spectra of Bishop spaces is \textit{impredicative}, as it requires quantification over the \textit{class} 
of Bishop spaces\footnote{A predicative formulation of a universal property can be given, if one is restricted to a given
set-indexed family of Bishop spaces.}. This is one reason we avoided the use of the universal properties in our proofs.

A formal system for $\BISH$ extended with such definitions is Myhill's formal system $\CST^*$
where $\CST^*$ is Myhill's extension of his formal system of constructive set theory 
$\CST$ with inductive definitions (see~\cite{My75}). A variation of $\CST^*$ is Aczel's system CZF together
with a weak version of Aczel's regular extension axiom (REA), 
to accommodate these inductive definitions (see~\cite{AR10}). 

For all basic facts on constructive analysis we refer to~\cite{BB85},
for all results on Bishop spaces that are used here without proof we refer to~\cite{Pe15}, and for all results
on $\BST$ that are not shown here we refer to~\cite{Pe20}.

\section{Fundamentals of Bishop set theory}
\label{sec: basic}

The logical framework of $\BST$ is first-order intuitionistic logic with equality (see~\cite{SW12}, chapter~1). 
The primitive equality between terms is 
denoted by $s := t$\index{$s := t$}, and it is understood as a \textit{definitional}, or \textit{logical}, 
equality\index{definitional equality}\index{logical equality}. I.e., we read the equality $s := t$ as 
``the term $s$ is by definition equal to the term $t$''. If $\phi$ is an appropriate formula, for the standard axiom 
for equality $[a := b \ \& \ \phi(a)] \To \phi(b)$ we use the notation $[a := b \ \& \ \phi(a)] :\To \phi(b)$.
The equivalence notation $:\TOT$ is understood in the same way.
The set $(\Nat, =_{\Nat}, \neq_{\Nat})$ of natural numbers, where its canonical equality is given by\index{$=_{\Nat}$} 
$m =_{\Nat} n :\TOT m := n$, and its canonical inequality by $m \neq_{\Nat} n :\TOT \neg(m =_{\Nat} n)$, is primitive.
The standard Peano-axioms are associated to $\Nat$. 

A global operation $( \cdot, \cdot)$ of pairing\index{pairing} is also considered primitive. I.e., if 
$s, t$ are terms, their pair $(s,t)$ is a new term. The corresponding equality axiom is  
$(s, t) := (s{'}, t{'}) :\TOT s := s{'} \ \& \ t := t{'}$. The $n$-tuples of given terms, for every $n$ larger
than $2$, are definable. The global projection routines $\prb_1(s, t) := s$ and $\prb_2(s, t) := t$ are also considered
primitive. The corresponding global projection routines for any $n$-tuples are definable.

An undefined notion of mathematical construction, or algorithm, or of finite routine is considered as primitive.
The main objects of $\BST$ are totalities\index{totality} and assignment 
routines\index{assignment routine}. Sets are special totalities and 
functions are special assignment routines, where an assignment routine is a special finite routine. 
All other equalities in $\BST$ are equalities on totalities defined though an equality condition.
A predicate\index{predicate on a set} on a set $X$ is a bounded formula $P(x)$ with $x$ a free variable ranging through $X$,
where a formula is bounded\index{bounded formula}, if every quantifier occurring in it is over a given set.

\begin{defi}\label{def: totalities}
 \normalfont (i) 
\itshape A \textit{primitive set}\index{primitive set} $\D A$ is a totality with a given membership 
 $x \in \D A$, and a given equality $x =_{\D A} y$, that satisfies axiomatically the properties of 
 an equivalence relation. The set $\Nat$ of natural numbers is the only primitive set considered here.\\[1mm]
 \normalfont (ii) 
\itshape A $($non-inductive$)$\textit{defined totality}\index{defined totality} $X$ is defined by a membership
condition $x \in X : \TOT \C M_X(x),$ where 
$\C M_X$ is a formula with $x$ as a free variable. \\[1mm]
\normalfont (iii)
 \itshape There is a special ``open-ended'' defined totality $\D V_0$, which is called the universe of $($predicative$)$ 
 sets. $\D V_0$
 is not defined through a membership-condition, but in an open-ended way. When we say that a defined totality $X$ is
 considered to be a set we ``introduce'' $X$ as an element of $\D V_0$. We do not add the corresponding induction, 
 or elimination principle, as we want to leave open the possibility of adding new sets in $\D V_0$. \\[1mm]
 \normalfont (iv) 
\itshape A \textit{defined preset}\index{defined preset} $X$, or simply, a preset\index{preset}, is a defined totality 
$X$ the membership condition $\C M_X$ of which expresses a construction. \\[1mm]
\normalfont (v) 
\itshape A \textit{defined totality $X$ with equality}\index{defined totality with equality}, or simply, a
totality $X$ with equality\index{totality with equality} is a defined totality $X$ equipped with an equality condition 
$x =_X y : \TOT \C E_X(x, y)$, where $\C E_X(x,y)$ is a formula with free variables $x$ and $y$ that 
satisfies the conditions of an equivalence relation.\\[1mm]
\normalfont (vi) 
\itshape A \textit{defined set}\index{defined set} is a preset with a given equality, specified by a bounded formula.\\[1mm]
\normalfont (vii) 
\itshape A \textit{set}\index{set} is either a primitive set, or a defined set.\\[1mm]
\normalfont (viii) 
\itshape A totality is a \textit{class}\index{class}, if it is the universe $\D V_0$, or if 
quantification over $\D V_0$ occurs in its membership condition.
\end{defi}

\begin{defi}\label{def: extsubset}
A bounded formula on a set $X$ is called an \textit{extensional property} on $X$\index{extensional property on a set}, if 
\[\forall_{x, y \in X}\big([x =_{X } y \ \& \ P(x)] \To P(y)\big). \]
The totality $X_P$\index{$P_X$} generated by $P(x)$ is defined by $x \in X_P : \TOT x \in X \ \& \ P(x)$, 
\[ x \in X_P : \TOT x \in X \ \& \ P(x), \]
and the equality 
of $X_P$ is inherited from the equality 
of $X$. We  also write $X_P := \{x \in X \mid P(x)\}$,
$X_P$ is considered to be a set, and it is called the  
\textit{extensional subset}\index{extensional subset} of $X$ generated by $P$.
\end{defi}

Using the properties of an equivalence relation, it is immediate to show that an equality condition
$\C E_X(x,y)$ on a totality $X$ is an extensional property on the product $X \times X$ i.e., 
$[(x, y) =_{X \times X} (x{'}, y{'}) \ \&\ x =_X y] \To x{'} =_X y{'}$. We consider 
the following extensional subsets of 
$\Nat$:
$$\D 1 := \{x \in \Nat \mid x =_{\Nat} 0\} := \{0\},$$ 
$$\D 2 := \{x \in \Nat \mid x =_{\Nat} 0 \ \vee x =_{\Nat} 1\} := \{0, 1\}.$$
Since $n =_{\Nat} m :\TOT n := m$, the property $P(x) :\TOT x =_{\Nat} 0 \ \vee x =_{\Nat} 1$ is extensional.

\begin{defi}\label{def: diagonal}
If $(X, =_X)$ is a set, its \textit{diagonal}\index{diagonal of a set} 
\index{$D(X, =_X)$} 
is the extensional subset of $X \times X$
\[ D(X, =_X) := \{(x, y) \in X \times X \mid x =_X y\}. \]
If $=_X$ is clear from the context, we just write $D(X)$\index{$D(X)$}. 
\end{defi}

\begin{defi}\label{def: ndar}
Let $X, Y$ be totalities. A \textit{non-dependent assignment routine}\index{non-dependent assignment routine}
$f$ from $X$ to 
$Y$, in symbols $f \colon X \sto Y$\index{$f \colon X \sto Y$}, is a finite routine that assigns an element $y$ of $Y$
to each given element $x$ of $X$. 
In this case we write $f(x) := y$. If $g \colon X \sto Y$, let 
\[f := g : \TOT \forall_{x \in X}\big(f(x) := g(x)\big). \]
If $f := g$, we say that  $f$ and $g$ are \textit{definitionally equal}\index{definitionally equal functions}.
If $(X, =_X)$ and $(Y, =_Y)$ are sets, an \textit{operation} from $X$ to $Y$ is a non-dependent assignment routine 
from $X$ to $Y$, while a \textit{function} from $(X, =_X)$ to $(Y, =_Y)$, in symbols 
$f \colon X \to Y$\index{$f \colon X \sto Y$},
is an operation from $(X, =_X)$ to $(Y, =_Y)$ that respects equality i.e.,
\[\forall_{x, x{'} \in X}\big(x =_X x{'} \To f(x) =_Y f(x{'})\big). \]
If $f \colon X \sto Y$ is a function from $X$ to $Y$, we say
that $f$ is a function, without mentioning the expression ``from $X$ to $Y$''. If $X$ is a set, the \textit{identity 
function} $\id_X \colon X \to X$ is defined by the rule $x \mapsto x$, for every $x \in X$.
A function $\fXY$ is an \textit{embedding}\index{embedding}, in symbols 
$f \colon X \hookrightarrow Y$\index{$f \colon X \hookrightarrow Y$}, if 
\[\forall_{x, x{'} \in X}\big( f(x) =_Y f(x{'}) \To x =_X x{'}). \]
Let $X, Y$ be sets. The totality $\D O(X, Y)$\index{$\D O(X, Y)$} of
 operations from $X$ to $Y$  is equipped with the following canonical equality: 
 \vspace{-2mm}
 \[ f =_{\D O(X, Y)} g : \TOT \forall_{x \in X}\big(f(x) =_Y g(x)\big). \]
 The totality $\D O(X, Y)$ is considered to be a set. The set $\D F(X, Y)$\index{$\D F(X, Y)$} of functions
 from $X$ to $Y$  is defined by separation on $\D O(X, Y)$ through the extensional property $P(f) :\TOT
 \forall_{x, x{'} \in X}\big(x =_X x{'} \To f(x) =_Y f(x{'})\big)$. The equality $=_{\D F(X, Y)}$
 is inherited from $=_{\D O(X, Y)}$.

\end{defi}

The canonical equality on $\D V_0$ is defined by
\[ X =_{\D V_0} Y :\TOT \exists_{f \in \D F(X,Y)}\exists_{g \in \D F(Y,X)}\big(g \circ f = \id_X \ \& \ 
f \circ g = \id_Y\big).\]
In this case we write $(f, g) : X =_{\D V_0} Y$.

\begin{defi}\label{def: dependops}
Let $I$ be a set and $\lambda_0 \colon I \sto \D V_0$\index{$\lambda_0 \colon I \sto \D V_0$} a non-dependent assignment 
routine from $I$ to $\D V_0$. 
A \textit{dependent operation}\index{dependent operation} $\Phi$ over $\lambda_0$, in symbols
\[ \Phi \colon \bigcurlywedge_{i \in I} \lambda_0 (i), \]
is an assignment routine that assigns to each element $i$ in $I$ an element $\Phi(i)$ in the set
$\lambda_0(i)$. If $i \in I$, we call $\Phi(i)$ the $i$-\textit{component}\index{component of a dependent operation}
of $\Phi$, and we also use the notation $\Phi_i := \Phi(i)$\index{$\Phi(i)$}\index{$\Phi_i$}.
An assignment routine\index{assignment routine} is either a non-dependent assignment routine, or a dependent operation 
over some non-dependent assignment routine from a set to the universe.
If $\Psi \colon \bigcurlywedge_{i \in I} \lambda_0 (i)$, 
$\Phi := \Psi :\TOT \forall_{i \in I}\big(\Phi_i := \Psi_i\big).$
If $\Phi := \Psi$, we say that $\Phi$ and $\Psi$ are definitionally equal\index{definitionally
equal dependent operations}. Let 
$\D A(I, \lambda_0)$\index{$\D A(I, \lambda_0)$} be the totality of dependent operations over $\lambda_0$,
equipped with the 
canonical equality $\Phi =_{\D A(I, \lambda_0)} \Psi :\TOT \forall_{i \in I}\big(\Phi_i =_{\lambda_0(i)} \Psi_i\big)$.
The totality $\D A(I, \lambda_0)$ is considered to be a set. 

\end{defi}

\begin{defi}\label{def: subset}
If $X$ is a set, a \textit{subset} of $X$ is a pair $(A, i_A^X)$, where $A$ is a set and $i_A^X: A \emb X$ is
an embedding. If $(A, i_A^X)$ and $(B, i_B^X)$ are subsets of $X$, we say that $A$ is a \textit{subset} of $B$, and
we write $A \subseteq B$, if 
there is $f : A \to B$ such that the following diagram commutes
\begin{center}
\begin{tikzpicture}

\node (E) at (0,0) {$A$};
\node[right=of E] (B) {};
\node[right=of B] (F) {$B$};
\node[below=of B] (A) {$X$.};

\draw[->] (E)--(F) node [midway,above] {$f$};
\draw[right hook->] (E)--(A) node [midway,left] {$i_A^X \ $};
\draw[left hook->] (F)--(A) node [midway,right] {$\ i_B^X$};

\end{tikzpicture}
\end{center} 
In this case we write $f : A \subseteq B$, and usually we write $A$ instead of $(A, i_A^X)$, and 
$i_A$ instead of $i_A^X$, if $X$ is clear from the context.
The totality of the subsets of $X$ is the \textit{powerset} $\C P(X)$, and it is equipped with the equality
$(A, i_A) =_{\C P(X)} (B, i_B) :\TOT A \subseteq B \ \& \ B \subseteq A$. If $f : A \subseteq B$ and 
$g : B \subseteq A$, we write $(f, g) : A =_{\C P(X)} B$.
If $f : A \subseteq B$, then $f$ is an embedding: if $a, a{'} \in A$ are such that $f(a) =_B f(a{'})$, then 
$i_B(f(a)) =_X i_B(f(a{'})) \TOT i_A(a) =_X i_A (a{'})$, which implies $a =_A a{'}$. 
\end{defi}

If 
$(f, g) : A =_{\C P(X)} B$, all the following diagrams commute
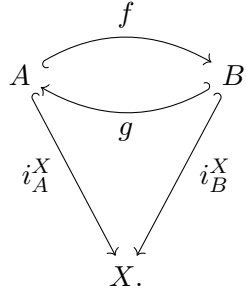
\begin{center}
\begin{tikzpicture}

\node (E) at (0,0) {$A$};
\node[right=of E] (B) {};
\node[right=of B] (F) {$B$};
\node[below=of B] (C) {};
\node[below=of C] (A) {$X$.};

\draw[left hook->,bend left] (E) to node [midway,above] {$f$} (F);
\draw[left hook->,bend left] (F) to node [midway,below] {$g$} (E);
\draw[right hook->] (E)--(A) node [midway,left] {$i_A^X \ $};
\draw[left hook->] (F)--(A) node [midway,right] {$ \ i_B^X$};

\end{tikzpicture}
\end{center} 
Moreover, the internal equality of subsets implies their external equality as sets  i.e., 
$$(f, g) : A =_{\C P(X)} B \To (f, g) : A =_{\D V_0} B.$$
If $a \in A$, then $i_A(g(f(a))) =_X i_B(f(a)) =_X i_A (a)$, hence $g(f(a)) =_A a$, and 
$g \circ f =_{\D F(A, A)} \id_A$. Similarly, $f \circ g =_{\D F(B)} \id_B$. Since the membership condition 
of $\C P(X)$ requires quantification over $\D V_0$, the totality $\C P(X)$ is a class.  
If $X_P$ is an extensional subset of $X$, then 
$(X_P, i_{X_P}) \subseteq X$, where $i_{X_P} \colon X_P \to X$ is defined by $x \mapsto x$, for every $x \in X_P$.

\section{Set-indexed families of sets}
\label{sec: famsets}

An $I$-family of sets is an assignment routine $\lambda_0 : I \sto \D V_0$ that 
behaves like a function i.e., if $i =_I j$, then $\lambda_0(i) =_{\D V_0} \lambda_0 (j)$. The 
following definition is an exact formulation of this rough description\footnote{In~\cite{MRR88} Richman 
saw a set $I$ as a category with objects its elements and 
$\Hom_{=_I}(i, j) := \{x \in \{0\} \mid i =_I j\}$, for every $i, j \in I$. 
If we view $\D V_0$ 
as a category with objects its elements and 
$\Hom_{=_{\D V_0}}(X, Y) := \big \{(f, f{'}) : \D F(X, Y) \times \D F(Y, X) \mid (f, f{'}) : 
X =_{\D V_0} Y \big\},$ for every $X, Y \in \D V_0$, then
an $I$-family of sets is a
functor from the category $I$ to the category $\D V_0$. Notice that in the definitions of $\Hom_{=_I}(i, j)$
and of $\Hom_{=_{\D V_0}}(X, Y)$ the properties $P_I(x) := i =_I j$ and $Q_{X,Y}(f, f{'}) := (f, f{'}) : 
X =_{\D V_0} Y$ are extensional.}.

\begin{defi}\label{def: familyofsets}
If $I$ is a set, 
a \textit{family of sets}\index{family of sets} indexed by $I$, or an $I$-\textit{family 
of sets}\index{$I$-family of sets}, is a pair $\Lambda := (\lambda_0, \lambda_1)$, where
$\lambda_0 : I \sto \D V_0$, and\footnote{More accurately, we should write 
$\lambda_1 : \bigcurlywedge_{z \in D(I)}\D F\big(\lambda_0(\pr_1(z)), \lambda_0(\pr_2(z))\big)$.} 
$$\lambda_1 : \bigcurlywedge_{(i, j) \in D(I)}\D F\big(\lambda_0(i), \lambda_0(j)\big), 
\ \ \ \lambda_1(i, j) =: \lambda_{ij}, \ \ \ (i, j) \in D(I),$$
such that the following conditions hold:\\[1mm]
$(a)$ For every $i \in I$, we have that $\lambda_{ii} := \id_{\lambda_0(i)}$.\\[1mm]
$(b)$ If $i =_I j$ and $j =_I k$, the following diagram commutes
\begin{center}
\begin{tikzpicture}

\node (E) at (0,0) {$\lambda_0(j)$};
\node[right=of E] (F) {$\lambda_0(k).$};
\node [above=of E] (D) {$\lambda_0(i)$};

\draw[->] (E)--(F) node [midway,below] {$\lambda_{jk}$};
\draw[->] (D)--(E) node [midway,left] {$\lambda_{ij}$};
\draw[->] (D)--(F) node [midway,right] {$\ \lambda_{ik}$};

\end{tikzpicture}
\end{center}
If $i =_I j$, we call the function $\lambda_{ij}$ the 
\textit{transport map}\footnote{We draw this term from MLTT.} from $\lambda_0 (i)$ to $\lambda_0 (j)$. 
We call the assignment routine
$\lambda_1$ the \textit{modulus of function-likeness of} 
$\lambda_0$\footnote{Since $\lambda_{ii} = \lambda_{ji} \circ \lambda_{ij}$ and 
$\lambda_{jj} = \lambda_{ij} \circ \lambda_{ji}$, we have that 
$(\lambda_{ij}, \lambda_{ji}) : \lambda_0 (i) =_{\D V_0} \lambda_0 (j)$.}. An $I$-family of sets is called
an $I$-\textit{set} of sets, if 
$$\forall_{i, j \in I}\big(\lambda_0(i) =_{\D V_0} \lambda_0(j) \To i =_I j\big).$$
If $A$ is a set, the \textit{constant} $I$-\textit{family} $A$ is the pair 
$\Lambda^A := (\lambda_0^A, \lambda_1^A)$, where $\lambda_0 (i) := A$, for every 
$i \in I$, and $\lambda_1 (i, j) := \id_A$, for every $(i, j) \in D(I)$.

\end{defi}

\begin{defi}\label{def: map}
Let $\Lambda := (\lambda_0, \lambda_1)$ and $M := (\mu_0, \mu_1)$ be $I$-families of sets.
A \textit{family-map} from $\Lambda$ to $M$\index{family-map} is a dependent operation\footnote{In 
the categorical setting of Richman, a family map $\Psi \in \Map_I(\Lambda, M)$ is a natural transformation 
from the functor $\Lambda$ to the functor $M$. In our view, the fact that the most fundamental 
concepts of category theory, that of a functor and of a natural transformation, are formulated in a
natural way in $\BST$ through the notion of a dependent operation explains why category theory 
is so closely connected to $\BST$, or $\MLTT$. Both, $\BST$ and $\MLTT$,
highlight the role of dependent operations.}
$\Psi : \bigcurlywedge_{i \in I}\D F \big(\lambda_0(i), \mu_0(i)\big),$
such that for every $(i, j) \in D(I)$ the following diagram commutes
\begin{center}
\begin{tikzpicture}

\node (E) at (0,0) {$\mu_0(i)$};
\node[right=of E] (F) {$\mu_0(j)$.};
\node[above=of F] (A) {$\lambda_0(j)$};
\node [above=of E] (D) {$\lambda_0(i)$};

\draw[->] (E)--(F) node [midway,below]{$\mu_{ij}$};
\draw[->] (D)--(A) node [midway,above] {$\lambda_{ij}$};
\draw[->] (D)--(E) node [midway,left] {$\Psi_i$};
\draw[->] (A)--(F) node [midway,right] {$\Psi_j$};

\end{tikzpicture}
\end{center}
where $\Psi_i := \Psi(i)$ is the $i$-\textit{component} of $\Psi$, for every $i \in I$.
We denote by $\Map_I(\Lambda, M)$ the totality of family-maps from $\Lambda$ to $M$, which is equipped
with the equality 
$$\Psi =_{\Map_I(\Lambda, M)} \Xi :\TOT \forall_{i \in I}\big(\Psi_i =_{\D F(\lambda_0(i), \mu_0(i))} 
\Xi_i\big).$$
We also write $\Phi : \Lambda \To M$ to denote an element of $\Map_I(\Lambda, M)$.
If $\Psi : \Lambda \To M$ and $\Xi : M \To N$, the \textit{composition family-map} 
$\Xi \circ \Psi : \Lambda \To N$ is defined by 
$(\Xi \circ \Psi)(i) := \Xi_i \circ \Psi_i$
\begin{center}
\begin{tikzpicture}

\node (E) at (0,0) {$\lambda_0(i)$};
\node[right=of E] (F) {$\lambda_0(j)$};
\node[below=of F] (A) {$\mu_0(j)$};
\node[below=of E] (B) {$\mu_0(i)$};
\node[below=of B] (K) {$\nu_0(i)$};
\node[below=of A] (L) {$\nu_0(j)$,};

\draw[->] (E)--(F) node [midway,above] {$\lambda_{ij}$};
\draw[->] (F)--(A) node [midway,left] {$\Psi_j$};
\draw[->] (B)--(A) node [midway,below] {$\mu_{ij}$};
\draw[->] (E) to node [midway,right] {$\Psi_i$} (B);
\draw[->] (B) to node [midway,right] {$\Xi_i$} (K);
\draw[->] (A)--(L) node [midway,left] {$\Xi_j$};
\draw[->] (K)--(L) node [midway,below] {$\nu_{ij}$};
\draw[->,bend right] (E) to node [midway,left] {$(\Xi \circ \Psi)_i  \ $} (K);
\draw[->,bend left] (F) to node [midway,right] {$\  (\Xi \circ \Psi)_j$} (L);

\end{tikzpicture}
\end{center} 
for every $i \in I$. The dependent operation
$$\Id_{\Lambda} : \bigcurlywedge_{i \in I}\D F 
\big(\lambda_0(i), \lambda_0(i)\big), \ \ \ \Id_{\Lambda}(i) := \id_{\lambda_0(i)}, \ \ \ i \in I,$$
is the \textit{identity family-map} from $\Lambda$ to $\Lambda$. 
The totality of $I$-families is denoted by $\Fam(I)$, and
$$\Lambda =_{\Fam(I)} M :\TOT 
\exists_{\Phi \in \Map_I(\Lambda, M)}\exists_{\Xi \in \Map_I(M, \Lambda)}\big(\Phi \circ \Xi = \id_M \ \& \
\Xi \circ \Phi = \id_{\Lambda}\big).$$

\end{defi}

The equalities on $\Map_I(\Lambda, M)$ and $\Fam(I)$ are 
equivalence relations. It is natural to accept the totality $\Map(\Lambda, M)$ as a set. If $\Fam(I)$ was a set
though, the constant $I$-family with value $\Fam(I)$ would be
defined though a totality in which it belongs to. From a predicative point of view, this cannot be accepted. The 
membership condition of the totality
$\Fam(I)$ though, does not involve quantification over the universe $\D V_0$, therefore it is also natural not
to consider $\Fam(I)$ to be a class. Hence, $\Fam(I)$ is a totality
``between'' a (predicative) set and a class. For this reason, we say that $\Fam(I)$ is an 
\textit{impredicative set}.

\begin{defi}\label{def: newfamsofsets1}
If $K$ is a set, $\Sigma := (\sigma_0, \sigma_1)$ is 
a $K$-family of sets and $h : I \to K$, 
the $h$-\textit{subfamily} of $\Sigma$ is the pair\index{$\Sigma \circ h$} $\Sigma \circ h :=
(\sigma_0 \circ h, \sigma_1 \circ h)$, where
\[ (\sigma_0 \circ h)(i) := \sigma_0 (h(i)); \ \ \ \ i \in I, \] 
\[ (\sigma_1 \circ h)_{ij} := (\sigma_1 \circ h)(i, j) \colon \sigma_0 (h(i)) \to \sigma_0 (h(j)); \ \ \ \ 
(i, j) \in D(I), \]
\[ (\sigma_1 \circ h)_{ij} := \sigma_{h(i)h(j)}. \]
Clearly, $\Sigma \circ h \in \Fam(I)$, and we write $(\Sigma \circ h)_I\leq \Sigma_K$.
\end{defi}

\begin{defi}\label{def: exteriorunion}
Let $\Lambda := (\lambda_0, \lambda_1)$ be an $I$-family of sets. The\footnote{For the sake of readability we denote 
the totalities of exterior union and dependent functions over $\Lambda$ only with reference to the assignment routine
$\lambda_0$, while $\lambda_1$ is used in the equality formula of both totalities. In this way our notation is 
similar to the notation of the $\Sigma$-type and $\Pi$-type in $\MLTT$.}
\textit{exterior union}\index{exterior union}, or 
\textit{disjoint union}\index{disjoint union}, $\sum_{i \in I}\lambda_0 (i)$ of $\Lambda$ is defined by
$$w \in \sum_{i \in I}\lambda_0 (i) : 
\TOT \exists_{i \in I}\exists_{x \in \lambda_0 (i)}\big(w := (i, x)\big),$$
$$(i, x) =_{\sum_{i \in I}\lambda_0 (i)} (j, y) : \TOT i =_I j \ \& \ \lambda_{ij} (x) =_{\lambda_0 (j)} y.$$
The totality $\prod_{i \in I}\lambda_0(i)$ of \textit{dependent functions over} $\Lambda$ is defined by
$$\Phi \in \prod_{i \in I}\lambda_0(i) :\TOT \Phi \in \D A(I, \lambda_0) \ \& \ \forall_{(i,j) \in D(I)}\big(\Phi_j 
=_{\lambda_0(j)} \lambda_{ij}(\Phi_i)\big),$$
and it is equipped with the equality of $\D A(I, \lambda_0)$.
\end{defi}

The equalities on $\sum_{i \in I}\lambda_0 (i)$ and $\prod_{i \in I}\lambda_0(i)$ are 
equivalence relations, and both these totalities are sets. If $X, Y$ are sets, let $\Lambda(X,Y) 
:= \big(\lambda_0^{X,Y}, \lambda_1^{X,Y}\big)$ be the $\two$-family of $X$ and $Y$, where 
$\lambda_0^{X,Y} : \{0, 1\} \sto \D V_0$ is defined by 
$$\lambda_0^{X,Y}(0) := X \ \& \ 
\lambda_0^{X,Y}(1) := Y, \ \ \ \& \ \ \ \lambda_1^{X,Y}(0,0) := \id_X \ \& \ \lambda_1^{X,Y}(1,1) := \id_Y.$$ 
It is easy to show that the dependent functions over $\Lambda(X,Y)$ are equal in $\D V_0$ to $X \times Y$.
The \textit{first projection}
on $\sum_{i \in I}\lambda_0 (i)$ is the assignment routine $\pr_1^{\Lambda} : \sum_{i \in I}\lambda_0 (i) \sto I$, 
defined by 
$$\pr_1^{\Lambda} (i, x) : = \pr_1 (i, x) := i,$$
for every $(i, x) \in \sum_{i \in I}\lambda_0 (i)$. It is immediate to show that $\pr_1^{\Lambda}$ is a function
from $\sum_{i \in I}\lambda_0 (i)$ to $I$.
Moreover, it is easy to show that the pair $\Sigma^{\Lambda} := (\sigma_0^{\Lambda}, \sigma_1^{\Lambda})$,
where $\sigma_0^{\Lambda}
: \sum_{i \in I}\lambda_0 (i) \sto \D V_0$ is defined by $\sigma_0^{\Lambda}(i, x) := \lambda_0(i)$, and
$\sigma_1^{\Lambda}\big((i,x),(j,y)\big) := \lambda_{ij}$, is a family of sets over $\sum_{i \in I}\lambda_0 (i)$.
The \textit{second projection} on $\sum_{i \in I}\lambda_0 (i)$ is the dependent operation
$$\pr_2^{\Lambda} : \bigcurlywedge_{(i,x) \in \sum_{i \in I}\lambda_0 (i)}\lambda_0(i),$$
$$\pr_2^{\Lambda}(i,x) := \pr_2(i,x) := x,$$
and it is a dependent function over 
$\Sigma^{\Lambda}$. The following facts are easy to show.

\begin{prop}\label{prp: map1}
Let $\Lambda := (\lambda_0, \lambda_1)$, $M := (\mu_0, \mu_1) \in \Fam(I)$, and $\Psi : \Lambda \To M$.\\[1mm]
$(i)$ For every $i \in I$ the assignment routine
$e_i^{\Lambda} : \lambda_0(i) \sto \sum_{i \in I}\lambda_0(i)$, defined by 
$x \mapsto (i, x)$, is an embedding of $\lambda_0(i)$ into $\sum_{i \in I}\lambda_0(i)$.\\[1mm]
$(ii)$ The assignment routine
$\Sigma \Psi : \sum_{i \in I}\lambda_0(i) \sto \sum_{i \in I}\mu_0(i)$, defined by
$\Sigma \Psi (i, x) := (i, \Psi_i (x))$, 
is a function from $\sum_{i \in I}\lambda_0(i)$ to $\sum_{i \in I}\mu_0(i)$, such that for every
$i \in I$ the following diagram commutes
\begin{center}
\begin{tikzpicture}

\node (E) at (0,0) {$\sum_{i \in I}\lambda_0(i)$};
\node[right=of E] (F) {$\sum_{i \in I}\mu_0(i)$.};
\node[above=of F] (A) {$\mu_0(i)$};
\node [above=of E] (D) {$\lambda_0(i)$};

\draw[->] (E)--(F) node [midway,below]{$\Sigma \Psi$};
\draw[->] (D)--(A) node [midway,above] {$\Psi_{i}$};
\draw[->] (D)--(E) node [midway,left] {$e_i^{\Lambda}$};
\draw[->] (A)--(F) node [midway,right] {$e_i^M$};

\end{tikzpicture}
\end{center}
$(iii)$ If $\Psi_i$ is an embedding, for every $i \in I$, then $\Sigma \Psi$ is an
embedding.\\[1mm]
$(iv)$ For every $i \in I$ the assignment routine
$\pi_i^{\Lambda} : \prod_{i \in I}\lambda_0(i) \sto \lambda_0(i)$, defined by
$\Theta \mapsto \Theta_i,$ is a function from $\prod_{i \in I}\lambda_0(i)$ to $\lambda_0(i)$.\\[1mm]
$(v)$ The assignment routine
$\Pi \Psi : \prod_{i \in I}\lambda_0(i) \sto \prod_{i \in I}\mu_0(i)$, defined by
$[\Pi \Psi (\Theta)]_i := \Psi_i (\Theta_i)$, for every $i \in I$, is a function 
from $\prod_{i \in I}\lambda_0(i)$ to $\prod_{i \in I}\mu_0(i)$, such that for every $i \in I$ 
the following diagram commutes
\begin{center}
\begin{tikzpicture}

\node (E) at (0,0) {$\prod_{i \in I}\lambda_0(i)$};
\node[right=of E] (F) {$\prod_{i \in I}\mu_0(i)$.};
\node[above=of F] (A) {$\mu_0(i)$};
\node [above=of E] (D) {$\lambda_0(i)$};

\draw[->] (E)--(F) node [midway,below]{$\Pi \Psi$};
\draw[->] (D)--(A) node [midway,above] {$\Psi_{i}$};
\draw[->] (E)--(D) node [midway,left] {$\pi_i^{\Lambda}$};
\draw[->] (F)--(A) node [midway,right] {$\pi_i^M$};

\end{tikzpicture}
\end{center}
$(vi)$ If $\Psi_i$ is an embedding, for every $i \in I$, then $\Pi \Psi$ is an 
embedding.
\end{prop}

\section{Directed sets}
\label{sec: directed}

\begin{defi}\label{def: dirset}
Let $I$ be a set and $i \lt_I j$ a binary extensional relation on $I$ i.e., an extensional property on $I \times I$
\[\forall_{i, j, i{'}, j{'} \in I}\big(i =_I i{'} \ \& \ j =_I j{'} \ \& \ i \lt_I j \To i{'} \lt_I j{'}\big).\]
If $i \lt_I j$ is reflexive and transitive, then $(I, \lt_I)$ is called a 
a \textit{preorder}. We call a preorder $(I, \lt_I)$ a \textit{directed set}\index{directed set}, and 
\textit{inverse-directed}\index{inverse-directed set}, respectively, if
\[ \forall_{i, j \in I}\exists_{k \in I}\big(i \lt_I k \ \& \ j \lt_I k\big), \]
\[\forall_{i, j \in I}\exists_{k \in I}\big(i \mt_I k \ \& \ j \mt_I  k\big) .\]
The \textit{covariant diagonal} $D^{\lt}(I)$\index{$D^{\lt}(I)$} of $\lt_I$, the 
\textit{contravariant diagonal} $D^{\mt}(I)$\index{$D^{\mt}(I)$} of $\lt_I$, and
the $\lt_I$-\textit{upper set} $I_{ij}^{\lt}$\index{$I_{ij}^{\lt}$}
of $i, j \in I$\index{$\lt_I$-upper set} are defined, respectively, by 
\[ D^{\lt}(I) := \big\{(i, j) \in I \times I \mid i \lt_I j \big\}, \]
\[ D^{\mt}(I) := \big\{(j, i) \in I \times I \mid j \mt_I i \big\}, \]
\[ I_{ij}^{\lt} := \{k \in I \mid  i \lt_I k \ \& \ j \lt_I k\}. \]

\end{defi}

As $i \lt_I j$ is extensional, $D^{\lt}(I), D^{\mt}(I)$ are extensional 
subsets of $I \times I$, and $I_{ij}^{\lt}$ of $I$.

\begin{defi}\label{def: modulus}
Let $(I, \lt_I)$ be a poset i.e., a preorder such that $\big[i \lt_I j \ \& \ j \lt_I i \big]\To i =_I j$, for every 
$i, j \in I$. A \textit{modulus of directedness}\index{modulus of directedness} for $I$ is a 
function $\delta \colon I \times I \to I$, such
that for every $i, j, k \in I$ the following conditions are satisfied:\\[1mm]
$(\delta_1)$ $i \lt_I \delta(i, j)$ and $j \lt_I \delta(i, j)$.\\[1mm]
$(\delta_2)$ If $i \lt_I j$, then $\delta(i, j) =_I \delta(j, i) =_I j$.\\[1mm]
$(\delta_3)$ $\delta\big(\delta(i, j), k\big) =_I \delta\big(i, \delta(j, k)\big)$.
\end{defi}

In what follows we avoid for simplicity the use of subscripts on the relation symbols.
If $(I, \lt)$ is a directed set and $(J, e) \subseteq I$, where $e : J \hookrightarrow I$, 
and using for simplicity the same symbol $\lt$, if we define
$j \lt j{'} : \TOT e(j) \lt e(j{'}),$
for every $j, j{'} \in J$, then $(J, \lt)$ is only a preordered set. If $J$ is a cofinal subset of $I$, 
which classically it is defined by the condition
$\forall_{i \in I}\exists_{j \in J}\big(i \lt j \big)$, then $(J, \lt)$ becomes a
directed set. To avoid the use of dependent choice, we add in the definition of a cofinal subset $J$
of $I$ a modulus of cofinality for $J$.

\begin{defi}\label{def: cofinal}
Let $(I, \lt)$ be a directed set and $(J, e) \subseteq I$, and let 
$j \lt j{'} : \TOT e(j) \lt e(j{'})$, for every $j, j{'} \in J$. We say that $J$ is 
\textit{cofinal in} $I$, if there is a function\footnote{Categorically speaking, $e$ is a split monomorphism.}
$\cof_J : I \to J$, which we call a \textit{modulus of
cofinality} of $J$ in $I$, that satisfies the following conditions:\\[1mm]
$(\Cf_1)$ $\forall_{j \in J}\big(\cof_J(e(j)) =_J j\big)$.
\begin{center}
\begin{tikzpicture}

\node (E) at (0,0) {$J$};
\node[right=of E] (F) {$I$};
\node[right=of F] (A) {$J$.};

\draw[right hook->] (E)--(F) node [midway,below] {$e$};
\draw[->] (F)--(A) node [midway,below] {$\cof_J$};
\draw[->,bend left] (E) to node [midway,above] {$\id_J$} (A) ;

\end{tikzpicture}
\end{center}
$(\Cf_2)$  $\forall_{i, i{'} \in I}\big(i \lt i{'} \To \cof_J(i) \lt \cof_J(i{'})\big)$.\\[1mm] 
$(\Cf_3)$ $\forall_{i \in I}\big(i \lt e(\cof_J (i))\big)$.\\[1mm]
We denote the fact that $J$ is cofinal in $I$ by $(J, e, \cof_J) \subseteq^{\cof} I$, or, simpler, 
by $J \subseteq^{\cof} I$. 
\end{defi}

Taking into account the embedding $e$ of $J$ into $I$, condition $(\Cf_3)$ is the exact writing of
the classical defining condition $\forall_{i \in I}\exists_{j \in J}\big(i \lt j \big)$. 
To add condition $(\Cf_1)$ is natural, as $\lt$ is reflexive. If we consider condition $(\Cf_3)$ on $e(j)$,
for some $j \in J$, then by condition $(\Cf_1)$ we get $e(j) \lt e(\cof_J (e(j)))  = e(j)$.
Condition $(\Cf_2)$ is also harmless to add. In the classical setting if $i \lt i{'}$, and $j, j{'} \in J$
such that $i \lt j$ and $i{'} \lt j{'}$, then there is some $i{''} \in I$ such that $j{'} \lt i{''}$ and
$j \lt i{''}$. If $i{''} \lt j{''}$, for some $j{''} \in J$, 
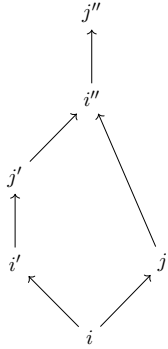
\begin{center}
\resizebox{2.4cm}{!}{%
\begin{tikzpicture}

\node (E) at (0,0) {$i$};
\node[above=of E] (F) {};
\node[right=of F] (A) {$j$};
\node[left=of F] (B) {$i{'}$};
\node[above=of B] (C) {$j{'}$};
\node[above=of C] (D) {};
\node[right=of D] (H) {$i{''}$};
\node[above=of H] (G) {$j{''}$};

\draw[->] (E)--(B) node [midway,below] {};
\draw[->] (E)--(A) node [midway,left] {};
\draw[->] (B)--(C) node [midway,below] {};
\draw[->] (C)--(H) node [midway,right] {};
\draw[->] (A)--(H) node [midway,right] {};
\draw[->] (H)--(G) node [midway,right] {};

\end{tikzpicture}
}
\end{center}
then $j \lt j{''}$. Since $i{'} \lt j{''}$ too,
condition $(\Cf_2)$ is justified. The added conditions $(\Cf_1)$ and $(\Cf_2)$ are used in the proofs of 
Theorem~\ref{thm: cofinal2} and Lemma~\ref{lem: cofinallemma}(ii), respectively. Moreover, they are used
in the proof of Theorem~\ref{thm: cofinal3}.
The extensionality of $\lt$
is also used in the proofs of Theorem~\ref{thm: cofinal2} and Theorem~\ref{thm: cofinal3}.

E.g., if $\Even$ and $\Odd$ denote the sets of even and odd natural numbers, respectively, 
let $e \colon \Even \eto \Nat$, defined by the identity map-rule, and $\cof_{\Even} \colon \Nat \to 2 \Nat$, 
defined by the rule
\[ \cof_{\Even}(n) := \left\{ \begin{array}{ll}
                 n    &\mbox{, $n \in \Even$}\\
                 n+1           &\mbox{, $n \in \Odd$.}
                 \end{array}
          \right.\]           
Then $(\Even, e, \cof_{\Even}) \subseteq \Nat$.

\begin{rem}\label{rem: cofinal1}
If $(I, \lt)$ is a directed set and $(J, e, \cof_J) \subseteq^{\cof} I$, then $(J, \lt)$ is directed.
\end{rem}

\begin{proof}
 Let $j, j{'} \in J$ and let $i \in I$ such that $e(j) \lt i$ and $e(j{'}) \lt i$. Since $i \lt e(\cof_J (i))$, 
 we get $e(j) \lt e(\cof_J (i))$ and $e(j{'}) \lt e(\cof_J (i))$ i.e., $j \lt \cof_J(i)$ and $j{'} \lt \cof_J(i)$.
\end{proof}

\section{Direct families of sets}
\label{sec: dirfamsets}

The next concept is a variation 
of the notion of a set-indexed family of sets\footnote{A directed set $(I, \lt_I)$ can also be seen as
a category with objects
the elements of $I$, and $\Hom_{\lt_I}(i, j) := \{x \in \{0\} \mid i \lt_I j\}.$
If the universe $\D V_0$ is seen as a category with objects its elements and 
$\Hom_{\lt_{\D V_0}}(X, Y) :=  \D F(X, Y),$ an $(I, \lt_I)$-family of sets is a
functor from the category $(I, \lt_I)$ to this new category $\D V_0$.}. A family of sets over a partial order is
also used in the definition of a Kripke model for intuitionistic predicate logic, and the corresponding
transport maps $\lambda_{ij}^{\lt}$ are called \textit{transition functions} (see~\cite{TD88}, p. 85).

\begin{defi}\label{def: dfamilyofsets}
Let $(I, \lt_I)$ be a directed set, and 
$D^{\lt}(I) := \big\{(i, j) \in I \times I \mid i \lt_I j \big\}$
the \textit{diagonal} of $\lt_I$.
A \textit{direct family of sets}\index{direct family of sets indexed by} $(I, \lt_I)$, 
or an $(I, \lt_I)$-\textit{family of sets}\index{$(I, \lt_I)$-family of sets}, 
is a pair $\Lambda^{\lt} := (\lambda_0, \lambda_1^{\lt})$, where
$\lambda_0 : I \sto \D V_0$, and $\lambda_1^{\lt}$\index{$\lambda_1^{\lt}$},
a \textit{modulus of transport maps} for $\lambda_0$,\index{modulus of transport maps} where 
\[ \lambda_1^{\lt} : \bigcurlywedge_{(i, j) \in D^{\lt}(I)}\D F\big(\lambda_0(i), \lambda_0(j)\big), \ \ \ 
\lambda_1^{\lt}(i, j) =: \lambda_{ij}^{\lt}, \ \ \ (i, j) \in D^{\lt}(I), \]
such that the \textit{transport maps}\index{transport map of a direct family of sets} $\lambda_{ij}^{\prec}$
of $\Lambda^{\lt}$ satisfy the following conditions:\\[1mm]
\normalfont (a) 
\itshape
For every $i \in I$, we have that $\lambda_{ii}^{\lt} := \id_{\lambda_0(i)}$.\\[1mm]
\normalfont (b) 
\itshape If $i \lt_I j$ and $j \lt_I k$, the following diagram commutes
\begin{center}
\begin{tikzpicture}

\node (E) at (0,0) {$\lambda_0(j)$};
\node[right=of E] (F) {$\lambda_0(k).$};
\node [above=of E] (D) {$\lambda_0(i)$};

\draw[->] (E)--(F) node [midway,below] {$\lambda_{jk}^{\lt}$};
\draw[->] (D)--(E) node [midway,left] {$\lambda_{ij}^{\lt}$};
\draw[->] (D)--(F) node [midway,right] {$\ \lambda_{ik}^{\lt}$};

\end{tikzpicture}
\end{center}
If $X \in \D V_0$, the \textit{constant} $(I, \lt_I)$-\textit{family}\index{constant family over $(I, \lt_I)$} 
$X$ is the pair $C^{\lt,X} := (\lambda_0^X, \lambda_1^{\lt,X})$, where $\lambda_0^X (i) := X$,  
and $\lambda_1^{\lt,X} (i, j) := \id_X$, for every $i \in I$ and $(i, j) \in  D^{\lt}(I)$.

\end{defi}

Since in general $\lt_I$ is not symmetric, the transport map $\lambda_{ij}^{\lt}$
does not necessarily have an inverse. Hence $\lambda_1^{\lt}$ is only a modulus of transport for 
$\lambda_0$, in the sense that determines the transport maps of $\Lambda^{\lt}$, and not necessarily 
a modulus of function-likeness for $\lambda_0$.

\begin{defi}\label{def: dirsumpi}
If $\Lambda^{\lt} := (\lambda_0, \lambda_1^{\lt})$ and $M^{\lt} 
:= (\mu_0, \mu_1^{\lt})$ are $(I, \lt_I)$-families of sets, a
\textit{direct family-map}\index{direct family-map} $\Phi$ from $\Lambda^{\lt}$ to $M^{\lt}$, denoted by 
$\Phi \colon \Lambda^{\lt} \To M^{\lt}$\index{$\Phi \colon \Lambda^{\lt} \To M^{\lt}$},
the set $\Map_{(I, \lt_I)}(\Lambda^{\lt}, M^{\lt})$\index{$\Map_{(I, \lt_I)}(\Lambda^{\lt}, M^{\lt})$},
and the totality $\Fam(I, \lt_I)$\index{$\Fam(I, \lt_I)$} of $(I, \lt_I)$-families are defined as
in Definition~\ref{def: map}.
The \textit{direct sum}\index{direct sum} 
$\sum_{i \in I}^{\lt}\lambda_0(i)$\index{$\sum_{i \in I}^{\lt}\lambda_0(i)$} over 
$\Lambda^{\lt}$ is the totality $\sum_{i \in I}\lambda_0(i)$ equipped with the equality
\[ (i, x) =_{\sum_{i \in I}^{\lt} \lambda_0(i)} (j, y) : \TOT \exists_{k \in I}\big(i \lt_I k \ 
\& \ j \lt_I k \ \& \ \lambda_{ik}^{\lt}(x) =_{\lambda_0(k)} \lambda_{jk}^{\lt}(y)\big). \] 
The totality $\prod_{i \in I}^{\lt}\lambda_0(i)$\index{$\prod_{i \in I}^{\lt}\lambda_0(i)$} 
of \textit{dependent functions over}\index{dependent functions over $\Lambda^{\lt}$} $\Lambda^{\lt}$ is defined by
\[ \Phi \in \prod_{i \in I}^{\lt}\lambda_0(i) :\TOT \Phi \in \D A(I, \lambda_0) 
\ \& \ \forall_{(i,j) \in D^{\lt}(I)}\big(\Phi_j =_{\lambda_0(j)} \lambda_{ij}^{\lt}(\Phi_i)\big), \]
and it is equipped with the equality of $\D A(I, \lambda_0)$.

\end{defi}

Clearly, the property $P(\Phi) :\TOT \forall_{(i,j) \in D^{\lt}(I)}\big(\Phi_j 
=_{\lambda_0(j)} \lambda_{ij}^{\lt}(\Phi_i)\big)$
is extensional on $\D A(I, \lambda_0)$,  
the equality on $\prod_{i \in I}^{\lt}\lambda_0(i)$ is an equivalence relation.   
$\prod_{i \in I}^{\lt}\lambda_0(i)$ is considered to be a set.

\begin{prop}\label{prp: direct1}
The relation $(i, x) =_{\sum_{i \in I}^{\lt} \lambda_0(i)} (j, y)$ 
is an equivalence relation.
\end{prop}

\begin{proof}
If $i \in I$, since $i \lt_I i$, there is $k \in I$
such that $i \lt_I k$, and by the reflexivity of the equality on $ \lambda_0(k)$ we get 
$\lambda_{ik}^{\lt}(x) =_{\lambda_0(k)} \lambda_{ik}^{\lt}(x)$. The symmetry of 
$=_{\sum_{i \in I}^{\lt} \lambda_0(i)}$ follows from the symmetry of the equalities
$=_{\lambda_0(k)}$. To prove transitivity, we suppose that 
\[ (i, x) =_{\sum_{i \in I}^{\lt} \lambda_0(i)} (j, y) : \TOT \exists_{k \in I}\big(i \lt_I k \ \& \ j 
\lt_I k \ \& \ \lambda_{ik}^{\lt}(x) =_{\lambda_0(k)} \lambda_{jk}^{\lt}(y)\big), \]
\[ (j, y) =_{\sum_{i \in I}^{\lt} \lambda_0(i)} (j{'}, z) : \TOT \exists_{k{'} \in I}\big(j \lt_I k{'} 
\ \& \ j {'} \lt_I k{'} \ \& \ \lambda_{jk{'}}^{\lt}(y) =_{\lambda_0(k{'})} 
\lambda_{j{'}k{'}}^{\lt}(z)\big), \]
and we show that
\[ (i, x) =_{\sum_{i \in I}^{\lt} \lambda_0(i)} (j{'}, z) : \TOT \exists_{k{''} \in I}\big(i \lt_I k{''} 
\ \& \ j{'} \lt_I k{''} \ \& \ \lambda_{ik{''}}^{\lt}(x) =_{\lambda_0(k{''})} 
\lambda_{j{'}k{''}}^{\lt}(z)\big). \]
By the definition of a directed set there is $k{''} \in I$ such that $k \lt_I k{''}$ and $k{'} \lt_I k{''}$
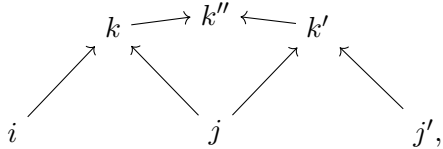
\begin{center}
\begin{tikzpicture}

\node (E) at (0,0) {$i$};
\node [right=of E] (F) {};
\node [above=of F] (A) {$k$};
\node [right=of F] (B) {$j$};
\node [right=of B] (C) {};
\node [right=of C] (D) {$j{'}$,};
\node [above=of C] (F) {$k{'}$};
\node [above=of B] (H) {$k{''}$};

\draw[->] (E)--(A) node [midway,below] {};
\draw[->] (B)--(A) node [midway,above] {};
\draw[->] (B)--(F) node [midway,above] {};
\draw[->] (D)--(F) node [midway,above] {};
\draw[->] (A)--(H) node [midway,above] {};
\draw[->] (F)--(H) node [midway,above] {};

\end{tikzpicture}
\end{center}
hence by transitivity $i \lt_I k{''}$ and $j{'} \lt_I k{''}$. Moreover,
\begin{align*}
\lambda_{ik{''}}^{\lt}(x) & \stackrel{i \lt_I k \lt_I k{''}} = \lambda_{kk{''}}^{\lt}
\big(\lambda_{ik}^{\lt}(x)\big) \\
& \ \ \  = \ \ \  \lambda_{kk{''}}^{\lt} \big(\lambda_{jk}^{\lt}(y)\big) \\
& \stackrel{j \lt_I k \lt_I k{''}} = \lambda_{jk{''}}^{\lt} (y) \\
& \stackrel{j \lt_I k{'} \lt_I k{''}} = \lambda_{k{'}k{''}}^{\lt} \big(\lambda_{jk{'}}^{\lt}(y)\big) \\
& \ \ \ = \ \ \ \lambda_{k{'}k{''}}^{\lt} \big(\lambda_{j{'}k{'}}^{\lt}(z)\big) \\
& \stackrel{j{'} \lt_I k{'} \lt_I k{''}} = \lambda_{j{'}k{''}}^{\lt}(z).
  \qedhere
\end{align*}
\end{proof}

Notice that the projection operation from $\sum_{i \in I}^{\lt}\lambda_0(i)$ to $I$ 
is not necessarily a function.

\begin{prop}\label{prp: preorderfamilymap1}
If $(I, \lt_I)$ is a directed set, $\Lambda^{\lt} := (\lambda_0, \lambda_1^{\lt})$, $M^{\lt} 
:= (\mu_0, \mu_1^{\lt})$ are $(I, \lt_I)$-families of sets, and $\Psi^{\lt} : \Lambda^{\lt} \To M^{\lt}$, 
the following hold.\\[1mm]
\normalfont (i) 
\itshape
For every $i \in I$ the operation
$e_i^{\Lambda^{\lt}} : \lambda_0(i) \sto \sum_{i \in I}^{\lt}\lambda_0(i)$, defined by 
$x \mapsto (i, x)$, for every $x \in \lambda_0(i)$, 
is a function from $\lambda_0(i)$ to $\sum_{i \in I}^{\lt}\lambda_0(i)$.\\[1mm]
\normalfont (ii) 
\itshape The operation
$\Sigma^{\lt} \Psi : \sum_{i \in I}^{\lt}\lambda_0(i) \sto \sum_{i \in I}^{\lt}\mu_0(i)$,
defined by
$\big(\Sigma^{\lt} \Psi\big) (i, x) := (i, \Psi_i (x))$, for every $(i, x) \in \sum_{i \in I}^{\lt}\lambda_0(i)$,
is a function from $\sum_{i \in I}^{\lt}\lambda_0(i)$ to $\sum_{i \in I}^{\lt}\mu_0(i)$ such that, for every
$i \in I$, the following left diagram commutes
\begin{center}
\begin{tikzpicture}

\node (E) at (0,0) {$\sum_{i \in I}^{\lt}\lambda_0(i)$};
\node[right=of E] (F) {$\sum_{i \in I}^{\lt}\mu_0(i)$};
\node[above=of F] (A) {$\mu_0(i)$};
\node [above=of E] (D) {$\lambda_0(i)$};

\node [right=of F] (G) {$\prod_{i \in I}^{\lt}\lambda_0(i)$};
\node[right=of G] (K) {$\prod_{i \in I}^{\lt}\mu_0(i)$.};
\node[above=of K] (L) {$\mu_0(i)$};
\node [above=of G] (M) {$\lambda_0(i)$};

\draw[->] (E)--(F) node [midway,below]{$\Sigma^{\lt} \Psi$};
\draw[->] (D)--(A) node [midway,above] {$\Psi_{i}$};
\draw[->] (D)--(E) node [midway,left] {$e_i^{\Lambda^{\lt}}$};
\draw[->] (A)--(F) node [midway,right] {$e_i^{M^{\lt}}$};
\draw[->] (G)--(K) node [midway,below]{$\Pi^{\lt} \Psi$};
\draw[->] (M)--(L) node [midway,above] {$\Psi_{i}$};
\draw[->] (G)--(M) node [midway,left] {$\pi_i^{\Lambda^{\lt}}$};
\draw[->] (K)--(L) node [midway,right] {$\pi_i^{M^{\lt}}$};

\end{tikzpicture}
\end{center}
\normalfont (iii) 
\itshape If $\Psi_i$ is an embedding, for every $i \in I$, then $\Sigma^{\lt} \Psi$ is an embedding.\\[1mm]
\normalfont (iv) 
\itshape For every $i \in I$ the operation
$\pi_i^{\Lambda^{\lt}} : \prod_{i \in I}^{\lt}\lambda_0(i) \sto \lambda_0(i)$, defined by
$\Theta \mapsto \Theta_i$, for every $\Theta \in \prod_{i \in I}^{\lt}\lambda_0(i)$, 
is a function from $\prod_{i \in I}^{\lt}\lambda_0(i)$ to $\lambda_0(i)$.\\[1mm]
\normalfont (v) 
\itshape The operation
$\Pi^{\lt} \Psi : \prod_{i \in I}^{\lt}\lambda_0(i) \sto \prod_{i \in I}^{\lt}\mu_0(i)$, defined by
$[\Pi^{\lt} \Psi (\Theta)]_i := \Psi_i (\Theta_i)$, for every $i \in I$ and $\Theta \in \prod_{i \in I}^{\lt}\lambda_0(i)$,
is a function from $\prod_{i \in I}^{\lt}\lambda_0(i)$ to $\prod_{i \in I}^{\lt}\mu_0(i)$, such that, for every $i \in I$, 
the above right diagram commutes.\\[1mm]
%
%
%
%
%
\normalfont (vi) 
\itshape
If $\Psi_i$ is an embedding, for every $i \in I$, then $\Pi^{\lt} \Psi$ is an 
embedding.
\end{prop}

\begin{proof}
(i) Let $x, y \in \lambda_0(i)$ with $x =_{\lambda_0 (i)} y$. As $\lt$ is reflexive, if we
take $k := i$, we get $\lambda_{ii}^{\lt}(x) := \id_{\lambda_0(i)}(x) := x =_{\lambda_0(i)} y =:
\id_{\lambda_0(i)}(y) =: \lambda_{ii}^{\lt}(y)$, hence 
$(i, x) =_{\sum_{i \in I}^{\lt}\lambda_0(i)} (i, y)$.\\
(ii) If $(i, x) =_{\sum_{i \in I}^{\lt}\lambda_0(i)} (j, y)$, there is $k \in I$ such that $i \lt_I k$,
$j \lt_I k$ and $\lambda_{ik}^{\lt}(x) =_{\lambda_0(k)} \lambda_{jk}^{\lt}(y)$. We show the following equality: 
\begin{align*}
\big(\Sigma^{\lt} \Psi\big)(i, x) =_{\sum_{i \in I}^{\lt}\mu_0(i)} 
\big(\Sigma^{\lt} \Psi\big)(j, y) & : \TOT (i, \Psi_i (x)) 
=_{\mathsmaller{\sum_{i \in I}^{\lt}\mu_0(i)}} (j, \Psi_j (y))\\
& : \TOT \exists_{k{'} \in I}\big(i, j \lt_I k{'} \ \& \ 
\mu_{ik{'}}^{\lt}(\Psi_i (x)) =_{\mu_0(k{'})} \mu_{jk{'}}^{\lt}(\Psi_j(y))\big).
\end{align*} 
If we take $k{'} := k$, by the commutativity of the following diagrams, and  since $\Psi_k$ is a function,
\begin{center}
\begin{tikzpicture}

\node (E) at (0,0) {$\mu_0(i)$};
\node[right=of E] (F) {$\mu_0(k)$};
\node[above=of F] (A) {$\lambda_0(k)$};
\node [above=of E] (D) {$\lambda_0(i)$};
\node [right=of F] (G) {$\mu_0 (j)$};
\node [above=of G] (H) {$\lambda_0(j)$};
\node [right=of G] (K) {$\mu_0 (k)$};
\node [above=of K] (L) {$\lambda_0(k)$};

\draw[->] (E)--(F) node [midway,below]{$\mu_{ik}^{\lt}$};
\draw[->] (D)--(A) node [midway,above] {$\lambda_{ik}^{\lt}$};
\draw[->] (D)--(E) node [midway,left] {$\Psi_i$};
\draw[->] (A)--(F) node [midway,right] {$\Psi_k$};
\draw[->] (G)--(K) node [midway,below] {$\mu_{jk}^{\lt}$};
\draw[->] (H)--(L) node [midway,above] {$\lambda_{jk}^{\lt}$};
\draw[->] (H)--(G) node [midway,left] {$\Psi_j$};
\draw[->] (L)--(K) node [midway,right] {$\Psi_k$};

\end{tikzpicture}
\end{center}
\[ \mu_{ik}^{\lt}\big(\Psi_i(x)\big) =_{\mu_0 (k)} \ \Psi_k \big(\lambda_{ik}^{\lt}(x)\big)
=_{\mu_0 (k)} \ \Psi_k \big(\lambda_{jk}^{\lt}(y)\big)
=_{\mu_0 (k)} \ \mu_{jk}^{\lt}\big(\Psi_j (y)\big). \]
(iii) If we suppose $\big(\Sigma^{\lt} \Psi\big)(i, x) =_{\mathsmaller{\sum_{i \in I}^{\lt}\mu_0(i)}} 
\big(\Sigma^{\lt} \Psi\big)(j, y)$ i.e.,
$\mu_{ik}^{\lt}(\Psi_i (x)) =_{\mu_0(k)} \mu_{jk}^{\lt}(\Psi_j(y))\big)$,
for some $k \in I$ with $i, j \lt_I k$, by the proof of case (ii) we get  
$\Psi_k \big(\lambda_{ik}^{\lt}(x)\big) =_{\mu_0 (k)}  
\Psi_k \big(\lambda_{jk}^{\lt}(y)\big)$, and since
$\Psi_k$ is an embedding, we get $\lambda_{ik}^{\lt}(x) =_{\lambda_0 (k)} 
\lambda_{jk}^{\lt}(y)$ i.e., $(i, x) =_{\sum_{i \in I}^{\lt}\lambda_0(i)} (j, y)$.\\
(iv)-(vi) Their proof is omitted, since a proof of their contravariant version\footnote{In a contravariant 
family the transport maps are of type $\lambda_0(j) \to \lambda_0(i)$, if $i \lt j$.} 
is given in the proof of Theorem~\ref{thm: inverselimitmap}. 
\end{proof}

Since the transport functions $\lambda_{ik}^{\lt}$ are not in general embeddings, 
we cannot show in general that $e_i^{\Lambda^{\lt}}$ is an embedding, as it is the case for the map 
$e_i^{\Lambda}$ in Proposition~\ref{prp: map1}(i). 

\section{On Bishop spaces}
\label{sec: bspaces}

From now on we work within the extension $\BST^*$ of $\BST$.

\begin{defi}\label{def: cont1}
If $X$ is a set and $\Real$ is the set of real numbers, we denote by $\mathbb{F}(X)$
the set of functions from $X$ to $\Real$, and by $\Const(X)$\index{$\Const(X)$} the 
subset of $\mathbb{F}(X)$ of all constant real functions on $X$. 
If $a \in \Real$, we denote by $\overline{a}^X$\index{$\overline{a}^X$}
the constant function on $X$ with value $a$. We denote by $\Nat^+$ the set of non-zero natural numbers.  
A function $\phi: \Real \rightarrow \Real$ is called \textit{Bishop continuous}, or simply continuous,
if for every $n \in \Nat^+$ there is a function\index{$\omega_{\phi,n}$} $\omega_{\phi, n}:
\mathbb{R}^{+} \rightarrow \mathbb{R}^{+}$,
$\epsilon \mapsto \omega_{\phi, n}(\epsilon)$, which is called a \textit{modulus 
of continuity}\index{modulus of (uniform) continuity} of $\phi$ on $[-n, n]$, such that the following 
condition is satisfied
\[\forall_{x, y \in [-n, n]}(|x - y| < \omega_{\phi, n}(\epsilon) \Rightarrow 
|\phi(x) - \phi(y)| \leq \epsilon),\]
for every $\epsilon > 0$ and every $n \in \Nat^+$. We denote by $\BR$\index{$\BR$} the set of continuous functions 
from $\Real$ to $\Real$, which is equipped with the equality inherited from $\D F(\Real)$.

\end{defi}

We could have defined the modulus of continuity $\omega_{\phi, n}$ as a function from $\Nat^+$ 
to $\Nat^+$. A continuous function $\phi \colon \Real \to \Real$ is uniformly continuous on every bounded 
subset of $\Real$. The latter is an impredicative formulation of uniform continuity, as it 
requires quantification over the class $\C P(\Real)$. The formulation of uniform continuity 
in the Definition~\ref{def: cont1} though, is predicative, since it requires quantification over the sets
$\Nat^+, \D F(\Real^+, \Real^+)$ and $[-n, n]$.

\begin{defi}\label{def: notation1}
If $X$ is a set, $f, g \in \mathbb{F}(X)$, $\epsilon > 0$, and $\Phi \subseteq \mathbb{F}(X)$, 
let\index{$U(X; \Phi, g, \epsilon)$}  
\[ U(X; f, g, \epsilon) :\TOT \forall_{x \in X}\big(|g(x) - f(x)| \leq \epsilon\big),\]
\[ U(X; \Phi, f) :\TOT \forall_{\epsilon > 0}\exists_{g \in \Phi}\big(U(f, g, \epsilon)\big).\]
If the set $X$ is clear from the context, we write simply $U(f, g, \epsilon)$\index{$U(f,g, \epsilon)$} 
and\index{$U(\Phi, f)$} $U(\Phi, f)$, respectively.\index{$U(X; \Phi, f)$}
We denote by $\Phi^*$ the bounded elements of $\Phi$, and its uniform 
closure\index{uniform closure} $\overline{\Phi}$ is defined by\index{$\overline{\Phi}$} 
\[ \overline{\Phi} := \{f \in \D F(X) \mid U(\Phi, f)\}.\]
\end{defi}

A Bishop topology on $X$ is a certain subset of $\D F(X)$. Since the Bishop topologies considered here
are all extensional subsets of $\D F(X)$, we do not mention the embedding $i_F^{\D F(X)} \colon F \eto \D F(X)$, 
which is given in all cases by the identity map-rule.

\begin{defi}\label{def: bishopspace}
A \textit{Bishop space}\index{Bishop space} is a pair $\C F := (X, F)$, where $F$ is an extensional subset of $\D F(X)$,
which is called a \textit{Bishop topology}, or simply a \textit{topology}\index{Bishop topology}
of functions on $X$, that satisfies the following conditions:\\[1mm]
$(\BS_1)$ If $a \in \Real$, then $\overline{a}^X \in F$.\\[1mm]
$(\BS_2)$ If $f, g \in F$, then $f + g \in F$.\\[1mm]
$(\BS_3)$ If $f \in F$ and $\phi \in \Bic(\Real)$, then $\phi \circ f \in F$
\begin{center}
\begin{tikzpicture}

\node (E) at (0,0) {$X$};
\node[right=of E] (F) {$\Real$};
\node[below=of F] (A) {$\Real$.};

\draw[->] (E)--(F) node [midway,above] {$f$};
\draw[->] (E)--(A) node [midway,left] {$F \ni \phi \circ f \ $};
\draw[->] (F)--(A) node [midway,right] {$\phi \in \BR$};

\end{tikzpicture}
\end{center}
$(\BS_4)$ $\overline{F} = F$.
\end{defi}

If $F$ is inhabited, then $(\BS_1)$ is provable by $(\BS_3)$.
The set of constant functions
$\Const(X)$ is the \textit{trivial}
topology on $X$, while $\D F(X)$ is the \textit{discrete} topology on $X$. Clearly, if $F$ is a topology on $X$,
then $\Const(X) \subseteq F \subseteq \D F(X)$, and the set of its bounded elements
$F^{*}$ is also a topology on $X$. We denote by $\C F^* := (X, F^*)$ the Bishop space of 
bounded elements of a Bishop topology $F$.
It is easy to see that the pair $\C R := (\Real, \BR)$ is a
Bishop space, which we call the \textit{Bishop space of reals}.
A Bishop topology $F$ is a ring and a lattice; since $|\id_{\Real}| \in \Bic(\Real)$, where $\id_{\Real}$ is the 
identity function on $\Real$, by BS$_{3}$, if $f \in F$, then $|f| \in F$. 
By BS$_{2}$ and BS$_{3}$, and using the following equalities  
\[ f{\cdot}g = \frac{(f + g)^{2} - f^{2} - g^{2}}{2} \in F,\]
\[ f \vee g = \max\{f, g\} = \frac{f + g + |f - g|}{2} \in F, \]  
\[ f \wedge g = \min\{f, g\} = \frac{f + g - |f - g|}{2} \in F,\]
we get similarly that if $f, g \in F$, then $f{\cdot}g, f \vee g, f \wedge g \in F$.
Turning the definitional clauses of a Bishop topology into inductive rules, Bishop defined in~\cite{Bi67}, p.~72,
the least topology including a given subbase $F_{0}$. This inductive definition, which 
is also found in~\cite{BB85}, p.~78, is crucial to the definition of new Bishop topologies from given ones.

\begin{defi}\label{def: Least}
The \textit{Bishop closure} of $F_{0}$, or the \textit{least topology}\index{least Bishop topology} 
$\bigvee{F_{0}}$\index{$\bigvee F_0$} 
generated by some $F_{0} \subseteq \mathbb{F}(X)$, is defined by the following inductive rules:
\[ \frac{f_{0} \in F_{0}}{f_{0} \in \bigvee F_{0}}, \ \ \ \frac{a \in \Real}{\overline{a}^{X} \in
\bigvee F_{0}},
\ \ \ \frac{f, g \in \bigvee F_{0}}{f + g \in \bigvee F_{0}}, \ \ \  \frac{f \in \bigvee F_{0}, \ 
g =_{\D F(X)} f}{g \in \bigvee F_{0}},\]
\[ \frac{f \in \bigvee F_{0}, \ \phi \in \Bic(\Real)}{\phi \circ f \in \bigvee F_{0}},
\ \ \ \ \ \frac{\big(g \in \bigvee F_{0} \ \& \ U(f, g, \epsilon)\big)_{\epsilon > 0}}{f \in \bigvee F_{0}}.\]
We call $\bigvee F_{0}$ the \textit{Bishop closure} of $F_{0}$, and $F_{0}$ a \textit{subbase}\index{subbase} of 
$\bigvee F_{0}$.
\end{defi}

The last, most complex rule
above can be replaced by the rule  
\[ \frac{g_{1} \in \bigvee F_{0} \ \& \ U\big(f, g_{1}, \frac{1}{2}\big), 
\ \  g_{2} \in \bigvee F_{0} \ \& \ 
      U\big(f, g_{2}, \frac{1}{2^{2}}\big), \   \ldots}{f \in \bigvee F_{0}},\]
a rule with countably many premisses. 
The corresponding induction principle $\Ind_{\bigvee F_{0}}$\index{$\Ind_{\bigvee F_0}$} is 
\[\bigg[\forall_{f_{0} \in F_{0}}\big(P(f_{0})\big) \ \& \
\forall_{a \in \mathbb{R}}\big(P(\overline{a}^{X})\big) \ 
\&  \ \forall_{f, g \in \bigvee F_{0}}\big(P(f) \ \& \ P(g) \Rightarrow P(f + g) \]
\[ \& \ \forall_{f \in \bigvee F_0}\forall_{g \in \D F(X)}\big(g =_{\D F(X)} f \ \& \ P(f) \To P(g)\big)\]
\[ \& \ \forall_{f \in \bigvee F_{0}}\forall_{\phi \in \Bic(\mathbb{R})}\big(P(f)
\Rightarrow P(\phi \circ f)\big) \]
\[\ \ \ \ \ \ \ \ \ \ \ \ \ \ \ \ \ \ \ \ \ \& \  \forall_{f \in \bigvee F_{0}}\big( 
 \forall_{\epsilon > 0}\exists_{g \in \bigvee F_{0}}(P(g) \ \& \ U(f, g, \epsilon))
 \Rightarrow P(f)\big)\bigg]\]
 \[\Rightarrow \forall_{f \in \bigvee F_{0}}\big(P(f)\big),\]
where $P$ is any bounded formula.
Next we define the notion of a Bishop morphism\index{Bishop morphism} between Bishop spaces. 
The Bishop morphisms are the arrows in the category of Bishop\index{$\Bis$} spaces $\Bis$.

\begin{defi}\label{def: bmorphism}
If $\C F := (X, F)$ and $\C G = (Y, G)$ are Bishop spaces, a function $h: X \rightarrow Y$ is called
a \textit{Bishop morphism}, if $\forall_{g \in G}(g \circ h \in F)$ 
\begin{center}
\begin{tikzpicture}

\node (E) at (0,0) {$X$};
\node[right=of E] (F) {$Y$};
\node[below=of F] (A) {$\Real$.};

\draw[->] (E)--(F) node [midway,above] {$h$};
\draw[->] (E)--(A) node [midway,left] {$F \ni g \circ h \ $};
\draw[->] (F)--(A) node [midway,right] {$g \in G$};

\end{tikzpicture}
\end{center}
We denote by $\Mor(\C F, \C G)$\index{$\Mor(\C F, \C G)$} the set of Bishop morphisms 
from $\C F$ to $\C G$. As $F$ is an extensional subset of $\D F(X)$, $\Mor(\C F, \C G)$ is an
extensional subset of $\D F(X,Y)$.  If $h \in \Mor(\C F, \C G)$, the \textit{induced mapping} 
$h^* \colon G \to F$ from $h$\index{induced mapping from a Bishop morphism}\index{$h^*$}
is defined by the rule
\[ h^*(g) := g \circ h; \ \ \ \ g \in G.\]
\end{defi}

If $\mathcal{F} := (X, F)$ is a Bishop space, then $F = \Mor(\mathcal{F}, \mathcal{R})$,
and one can show inductively that if $\C G := (Y, \bigvee G_0)$, then 
$h: X \rightarrow Y \in \Mor(\mathcal{F}, \C G)$ if and only if
$\forall_{g_{0} \in G_{0}}(g_{0} \circ h \in F)$
\begin{center}
\begin{tikzpicture}

\node (E) at (0,0) {$X$};
\node[right=of E] (F) {$Y$};
\node[below=of F] (A) {$\Real$.};

\draw[->] (E)--(F) node [midway,above] {$h$};
\draw[->] (E)--(A) node [midway,left] {$F \ni g_0 \circ h \ \ $};
\draw[->] (F)--(A) node [midway,right] {$g_0 \in G_0$};

\end{tikzpicture}
\end{center}
We call this fundamental fact the\index{$\bigvee$-lifting of morphisms} $\bigvee$-\textit{lifting of morphisms}. 
A Bishop morphism is a \textit{Bishop isomorphism}\index{Bishop isomorphism}, if it is an isomorphism in
the category $\Bis$. We write $\C F \simeq \C G$\index{$\C F \simeq \C G$} to denote that $\C F$ and $\C G$ 
are Bishop isomorphic. If $h \in \Mor(\C F, \C G)$ is a bijection, then 
$h$ is a Bishop isomorphism if and only if it is \textit{open}\index{open morphism} i.e., 
$\forall_{f \in F}\exists_{g \in G}\big(f = g \circ h\big)$. 
%

\begin{defi}\label{def: new} Let $\C F := (X, F), \C G := (Y, G)$ be Bishop spaces, $(A, i_A) \subseteq X$ 
inhabited, and $\phi: X \rightarrow Y$ a surjection. The \textit{product} Bishop space\index{product Bishop space}
$\C F \times \C G := (X \times Y, F \times G)$ of $\C F$\index{$F \times G$}
and $\C G$, the \textit{relative}\index{relative Bishop space} Bishop space $\Fii_{|A} := (A, F_{|A})$ on $A$, and 
the \textit{pointwise exponential Bishop space}\index{$F_{|A}$} \index{$F \to G$}
$\mathcal{F} \rightarrow \mathcal{G} = (\Mor(\mathcal{F}, \mathcal{G}), 
F \rightarrow G)$ are defined, respectively, by 
\[ F \times G := \bigvee \left[\{f \circ \pr_{X}, \mid f \in F\} \cup \{g \circ \pr_{Y} \mid g \in G\}\right] 
=: \bigvee_{f \in F}^{g \in G}f \circ \pr_{X}, g \circ \pr_{Y},\]
\[ F_{|A} = \bigvee \{f_{|A} \mid f \in F\} =: \bigvee_{f \in F}f_{|A}\]
\begin{center}
\begin{tikzpicture}

\node (E) at (0,0) {$A$};
\node[right=of E] (F) {$X$};
\node [right=of F] (B) {$\Real$,};

\draw[right hook->] (E)--(F) node [midway,above] {$i_A$};
\draw[->] (F)--(B) node [midway,above] {$f$};
\draw[->,bend right] (E) to node [midway,below] {$f_{|A}$} (B);

\end{tikzpicture}
\end{center}
\[ F \rightarrow G := \bigvee \big\{\phi_{x, g} \mid x \in X, g \in G \big\} :=
\bigvee_{x \in X}^{g \in G}\phi_{x, g},\]
\[ \phi_{x, g}: \Mor(\mathcal{F}, \mathcal{G}) \rightarrow \mathbb{R}, \ \ \ \ \phi_{x, g}(h) = g(h(x));
\ \ \ \ x \in X, \ g \in G.\]
\end{defi}

One can show inductively the following $\bigvee$-liftings
\begin{align*}
\bigvee F_{0} \times \bigvee G_{0} & := \bigvee \left[\{f_{0} \circ \pr_{X}, \mid f_{0} \in F_{0}\} \cup
\{g_{0} \circ \pr_{Y} \mid g_{0} \in G_{0}\}\right]\\
& =: \bigvee_{f_{0} \in F_{0}}^{g_{0} \in G_{0}}f_{0} \circ \pr_{X}, g_{0} \circ \pr_{Y},
\end{align*}
\[ \big(\bigvee F_{0}\big)_{|A} = \bigvee \{{f_{0}}_{|A} \mid f_{0} \in F_{0}\} =: \bigvee_{f_{0}
\in F_{0}}{f_{0}}_{|A},\]
\[ F \rightarrow \bigvee G_0 = \bigvee \big\{\phi_{x, g_0} \mid x \in X, g_0 \in G_0 \big\} :=
\bigvee_{x \in X}^{g_0 \in G_0}\phi_{x, g_0}.\]
$F_{|A}$ is the least topology on $A$ that makes $i_A$ a Bishop morphism, and 
the product topology $F \times G$ is the least topology on $X \times Y$ 
that makes the projections $\pr_X$ and $\pr_Y$ Bishop morphisms. The term pointwise exponential Bishop
topology is due to the fact that $F \to G$ behaves like the the classical topology of the pointwise convergence 
on $C(X, Y)$, the set of continuous functions from the topological space $X$ to the topological space $Y$.

\section{Direct spectra of Bishop spaces}
\label{sec: direxternalspectra}

Roughly speaking, if $S$ is a structure on some set, an $S$-spectrum\index{$S$-spectrum} is 
an $I$-family of sets $\Lambda$ such that each set $\lambda_0 (i)$ is equipped with a structure $S_i$, which is
compatible with the transport maps $\lambda_{ij}$ of $\Lambda$. 
Accordingly, a spectrum of Bishop spaces is an $I$-family of sets $\Lambda$ such that each set $\lambda_0 (i)$
is equipped with a Bishop topology $F_i$, which is compatible with the transport maps of $\Lambda$. As expected, 
in the case of a spectrum of Bishop spaces this compatibility condition is that the transport
maps $\lambda_{ij}$ are Bishop morphisms i.e. $\lambda_{ij} \in \Mor(\C F_i, \C F_j)$. It is natural to associate to 
$\Lambda$ an $I$-family of sets $\Phi := (\phi_0^{\Lambda}, \phi_1^{\Lambda})$ such that $\C F_i := 
\big(\lambda_0 (i), \phi_0^{\Lambda} (i)\big)$ is the Bishop space corresponding to $i \in I$.
If $i =_I j$, and if we put no restriction to the definition of $\phi_{ij}^{\Lambda} : F_i \to F_j$,
we need to add extra data in the definition of a map between spectra of
Bishop spaces. Since the map $\lambda_{ji}^* : F_i \to F_j$, where $\lambda_{ji}^*$ is the element 
of $\D F(F_i, F_j)$ induced by the Bishop morphism $\Lambda_{ji} \in \Mor(\C F_j, \C F_i)$, is 
generated by the data of $\Lambda$, it is natural to define $\phi_{ij} := \lambda_{ji}^*$. 
In this way proofs of properties of maps between spectra of Bishop spaces become easier. 
Every subset of $\D F(X)$ considered 
in this section is an extensional subset of it.

\begin{defi}\label{def: spectrum}
Let $\Lambda := (\lambda_0, \lambda_1)$. A \textit{family of Bishop topologies associated to}
$\Lambda$\index{family of Bishop topologies associated
to a family of sets} is a pair $\Phi^{\Lambda} := \big(\phi_0^{\Lambda}, \phi_1^{\Lambda}\big)$, where
$\phi_0^{\Lambda} \colon I \sto \D V_0$ and $\phi^{\Lambda}_1 \colon \bigcurlywedge_{(i,j) \in D(I)}\D F\big(\phi_0^{\Lambda}(i), \phi_0^{\Lambda}(j)\big)$, such that the following conditions hold:\\[1mm]
\normalfont (i)
\itshape  $\phi_0^{\Lambda}(i) := F_i \subseteq \D F(\lambda_0(i))$, and $\C F_i := (\lambda_0(i), F_i)$
is a Bishop space, for every $i \in I$.\\[1mm]
\normalfont (ii) 
\itshape $\lambda_{ij} \in \Mor(\C F_i, \C F_j)$, for every $(i, j) \in D(I)$.\\[1mm]
\normalfont (iii) $\phi_1^{\Lambda}(i, j) := \lambda_{ji}^*$, for every $(i, j) \in D(I)$,
where, if $f \in F_i$, the induced map $\lambda_{ji}^* \colon F_i \to F_j$ from $\lambda_{ji}$ is defined by 
 $\lambda_{ji}^* (f) := f \circ \lambda_{ji}$, for every $f \in F_i$.\\[1mm]
The structure $S(\Lambda) := (\lambda_0, \lambda_1, \phi_0^{\Lambda}, \phi_1^{\Lambda})$
is called a \textit{spectrum of Bishop spaces} over $I$\index{spectrum over a set}, or an 
$I$-\textit{spectrum}\index{$I$-spectrum} with Bishop spaces $(\C F_i)_{i \in I}$ and Bishop isomorphisms
$(\lambda_{ij})_{(i,j) \in D(I)}$. 

If $M := (\mu_0, \mu_1) \in \Fam(I)$ and 
$S(M) := (\mu_0, \mu_1, \phi_0^M, \phi_1^M)$ is an $I$-spectrum with Bishop spaces $(\C G_i)_{i \in I}$ and
Bishop isomorphisms $(\mu_{ij})_{(i,j) \in D(I)}$, 
a \textit{spectrum-map} $\Psi$\index{spectrum-map} from $S(\Lambda)$ to $S(M)$\index{spectrum-map}, in
symbols $\Psi \colon S(\Lambda) \To S(M)$\index{$\Phi \colon S(\Lambda) \To S(M)$}, is a family-map
$\Psi \colon \Lambda \To M$. A spectrum-map $\Phi \colon S(\Lambda) \To S(M)$ is called \textit{continuous}\index{continuous spectrum-map}, if 
$\Psi_i \in\Mor(\C F_i, \C G_i)$, for every $i \in I$.

\end{defi}

As in the case of a family of Bishop spaces associated to an $I$-family of sets,
the family of Bishop spaces associated to an $(I, \lt)$-family of sets is defined in a minimal way from 
the data of $\Lambda^{\lt}$. According to these
data, the corresponding functions $\phi_{ij}^{\lt}$ behave necessarily in a contravariant manner i.e., 
$\phi_{ij}^{\Lambda^{\lt}} \colon F_j \to F_i$. Moreover, the transport maps $\lambda_{ij}^{\lt}$ are
Bishop morphisms, and not necessarily Bishop isomorphisms.

\begin{defi}\label{def: preorderedspectrum}
Let $(I, \lt)$ be a directed set, and let $\Lambda^{\lt} := (\lambda_0, \lambda_1^{\lt}) \in \Fam(I, \lt)$. 
A \textit{family of Bishop topologies associated to} $\Lambda^{\lt}$ is a pair $\Phi^{\Lambda^{\lt}} := 
\big(\phi_0^{\Lambda^{\lt}}, \phi_1^{\Lambda^{\lt}}\big)$, where $\phi_0^{\Lambda^{\lt}} : I \sto \D V_0$
and $\phi_1^{\Lambda^{\lt}} : \bigcurlywedge_{(i,j) \in \lt(I)} \D F \big(\phi_0^{\Lambda^{\lt}}(j), 
\phi_0^{\Lambda^{\lt}}(i)\big)$, 
such that the following conditions hold:\\[1mm]
\normalfont (i)
\itshape  $\phi_0^{\Lambda^{\lt}}(i) := F_i \subseteq \D F(\lambda_0(i))$,
and $\C F_i := (\lambda_0(i), F_i)$ is a Bishop space, for every $i \in I$.\\[1mm]
\normalfont (ii)
\itshape $\lambda_{ij}^{\lt} \in \Mor(\C F_i, \C F_j)$, for every $(i, j) \in D^{\lt}(I)$.\\[1mm]
\normalfont (iii)
\itshape $\phi_1^{\Lambda^{\lt}}(i, j) := \big(\lambda_{ij}^{\lt}\big)^*$, 
for every $(i, j) \in D^{\lt}(I)$, where, if $f \in F_j$, 
$\big(\lambda_{ij}^{\lt}\big)^* (f) := f \circ \lambda_{ij}^{\lt}$.\\[1mm]
The structure $S(\Lambda^{\lt}) := (\lambda_0, \lambda_1^{\lt}, \phi_0^{\Lambda^{\lt}}, \phi_1^{\Lambda^{\lt}})$
is called a \textit{direct spectrum}\index{direct spectrum}\index{$S(\Lambda^{\lt})$} over $(I, \lt)$, or an 
$(I, \lt)$-\textit{spectrum} with Bishop spaces $(\C F_i)_{i \in I}$ and Bishop morphisms
 $(\lambda_{ij}^{\lt})_{(i,j) \in D^{\lt}(I)}$. 
 
If $M^{\lt} := (\mu_0, \mu_1^{\lt}) \in \Fam(I, \lt_I)$ and
$S(M^{\lt}) := (\mu_0, \mu_1, \phi_{0}^{M^{\lt}}, \phi_{1}^{M^{\lt}})$ is an $(I, \lt)$-spectrum
with Bishop spaces $(\C G_i)_{i \in I}$ and Bishop morphisms $(\mu_{ij}^{\lt})_{(i,j) \in D^{\lt}(I)}$, 
a \textit{direct spectrum-map}\index{direct spectrum-map} $\Psi$ from $S(\Lambda^{\lt})$ to $S(M^{\lt})$,
 in symbols $\Psi \colon S(\Lambda^{\lt}) \To S(M^{\lt})$,\index{$\Psi \colon S(\Lambda^{\lt}) \To S(M^{\lt})$} 
is a direct family-map $\Psi : \Lambda^{\lt} \To M^{\lt}$. The totality of direct spectrum-maps 
from $S(\Lambda^{\lt})$ to $S(M^{\lt})$ is denoted by $\Map_{(I, \lt_I)}(S(\Lambda^{\lt}), 
S(M^{\lt}))$\index{$\Map_{(I, \lt_I)}(S(\Lambda^{\lt}), S(M^{\lt}))$} and it is equipped with the 
equality of $\Map_I(\Lambda^{\lt}, M^{\lt})$.
A direct spectrum-map $\Psi : S(\Lambda^{\lt}) \To S(M^{\lt})$
is called \textit{continuous}, if\index{continuous direct spectrum-map} 
$\forall_{i \in I}\big(\Psi_i \in\Mor(\C F_i, \C G_i)\big)$, and let $\Cnt_{(I, \lt_I)}(S(\Lambda^{\lt}), S(M^{\lt}))$
be their totality, equipped with the equality of $\Map_I(\Lambda^{\lt}, M^{\lt})$. The 
totality $\Spec(I, \lt_I)$\index{$\Spec(I, \lt_I)$} of direct spectra over $(I, \lt_I)$ is equipped with an 
equality defined similarly to the equality on $\Spec(I)$.
A \textit{contravariant} direct spectrum\index{contravariant direct spectrum}\index{$S(\Lambda^{\mt})$} 
 $S(\Lambda^{\mt}) := (\lambda_0, \lambda_1^{\mt} ; \phi_0^{\Lambda^{\mt}}, \phi_1^{\Lambda^{\mt}})$
over $(I, \lt)$, a \textit{contravariant} direct spectrum-map $\Psi : S(\Lambda^{\mt}) \To S(M^{\mt})$, and 
their totalities 
 $\Map_{(I, \lt_I)}(S(\Lambda^{\mt}), S(M^{\mt}))$\index{$\Map_{(I, \lt_I)}(S(\Lambda^{\mt}), S(M^{\mt}))$}, 
 $\Spec(I, \mt_I)$\index{$\Spec(I, \mt_I)$} are defined similarly.
\end{defi}

The following is straightforward to show.

\begin{rem}\label{rem: preorderedspectrum1}
 Let $(I, \lt)$ be a directed set,  
 $S(\Lambda^{\lt}) := (\lambda_0, \lambda_1 ; \phi_0^{\Lambda^{\lt}}, \phi_1^{\Lambda^{\lt}}) \in \Spec(I, \lt_I)$  
 with Bishop spaces $(\C F_i)_{i \in I}$ and Bishop morphisms 
 $(\lambda_{ij}^{\lt})_{(i,j) \in D^{\lt}(I)}$, 
 $S(M^{\lt}) := (\mu_0, \mu_1, \phi_0^{M^{\lt}}, \phi_1^{M^{\lt}}) \in \Spec(I, \lt_I)$ 
 with Bishop spaces $(\C G_i)_{i \in I}$ and Bishop morphisms $(\mu_{ij}^{\lt})_{(i,j) \in D^{\lt}(I)}$,
 and $\Psi \colon S(\Lambda^{\lt}) \To S(M^{\lt})$. Then $\Phi^{\Lambda^{\lt}} := 
 \big(\phi_0^{\Lambda^{\lt}}, \phi_1^{\Lambda^{\lt}}\big)$ is an $(I, \mt)$-family of sets, defined in the obvious dual way, 
 and if $\Psi$ is continuous, then, for every $(i, j) \in D^{\lt}(I)$,
 the following diagram commutes
\begin{center}
\begin{tikzpicture}

\node (E) at (0,0) {$F_j$};
\node[right=of E] (K) {};
\node[right=of K] (F) {$F_i$.};
\node[above=of F] (A) {$\mathlarger{\mathlarger{G_i}}$};
\node [above=of E] (D) {$G_j$};

\draw[->] (E)--(F) node [midway,below]{$\big(\lambda_{ij}^{\lt}\big)^*$};
\draw[->] (D)--(A) node [midway,above] {$\big(\mu_{ij}^{\lt}\big)^*$};
\draw[->] (D)--(E) node [midway,left] {$\big(\Psi_j \big)^*$};
\draw[->] (A)--(F) node [midway,right] {$\big(\Psi_i \big)^*$};

\end{tikzpicture}
\end{center}
\end{rem}

\begin{rem}\label{rem: beforetopondirectedsigma}
Let $(I, \lt)$ be a directed set and
$S(\Lambda^{\lt}) := (\lambda_0, \lambda_1, \phi_0^{\Lambda^{\lt}}, \phi_1^{\Lambda^{\lt}}) \in \Spec(I, D^{\lt}(I))$  
with Bishop spaces $(\C F_i)_{i \in I}$ and Bishop morphisms 
$(\lambda_{ij}^{\lt})_{(i,j) \in \ D^{\lt}(I)}$. If $\Theta \in \prod_{i \in I}^{\mt}F_i$, 
the following operation is a function
\[ f_{\Theta} : \bigg(\sum_{i \in I}^{\lt}\lambda_0 (i) \bigg) \sto \Real, \ \ \ \ f_{\Theta}(i, x) := 
\Theta_i(x), \ \ \ \ 
(i,x) \in \sum_{i \in I}^{\lt}\lambda_0 (i). \]
 
\end{rem}

\begin{proof}
Let $(i, x) =_{\mathsmaller{\sum_{i \in I}^{\lt}\lambda_0 (i)}} (j, y) :\TOT
\exists_{k \mt i, j}\big(\lambda_{ik}^{\lt}(x) =_{\mathsmaller{\lambda_0 (k)}} 
\lambda_{jk}^{\lt}(y)\big)$. Since 
$\Theta_i = \phi_{ki}^{\mt}(\Theta_k) := (\lambda_{ik}^{\lt})^*(\Theta_k) := 
\Theta_k \circ \lambda_{ik}^{\lt},$
and similarly $\Theta_j = \Theta_k \circ \lambda_{jk}^{\lt}$, we have that
$\Theta_i(x) = \big[\Theta_k \circ \lambda_{ik}^{\lt}\big](x) := \Theta_k\big(\lambda_{ik}^{\lt}(x)\big)
= \Theta_k\big(\lambda_{jk}^{\lt}(y)\big) := \big[\Theta_k \circ \lambda_{jk}^{\lt}\big](y) = 
\Theta_j(y)$.
\end{proof}

\begin{defi}\label{def: topondirectedsigma}
Let $(I, \lt)$ be a directed set and
$S(\Lambda^{\lt}) := (\lambda_0, \lambda_1, \phi_0^{\Lambda^{\lt}}, \phi_1^{\Lambda^{\lt}}) \in \Spec(I, \lt_I)$  
with Bishop spaces $(\C F_i)_{i \in I}$ and Bishop morphisms 
$(\lambda_{ij}^{\lt})_{(i,j) \in D^{\lt}(I)}$.
The Bishop space 
\[ \sum_{i \in I}^{\lt}\C F_i := \bigg(\sum_{i \in I}^{\lt}\lambda_0 (i), \int_{i \in I}^{\lt}F_i \bigg) 
\ \ \ \ \mbox{where} \ \ 
\int_{i \in I}^{\lt}F_i := \bigvee_{\mathsmaller{\Theta \in \prod_{i \in I}^{\mt}F_i}}f_{\Theta}, \]
is the \textit{sum Bishop space} of $S(\Lambda^{\lt})$\index{sum Bishop space of $S^{\lt}$}. If $S^{\mt}$ is 
a contravariant direct spectrum over $(I, \lt)$, the sum Bishop space of $S(\Lambda^{\mt})$ is defined 
dually.
\end{defi}

\begin{lem}\label{lem: preorderdependent}
Let $S(\Lambda^{\lt}) := (\lambda_0, \lambda_1^{\lt}, \phi_0^{\Lambda^{\lt}}, \phi_1^{\Lambda^{\lt}}),   
S(M^{\lt}) := (\mu_0, \mu_1^{\lt}, \phi_0^{M^{\lt}}, \phi_1^{M^{\lt}}) \in \Spec(I, \lt_I)$ with
Bishop spaces $(\C F_i)_{i \in I}$ and with Bishop spaces $(\C G_i)_{i \in I}$, respectively, 
and let $\Psi : S(\Lambda^{\lt}) \to S(M^{\lt})$ be continuous.  
If $H \in \prod_{i \in I}^{\mt}G_i$, the dependent operation
$H^* : \bigcurlywedge_{i \in I}F_i$, defined by $H^*_i := \Psi^*_i(H_i) := H_i \circ \Psi_i$, for every $i \in I$, 
is in $\prod_{i \in I}^{\mt}F_i$.
\end{lem}

\begin{proof}
If $i \lt j$, we need to show that $H_i^* = (\lambda_{ij}^{\lt})^*(H_j^*) = H_j^* \circ \lambda_{ij}^{\lt}$. Since
$H \in \prod_{i \in I}^{\mt}G_i$, we have that $H_i = H_j \circ \mu_{ij}^{\lt}$, and by the continuity of $\Psi$
and the commutativity of the diagram
\begin{center}
\begin{tikzpicture}

\node (E) at (0,0) {$\mu_0(i)$};
\node[right=of E] (F) {$\mu_0(j)$,};
\node[above=of F] (A) {$\lambda_0(j)$};
\node [above=of E] (D) {$\lambda_0(i)$};

\draw[->] (E)--(F) node [midway,below]{$\mu_{ij}^{\lt}$};
\draw[->] (D)--(A) node [midway,above] {$\lambda_{ij}^{\lt}$};
\draw[->] (D)--(E) node [midway,left] {$\Psi_i$};
\draw[->] (A)--(F) node [midway,right] {$\Psi_j$};

\end{tikzpicture}
\end{center}
\[ H_j^* \circ \lambda_{ij}^{\lt} := \Psi_j^*(H_j) \circ \lambda_{ij}^{\lt}
:= \big(H_j \circ \Psi_j \big) \circ \lambda_{ij}^{\lt}
:= H_j(0) \circ \big(\Psi_j \circ \lambda_{ij}^{\lt}\big) \]
\[ = H_j \circ \big(\mu_{ij}^{\lt} \circ \Psi_i \big)
= \big(H_j \circ \mu_{ij}^{\lt}\big) \circ \Psi_i
= H_i \circ \Psi_i
:= \Psi^*_i(H_i)
:= H^*_i. \qedhere \]
\end{proof}

\begin{prop}\label{prp: Spectrum1}
Let $S(\Lambda^{\lt}) := (\lambda_0, \lambda_1^{\lt}, \phi_0^{\Lambda^{\lt}}, \phi_1^{\Lambda^{\lt}})$ and  
$S(M^{\lt}) := (\mu_0, \mu_1^{\lt}, \phi_0^{M^{\lt}}, \phi_1^{M^{\lt}})$ be spectra over $(I, \lt_I)$, 
and let $\Psi : S(\Lambda^{\lt}) \To S(M^{\lt})$. \\[1mm]
\normalfont (i)
\itshape If $i \in I$, then $e_i^{\Lambda^{\lt}} \in \Mor\big(\C F_i, \sum_{i \in I}^{\lt}\C F_i\big)$.\\[1mm]
\normalfont (ii)
\itshape If $\Psi$ is continuous, then 
$\Sigma^{\lt} \Psi \in \Mor\big(\sum_{i \in I}^{\lt}\C F_i, \sum_{i \in I}^{\lt}\C G_i \big)$.\\[1mm]
\end{prop}

\begin{proof}
(i) By the $\bigvee$-lifting of morphisms it suffices to show that 
$\forall_{\Theta \in \prod_{i \in I}^{\mt}F_i}\big(f_{\Theta} \circ e_i^{\Lambda^{\lt}}
\in F_i\big)$.
If $x \in \lambda_0(i)$, then
$\big(f_{\Theta} \circ e_i^{\Lambda^{\lt}}\big)(x) := f_{\Theta}(i, x) 
:= \Theta_i(x),$ hence
$f_{\Theta} \circ e_i^{\Lambda^{\lt}} := \Theta_i \in F_i$.\\
(ii) By the $\bigvee$-lifting of morphisms and Definition~\ref{def: topondirectedsigma} it suffices to show that 
\[ \forall_{H \in \prod_{i \in I}^{\mt}G_i}\bigg(g_{H} \circ \Sigma^{\lt} \Psi \in \int_{i \in I}^{\lt}F_i\bigg). \]
If $i \in I$ and $x \in \lambda_0(i)$, and if $H^* \in \prod_{i \in I}^{\mt}F_i$, defined 
in Lemma~\ref{lem: preorderdependent}, then
$ \big(g_{H} \circ \Sigma^{\lt} \Psi\big)(i, x) := g_{H}(i, \Psi_i(x))
:= H_i (\Psi_i(x))
=: (H_i \circ \Psi_i)(x)
:= f_{H^*}(i, x),
$
and $g_{H} \circ \Sigma^{\lt} \Psi := f_{H^*} \in \int_{i \in I}^{\lt}F_i$.
\end{proof}

\section{On families of equivalence classes}
\label{sec: fameqclass}

Roughly speaking, a family of subsets of a set $X$ indexed by some set $I$ is an assignment routine 
$\lambda_0 : I \sto \C P(X)$ that behaves like a function i.e., if $i =_I j$, then $\lambda_0(i) =_{\C P(X)} 
\lambda_0 (j)$. The following definition is a formulation of this rough description that reveals the witnesses of the 
equality $\lambda_0(i) =_{\C P(X)} \lambda_0 (j)$. This is done ``internally'', through the embeddings of the subsets 
into $X$. The equality $\lambda_0(i) =_{\D V_0} \lambda_0 (j)$, which 
is defined ``externally'' through the transport maps, follows, and a family of subsets
is also a family of sets. From the theory of families of subsets (see~\cite{Pe20}, chapter 4)
we present here only what is relevant to the topological part of this paper.

\begin{defi}\label{def: famofsubsets}
Let $X$ and $I$ be sets. A \textit{family of subsets}\index{family of subsets} of $X$ indexed by $I$,
or an $I$-\textit{family of subsets}\index{$I$-family of subsets} of $X$, is a triplet
$\Lambda(X) := (\lambda_0, \C E^X, \lambda_1)$, where
\index{$\Lambda(X)$}\index{$\C E^X$}
$\lambda_0 : I \sto \D V_0$,
\[ \C E^X : \bigcurlywedge_{i \in I}\D F\big(\lambda_0(i), X\big), \ \ \ \ \C E^X(i) =: \C E_i^X; \ \ \ \ i \in I, \]
\[ \lambda_1 : \bigcurlywedge_{(i, j) \in D(I)}\D F\big(\lambda_0(i), \lambda_0(j)\big), \ \ \ \ 
\lambda_1(i, j) =: \lambda_{ij}; \ \ \ \ (i, j) \in D(I), \]
such that the following conditions hold:\\[1mm]
\normalfont (a) 
\itshape For every $i \in I$, the function $\C E_i^X : \lambda_0(i) \to X$ is an embedding.\\[1mm]
\normalfont (b) 
\itshape For every $i \in I$, we have that $\lambda_{ii} := \id_{\lambda_0(i)}$.\\[1mm]
\normalfont (c) 
\itshape For every $(i, j) \in D(I)$ we have that 
$\C E_i^X = \C E_j^X \circ \lambda_{ij}$ and $\C E_j^X = \C E_i^X \circ \lambda_{ji}$
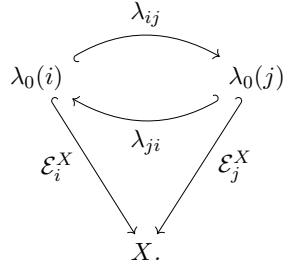
\begin{center}
\resizebox{4cm}{!}{%
\begin{tikzpicture}

\node (E) at (0,0) {$\lambda_0(i)$};
\node[right=of E] (B) {};
\node[right=of B] (F) {$\lambda_0(j)$};
\node[below=of B] (C) {};
\node[below=of C] (A) {$X$.};

\draw[left hook->,bend left] (E) to node [midway,above] {$\lambda_{ij}$} (F);
\draw[left hook->,bend left] (F) to node [midway,below] {$\lambda_{ji}$} (E);
\draw[right hook->] (E)--(A) node [midway,left] {$\C E_{i}^X \ $};
\draw[left hook->] (F)--(A) node [midway,right] {$ \ \C E_j^X$};

\end{tikzpicture}
}
\end{center} 
$\C E^X$ is a modulus of embeddings\index{modulus of embeddings} for $\lambda_0$, and $\lambda_1$ 
a modulus of transport maps for $\lambda_0$. Let 
$\Lambda := (\lambda_0, \lambda_1)$ be the $I$-family of sets 
that corresponds to $\Lambda(X)$\index{the $I$-family that corresponds to $\Lambda(X)$}.
If $(A, i_A^X) \in \C P(X)$, the \textit{constant} $I$-\textit{family of subsets}\index{constant family of subsets} 
$A$\index{$C^A_X$}
is the pair 
$C^{A}(X) := (\lambda_0^{A}, \C E^{X,A}, \lambda_1^A)$, where $\lambda_0 (i) := A$, $\C E_i^{X,A} := i_A^X$, and 
$\lambda_1 (i, j) := \id_A$, for every $i \in I$ and $(i, j) \in D(I)$ 

\end{defi}

\begin{defi}\label{def: subfammap}
If $\Lambda(X) := (\lambda_0, \C E^X, \lambda_1), M(X) := (\mu_0, \C Z^X, \mu_1)$ and 
$N(X) := (\nu_0, \C H^X, \nu_1)$ are $I$-families of subsets of $X$, a 
\textit{family of subsets-map}\index{family of subsets-map}
$\Psi \colon \Lambda(X) \To M(X)$\index{$\Psi \colon \Lambda(X) \To M(X)$}
from $\Lambda(X)$ to $M(X)$ is a dependent operation
$\Psi : \bigcurlywedge_{i \in I}\D F\big(\lambda_0(i), \mu_0(i)\big)$, where $\Psi(i) =: \Psi_i$,
for every $i \in I$, such that, for every $i \in I$, the following diagram commutes\footnote{Trivially, for 
every $i \in I$ the map $\Psi_i \colon \lambda_0(i) \to \mu_0(i)$ is an embedding.}
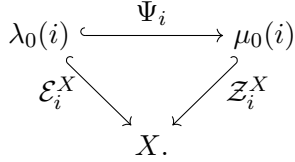
\begin{center}
\begin{tikzpicture}

\node (E) at (0,0) {};
\node[right=of E] (F) {$X$.};
\node[above=of E] (A) {$\lambda_0 (i)$};
\node [right=of F] (B) {};
\node [above=of B] (G) {$\mu_0 (i)$};

\draw[right hook->] (A)--(F) node [midway,left] {$\C E_{i}^X \ $};
\draw[left hook->] (G)--(F) node [midway,right] {$\ \C Z_i^X$};
\draw[right hook->] (A)--(G) node [midway,above] {$\Psi_{i}$};

\end{tikzpicture}
\end{center}
The totality $\Map_I(\Lambda(X), M(X))$\index{$\Map_I(\Lambda(X), M(X))$} of family of subsets-maps 
from $\Lambda(X)$ to $M(X)$ is
equipped with the pointwise equality. 
If $\Psi \colon \Lambda(X) \To M(X)$ and $\Xi \colon M(X) \To N(X)$, the \textit{composition family of subsets-map} 
\index{composition family of subsets-map} 
$\Xi \circ \Psi \colon \Lambda(X) \To N(X)$ is defined by 
$(\Xi \circ \Psi)(i) := \Xi_i \circ \Psi_i$,
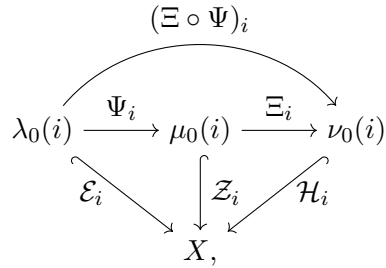
\begin{center}
\begin{tikzpicture}

\node (E) at (0,0) {$\lambda_0(i)$};
\node[right=of E] (B) {$\mu_0(i)$};
\node[right=of B] (F) {$\nu_0(i)$};
\node[below=of B] (C) {$X$,};

\draw[left hook->] (E)--(C) node [midway,left] {$\C E_i \ $};
\draw[right hook->] (B)--(C) node [midway,right] {$\C Z_i$};
\draw[right hook->] (F)--(C) node [midway,right] {$ \ \C  H_{i} $};
\draw[->] (E)--(B) node [midway,above] {$\Psi_i$};
\draw[->] (B)--(F) node [midway,above] {$\Xi_i$};
\draw[->,bend left=50] (E) to node [midway,above] {$(\Xi \circ \Psi)_{i}$} (F);

\end{tikzpicture}
\end{center} 
for every $i \in I$. The identity family of subsets-map $\Id_{\Lambda(X)} \colon \Lambda(X) \To \Lambda(X)$ 
and the equality on the totality $\Fam(I,X)$\index{$\Fam(I,X)$} 
of $I$-families of subsets of $X$ are defined as in Definition~\ref{def: map}.
\end{defi}

We see no obvious reason, like the one for $\Fam(I)$, not to consider 
$\Fam(I, X)$ to be a set. In the case of $\Fam(I)$ the constant $I$-family $\Fam(I)$ would be in $\Fam(I)$, while
the constant $I$-family $\Fam(I, X)$ is not clear how could be seen as a family of subsets of $X$. If 
$\nu_0(i) := \Fam(I, X)$, for every $i \in I$, we need to define a modulus of embeddings $\C N^X_i \colon \Fam(I, X) 
\eto X$, for every $i \in I$. From the given data one could define the assignment routine $\C N^X_i$ by the rule 
$\C N^X_i\big(\Lambda(X)\big) := \C E_i^X(u_i)$, if it is known that $u_i \in \lambda_0(i)$. Even in that case, the 
assignment routine $\C N^X_i$ cannot be shown to satisfy the expected properties. If $\C N^X_i$ was defined 
by the rule $\C N^X_i\big(\Lambda(X)\big) := x_0$, for some $x_0 \in X$, then it cannot be an embedding.

\begin{defi}\label{def: setofsubsets}
If $I, X \in \D V_0$, a \textit{set of subsets}\index{set of subsets} of $X$ indexed by $I$, or an $I$-\textit{set 
of subsets}\index{$I$-set of subsets} of $X$, is triplet $\Lambda(X) := (\lambda_0, \C E^X, \lambda_1) \in \Fam(I, X)$
such that the following condition is satisfied:
\[ Q(\Lambda(X)) :\TOT \forall_{i, j \in I}\big(\lambda_0(i) =_{\C P(X)} \lambda_0(j) \To i =_I j\big). \]
Let $\Set(I, X)$ be their totality, equipped with the canonical equality on $\Fam(I, X)$. 
\end{defi}

\begin{rem}\label{rem: Isetssubsets1}
If $\Lambda(X) \in \Set(I, X)$ and $M(X) \in \Fam(I, X)$ such that $\Lambda(X) =_{\Fam(I, X)} M(X)$, then 
$M(X) \in \Set(I, X)$. 
\end{rem}

\begin{proof}
Let $\Phi \colon \Lambda(X) \To M(X)$ and $\Psi \colon M(X) \To \Lambda(X)$ such that 
$(\Phi, \Psi) \colon \Lambda(X) =_{\Fam(I, X)} M(X)$. Let also $(f, g) \colon \mu_0(i) =_{\C P(X)} \mu_0(j)$. It 
suffices to show that $\lambda_0(i) =_{\C P(X)} \lambda_0(j)$.  
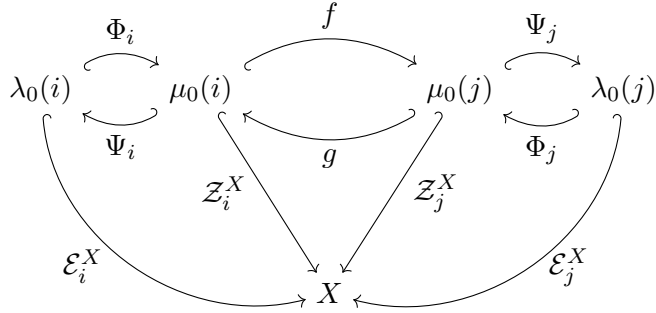
 \begin{center}
\resizebox{9cm}{!}{%
\begin{tikzpicture}

\node (E) at (0,0) {$\mu_0(i)$};
\node[right=of E] (B) {};
\node[right=of B] (F) {$\mu_0(j)$};
\node[below=of B] (C) {};
\node[below=of C] (A) {$X$};
\node[right=of F] (K) {$\lambda_0(j)$};
\node[left=of E] (L) {$\lambda_0(i)$};

\draw[left hook->,bend left] (E) to node [midway,above] {$f$} (F);
\draw[left hook->,bend left] (F) to node [midway,below] {$g$} (E);
\draw[right hook->] (E)--(A) node [midway,left] {$\C Z^X_i \ $};
\draw[left hook->] (F)--(A) node [midway,right] {$ \ \C Z^X_j$};

\draw[left hook->,bend left] (L) to node [midway,above] {$\Phi_i$} (E);
\draw[left hook->,bend left] (E) to node [midway,below] {$\Psi_i$} (L);
\draw[right hook->,bend right=50] (L) to node [midway,left] {$\C E^X_i \ \ $} (A);

\draw[left hook->,bend left] (F) to node [midway,above] {$\Psi_j$} (K);
\draw[left hook->,bend left] (K) to node [midway,below] {$\Phi_j$} (F);
\draw[left hook->,bend left=50] (K) to node [midway,right] {$\ \C E^X_j $} (A);

\end{tikzpicture}
}
\end{center}
If we define $f{'} := \Psi_j \circ f \circ \Phi_i$ and $g{'} := \Psi_i \circ g \circ \Phi_j$, it is straightforward to 
show that $(f{'}, g{'}) \colon \lambda_0(i) =_{\C P(X)} \lambda_0(j)$, hence $i =_I j$.
\end{proof}

By the previous remark $Q(\Lambda(X))$ is an extensional property on $\Fam(I,X)$. Since $\Set(I,X)$ is defined by 
separation on $\Fam(I,X)$, and since we see no objection to consider $\Fam(I,X)$ to be a set, we also see no objection  
to consider $\Set(I,X)$ to be a set.

\begin{defi}\label{def: sublambdaI}
Let $\Lambda(X) := (\lambda_0, \C E^X, \lambda_1) \in \Fam(I,X)$. Let the equality 
$=_I^{\Lambda(X)}$\index{$=_I^{\Lambda(X)}$} on $I$
\index{equality on the index-set induced by a family of subsets}
given by $ i =_I^{\Lambda(X)} j : \TOT \lambda_0 (i) =_{\C P(X)} \lambda_0 (j)$, for every $i, j \in I$.
The \textit{set $\lambda_0 I(X)$ of subsets of $X$
generated by $\Lambda(X)$}\index{set of subsets generated by a family of subsets}\index{$\lambda_0 I(X)$} is 
the totality $I$ equipped with the equality $=_I^{\Lambda(X)}$. We write $\lambda_0(i) \in \lambda_0I(X)$,
instead of $i \in I$, when $I$ is equipped with the equality $=_I^{\Lambda(X)}$. 
The new operation $\lambda_0^* : I \sto I$ from $(I, =_I)$ to $(I, =_I^{\Lambda(X)})$ is 
defined by the identity-rule.
\end{defi}

Clearly, $\lambda_0^*$ is a function. The following is easy to show (see also~\cite{Pe20}, section 3.7).

\begin{prop}\label{prp: subFamtoset1}
Let $\Lambda(X) := (\lambda_0, \C E^X, \lambda_1) \in \Set(I, X)$, and let $Y$ be a set.
If $f \colon I \to Y$, there is a unique function $\lambda_0 f \colon \lambda_0 I(X) \to Y$ such that the 
following diagram commutes
\begin{center}
\begin{tikzpicture}

\node (E) at (0,0) {$\lambda_0 I(X)$};
\node [above=of E] (D) {$I$};
\node[right=of D] (F) {$Y.$};

\draw[dashed, ->] (E)--(F) node [midway,right] {$\ \lambda_0 f$};
\draw [->]  (D)--(E) node [midway,left] {$\lambda_0$};
\draw[->] (D)--(F) node [midway,above] {$f$};

\end{tikzpicture}
\end{center}
Conversely, if $f \colon I \sto Y$ and $f^* \colon \lambda_0 I(X)  \to Y$ such that 
the corresponding diagram commutes, then $f$ is a function and $f^*$ is equal to the function from 
$\lambda_0 I(X)$ to $Y$ generated by $f$.
\end{prop}

Although a family of equivalence classes is not, in general, a set of subsets, 
we can define functions on them, if we use appropriate functions on their index-set.

\begin{defi}\label{def: equivstructure}
If $X$ is a set and $R_X(x, x{'})$ is an extensional property on $X \times X$ that 
satisfies the conditions of an equivalence relation, we call the pair $(X,R_X)$  
an \textit{equivalence structure}\index{equivalence structure}. If $(Y, S_Y)$ is an equivalence structure,
a function $f \colon X \to Y$ is an \textit{equivalence preserving function}\index{equivalence preserving
function}, or an $(R_X, S_Y)$-function\index{$(R_X, S_Y)$-function}, if
\[ \forall_{x, x{'} \in X}\big(R(x, x{'}) \To S(f(x), f(x{'}))\big). \]
If, for every $x,x{'} \in X$, the converse implication holds, we say that  
$f$ is an $(R_X, S_Y)$-embedding\index{$(R_X, S_Y)$-embedding}.
Let $\D F(R_X, S_Y)$\index{$\D F(R_X, S_Y)$} be the set of $(R_X, S_Y)$-functions\footnote{By the extensionality of $S_Y$
the property of being an $(R_X, S_Y)$-function is extensional on $\D F(X, Y)$.}. 
\end{defi}

\begin{prop}\label{prp: equivstr1}
If $(X, R_X)$ is an  equivalence structure, let $R(X) := \big(\rho_0, \C R^X, \rho_1\big)$, where\index{$R(X)$}
$\rho_0 \colon X \sto \D V_0$ is defined by 
$\rho_0(x) := \{y \in X \mid R_X(y, x)\}$, for every $x \in X$,
and the dependent operations 
$\C R^X \colon \bigcurlywedge_{x \in X}\D F\big(\rho_0(x), X\big)$, 
$\rho_1 \colon \bigcurlywedge_{(x,x{'}) \in D(X)}\D F\big(\rho_0(x), \rho_0(x{'})\big)$ are defined by 
\[ \C R^X_x \colon \rho_0(x) \eto X \ \ \ \ y \mapsto y; \ \ \ \ y \in \rho_0(x), \]
\[ \rho_1(x,x{'}) := \rho_{xx{'}} \colon \rho_0(x) \to \rho_0(x{'}) \ \ \ \ y \mapsto y; \ \ \ \ y \in 
\rho_0(x). \] 
Then $R(X) \in \Fam(X, X)$, and $\forall_{x x{'} \in X}\big(\rho_0(x) =_{\C P(X)} \rho_0(x{'}) \To 
R(x, x{'})\big)$.

\end{prop}

\begin{proof}
By the extensionality of $R_X$ the set $\rho_0(x)$ is a well-defined extensional subset of $X$. If $x =_X x{'}$ and 
$R_X(y, x)$, then by the extensionality of $R_X$ we get $R_X(y, x{'})$, hence $\rho_{xx{'}}$ is well-defined. 
Let $(f, g) \colon \rho_0(x) =_{\C P(X)} \rho_0(x{'})$
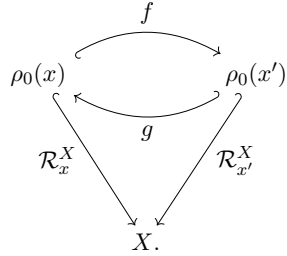
\begin{center}
\resizebox{4cm}{!}{%
\begin{tikzpicture}

\node (E) at (0,0) {$\rho_0(x)$};
\node[right=of E] (B) {};
\node[right=of B] (F) {$\rho_0(x{'})$};
\node[below=of B] (C) {};
\node[below=of C] (A) {$X$.};

\draw[left hook->,bend left] (E) to node [midway,above] {$f$} (F);
\draw[left hook->,bend left] (F) to node [midway,below] {$g$} (E);
\draw[right hook->] (E)--(A) node [midway,left] {$\C R^X_x \ $};
\draw[left hook->] (F)--(A) node [midway,right] {$ \ \C R^X_{x{'}}$};

\end{tikzpicture}
}
\end{center} 
If $y \in \rho_0(x) :\TOT R_X(y, x)$, then $f(y) \in \rho_0(x{'}) :\TOT R_X(f(y), x{'})$, and by the commutativity of 
the corresponding above diagram we get $f(y) =_X y$. Hence by the extensionality of $R_X$ we get $R_X(y, x{'})$.
Since $R_X(y, x)$ implies $R_X(x, y)$, by transitivity we get $R_X(x, x{'})$.
\end{proof}

\begin{cor}\label{cor: corequivstr1}
 Let $\Eql(X) := \big(\eql_0^X, \C E^X, \eql_1^X\big)$\index{$\Eql(X)$} be the $X$-family of subsets of $X$ 
 induced by the equivalence relation $=_X$\index{$\eql_0^X(x)$} i.e., $\eql_0^X(x) := \{y \in X \mid y =_X x\}$. 
 Then $\Eql(X) \in \Set(X,X)$.
\end{cor}

\begin{proof}
It follows immediately from Proposition~\ref{prp: equivstr1}. 
\end{proof}

\section{Direct limit of Bishop spaces}
\label{sec: directlimit}

If $X$ is a set, by Corollary~\ref{cor: corequivstr1} the family $\Eql(X) := \big(\eql_0^X, \C E^X, 
\eql_1^X\big) \in \Set(X, X)$,
where $\eql_0^X(x) := \{y \in X \mid y =_X x\}$. Consequently, if $f \colon X \to Y$ , there is unique 
$\eql_0 f \colon \eql_0 X(X) \to Y$ such that the following diagram commutes
\begin{center}
\begin{tikzpicture}

\node (E) at (0,0) {$\eql_0 X(X)$};
\node [above=of E] (D) {$X$};
\node[right=of D] (F) {$Y,$};

\draw[dashed, ->] (E)--(F) node [midway,right] {$ \ \eql_0 f$};
\draw [->] (D)--(E) node [midway,left] {$\eql^*_0$};
\draw[->] (D)--(F) node [midway,above] {$f$};

\end{tikzpicture}
\end{center}
where $\eql_0 X(X)$ is the totality $X$ with the equality $x =_{\eql_0 X (X)} x{'} 
:\TOT \eql_0^X(x) =_{\C P(X)} \eql_0^X(x{'})$. As $\Eql(X) \in \Set(X, X)$, we get $\eql_0^X(x) 
=_{\C P(X)} \eql_0^X(x{'}) \TOT x =_X x{'}$.
The map $\eql_0^* \colon X \to \eql_0 X(X)$ is defined by the identity map-rule, written in the 
form $x \mapsto \eql_0^X(x)$, for every $x \in X$. We use the set $\eql_0 X(X)$ to define the direct
limit of a direct spectrum of Bishop spaces\footnote{In this way, our definition of the direct limit of a direct spectrum of
Bishop spaces is in complete analogy to the corresponding definition for topological spaces.}. 
In what follows we avoid including the superscript $X$ in our notation.

\begin{defi}\label{def: Lim}
Let $S(\Lambda^{\lt}) := (\lambda_0, \lambda_1^{\lt}, \phi_0^{\Lambda^{\lt}}, 
\phi_1^{\Lambda^{\lt}}) \in \Spec(I, \lt_I)$ and 
$ \eql_0 \colon \sum_{i \in I}^{\lt} \lambda_0(i) \sto  \D V_0$, defined by 
\[ \eql_0 (i, x) := \bigg\{(j, y) \in \sum_{i \in I}^{\lt} \lambda_0(i) \mid (j, y)
=_{\mathsmaller{\sum_{i \in I}^{\lt} \lambda_0(i)}} (i, x)\bigg\}; \ \ \ \ (i, x) \in 
\sum_{i \in I}^{\lt} \lambda_0(i).
\]
The \textit{direct limit} $\underset{\to} \Lim \lambda_0 (i)$\index{$\underset{\to} \Lim \lambda_0 (i)$} 
of\index{direct limit of a spectrum} $S(\Lambda^{\lt})$ is the set
\[  \underset{\to} \Lim \lambda_0 (i) := \eql_0 \sum_{i \in I}^{\lt} \lambda_0(i)
\bigg(\sum_{i \in I}^{\lt} \lambda_0(i)\bigg), \]
\[ \eql_0 (i, x) =_{\underset{\to} \Lim \lambda_0 (i)} \eql_0 (j, y) : \TOT 
 \eql_0 (i, x) =_{\C P \big(\sum_{i \in I}^{\lt} \lambda_0(i)\big)} \eql_0 (j, y)
\TOT (i, x) =_{\sum_{i \in I}^{\lt} \lambda_0(i)} (j, y). \]
We write $\eql_0^{\Lambda^{\lt}}$\index{$\eql_0^{\Lambda^{\prec}}$} when we need to express 
the dependence of $\eql_0$ from $\Lambda^{\lt}$.
\end{defi}

\begin{rem}\label{rem: directLim1}
If $S(\Lambda^{\lt}) := (\lambda_0, \lambda_1^{\lt}, \phi_0^{\Lambda^{\lt}},
\phi_1^{\Lambda^{\lt}}) \in \Spec(I, \lt_I)$ and $i \in I$, the operation
$\eql_i \colon \lambda_0 (i) \sto 
\underset{\to} \Lim \lambda_0 (i)$, defined by $\eql_i (x) := \eql_0 (i, x)$, for every $x \in \lambda_0 (i)$,
is a function.
\end{rem}

\begin{proof}
 If $x, x{'} \in \lambda_0 (i)$ such that $x =_{\lambda_0 (i)} x{'}$, then
 \begin{align*}
  \eql_i (x) =_{\underset{\to} \Lim \lambda_0 (i)} \eql_i (x{'}) & : \TOT 
  \eql_0 (i, x) =_{\underset{\to} \Lim \lambda_0 (i)} \eql_0 (i, x{'})\\
  & \TOT (i, x) =_{\sum_{i \in I}^{\lt} \lambda_0(i)} (i, x{'})\\
  & : \TOT \exists_{k \in I}\big(i \lt k \ \& \ \lambda_{ik}^{\lt}(x) =_{\lambda_0(k)}
  \lambda_{ik}^{\lt}(x{'})\big),
 \end{align*}
which holds, since $\lambda_{ik}^{\lt}$ is a function, and hence 
if $x =_{\lambda_0 (i)} x{'}$, then
$\lambda_{ik}^{\lt}(x) =_{\lambda_0(k)} \lambda_{ik}^{\lt}(x{'})$, for every $k \in I$ such that
$i \lt k$. Such a $k \in I$ always exists e.g., one can take $k := i$.
\end{proof}

\begin{defi}\label{def: Limtop}
Let $S(\Lambda^{\lt}) := (\lambda_0, \lambda_1^{\lt}, \phi_0^{\Lambda^{\lt}}, 
\phi_1^{\Lambda^{\lt}}) \in \Spec(I, \lt_I)$ with Bishop spaces $(F_i)_{i \in I}$ and Bishop morphisms 
$(\lambda_{ij}^{\lt})_{(i,j) \in D^{\lt}(I)}$.
The \textit{direct limit of} $S(\Lambda^{\lt})$\index{direct limit of a direct spectrum} 
is the Bishop space\index{$\underset{\to} \Lim \C F_i$}\index{$\underset{\to} \Lim F_i$}
\[ \underset{\to} \Lim \C F_i := \big(\underset{\to} \Lim \lambda_0 (i), \underset{\to} \Lim F_i\big), 
\ \ \ \ \mbox{where} \] 
\[ \underset{\to} \Lim F_i := 
\bigvee_{\mathsmaller{\Theta \in \prod_{i \in I}^{\mt}F_i}}\eql_0  f_{\Theta}, \]
\[ \eql_0 f_{\Theta} \big(\eql_0 (i, x)\big) := f_{\Theta}(i, x) := \Theta_i(x); \ \ \ \ \eql_0 (i, x)
\in \underset{\to} \Lim \lambda_0 (i) \]
\begin{center}
\begin{tikzpicture}

\node (E) at (0,0) {$\underset{\to} \Lim \lambda_0 (i)$};
\node [above=of E] (D) {$\sum_{i \in I}^{\lt} \lambda_0(i)$};
\node[right=of D] (F) {$\Real$.};

\draw[dashed, ->] (E)--(F) node [midway,right] {$ \ \eql_0 f_{\Theta}$};
\draw [->] (D)--(E) node [midway,left] {$\eql^*_0$};
\draw[->] (D)--(F) node [midway,above] {$f_{\Theta}$};

\end{tikzpicture}
\end{center}\end{defi}

\begin{rem}\label{rem: Limtop1}
 If $(I, \lt)$ is a directed set, $\C G := (Y, G)$ is a Bishop space, and 
 $S(\Lambda^{\lt, Y})$
 is the constant direct spectrum over $(I, \lt_I)$ with Bishop space $\C G$ and Bishop morphism $\id_Y$, the direct limit 
 $\underset{\to} \Lim \C G$ of $S(\Lambda^{\lt, Y})$
 is Bishop-isomorphic to $\C G$. Moreover, every Bishop 
 space is Bishop-isomorphic to the direct limit of a direct spectrum over any given directed set.
\end{rem}

%
%
%

\begin{prop}[Universal property of the direct limit]\label{prp: universaldirect}
 If $S(\Lambda^{\lt}) := (\lambda_0, \lambda_1^{\lt}, \phi_0^{\Lambda^{\lt}}, \phi_1^{\Lambda^{\lt}}) \in \Spec(I, \lt_I)$ 
with Bishop spaces $(\C F_i)_{i \in I}$ and Bishop morphisms $(\lambda_{ij}^{\lt})_{(i,j) \in \lt(I)}$, 
its direct limit $\underset{\to} \Lim \C F_i$ satisfies the universal property of
direct limits\index{universal property of direct limits} i.e., \\[1mm]
\normalfont (i)
\itshape For every $i \in I$, we have that $\eql_i \in \Mor(\C F_i, \underset{\to} \Lim \C F_i)$.\\[1mm]
\normalfont (ii)
\itshape If $i \lt_I j$, the following left diagram commutes
\begin{center}
\begin{tikzpicture}

\node (E) at (0,0) {$\underset{\to} \Lim \lambda_0(i)$};
\node[below=of E] (F) {};
\node [right=of F] (B) {$\lambda_0(j)$};
\node [left=of F] (C) {$\lambda_0(i)$};
\node [right=of B] (D) {$\lambda_0(i)$};
\node[right=of D] (K) {};
\node[right=of K] (L) {$\lambda_0(j)$.};
\node[above=of K] (M) {$Y$};

\draw[->] (B)--(E) node [midway,right] {$ \ \eql_j$};
\draw[->] (C)--(E) node [midway,left] {$\eql_i \ $};
\draw[->] (C)--(B) node [midway,below] {$\lambda_{ij}^{\lt}$};

\draw[->] (L)--(M) node [midway,right] {$ \ \varepsilon_j$};
\draw[->] (D)--(M) node [midway,left] {$\varepsilon_i \ $};
\draw[->] (D)--(L) node [midway,below] {$\lambda_{ij}^{\lt}$};

\end{tikzpicture}
\end{center}
\normalfont (iii)
\itshape If $\C G := (Y, G)$ is a Bishop space and 
$\varepsilon_i : \lambda_0(i) \to Y \in \Mor(\C F_i, \C G)$, for every $i \in I$, are such that if 
$i \lt j$, the
above right diagram commutes,
there is a unique function
$h \colon \underset{\to} \Lim \lambda_0(i) \to Y \in \Mor(\underset{\to} \Lim \C F_i, \C G)$ such that 
the following diagrams commute
\begin{center}
\begin{tikzpicture}

\node (E) at (0,0) {$Y$};
\node[below=of E] (F) {};
\node [right=of F] (B) {$\lambda_0(j)$,};
\node [left=of F] (C) {$\lambda_0(i)$};
\node[below=of F] (G) {$\underset{\to} \Lim \lambda_0(i)$.};

\draw[->] (B)--(E) node [midway,right] {$ \ \  \varepsilon_j$};
\draw[->] (C)--(E) node [midway,left] {$\varepsilon_i \ $};
\draw[->] (C)--(B) node [midway,below] {$\lambda_{ij}^{\lt} \ \ \ \ \ \  $};
\draw[->] (B)--(G) node [midway,right] {$ \ \eql_j$};
\draw[->] (C)--(G) node [midway,left] {$\eql_i \ $};
\draw[dashed,->] (G)--(E) node [midway,near end] {$\ \ \ h$};

\end{tikzpicture}
\end{center}
\end{prop}

\begin{proof}
For the proof of (i), we use the $\bigvee$-lifting of morphisms. We have that
\[ \eql_i \in \Mor(\C F_i, \underset{\to} \Lim \C F_i) \TOT \forall_{\Theta \in 
\prod_{i \in I}^{\lt}F_i}\big(\eql_0 f_{\Theta} \circ \eql_i \in F_i\big). \]
If $x \in \lambda_0 (i)$, then 
$\big(\eql_0 f_{\Theta} \circ \eql_i\big)(x) := \eql_0 f_{\Theta}\big(\eql_0 (i, x)\big)
:= f_{\Theta} (i, x) := \Theta_i(x)$
hence $\eql_0 f_{\Theta} \circ \eql_i := \Theta_i \in F_i$. For the proof of (ii), if $x \in \lambda_0 (i)$, then 
\begin{align*}
 \eql_j (\lambda_{ij}^{\lt}(x)) =_{\underset{\to} \Lim \lambda_0 (i)} \eql_i (x) & : \TOT
 \eql_0 \big(j, \lambda_{ij}^{\lt}(x)\big) =_{\underset{\to} \Lim \lambda_0 (i)} \eql_0 (i, x)\\
 & \TOT \big(j, \lambda_{ij}^{\lt}(x)\big) =_{\sum_{i \in I}^{\lt}\lambda_0 (i)} (i, x)\\
 & : \TOT \exists_{k \in I} \big(i \lt k \ \& \ j \lt k \ \& \ \lambda_{ik}^{\lt}(x) =_{\lambda_0 (k)} 
 \lambda_{jk}^{\lt}(\lambda_{ij}^{\lt}(x))\big),
\end{align*}
which holds, since if $k \in I$ with $j \lt k$, the equality  
$ \lambda_{ik}^{\lt}(x) =_{\lambda_0 (k)} \lambda_{jk}^{\lt}(\lambda_{ij}^{\lt}(x))$ holds by
the definition of a direct family of sets, and by the definition of a directed set such a $k$ always exists.  To prove  
(iii) let the operation $h \colon \underset{\to} \Lim \lambda_0(i) \sto Y$, defined by 
$h\big(\eql_0(i, x)\big) := \varepsilon_i(x)$, for every $\eql_0(i, x) \in \underset{\to} \Lim \lambda_0(i)$. 
First we show that $h$ is a function. Let 
$$ \eql_0 (i, x) =_{\underset{\to} \Lim \lambda_0 (i)} \eql_0 (j, y) \TOT 
\exists_{k \in I}\big(i, j \lt k \ \& \ \lambda_{ik}^{\lt}(x) =_{\lambda_0(k)}
  \lambda_{jk}^{\lt}(y)\big).$$
By the supposed commutativity of the following diagrams
\begin{center}
\begin{tikzpicture}

\node (E) at (0,0) {$Y$};
\node[below=of E] (F) {};
\node [right=of F] (B) {$\lambda_0(k)$};
\node [left=of F] (C) {$\lambda_0(i)$};
\node [right=of B] (D) {$\lambda_0(j)$};
\node[right=of D] (K) {};
\node[right=of K] (L) {$\lambda_0(k)$};
\node[above=of K] (M) {$Y$};

\draw[->] (B)--(E) node [midway,right] {$ \ \varepsilon_k$};
\draw[->] (C)--(E) node [midway,left] {$\varepsilon_i \ $};
\draw[->] (C)--(B) node [midway,below] {$\lambda_{ik}^{\lt}$};

\draw[->] (L)--(M) node [midway,right] {$ \ \varepsilon_k$};
\draw[->] (D)--(M) node [midway,left] {$\varepsilon_j \ $};
\draw[->] (D)--(L) node [midway,below] {$\lambda_{jk}^{\lt}$};

\end{tikzpicture}
\end{center}
we get
$ h\big(\eql_0(i, x)\big) := \varepsilon_i(x) = \varepsilon_k\big(\lambda_{ik}^{\lt}(x)\big) =
\varepsilon_k\big(\lambda_{jk}^{\lt}(y)\big) = \varepsilon_j(y) := h\big(\eql_0(j, y)\big)$.
Next we show that $h$ is a Bishop morphism. By definition
$h \in \Mor(\underset{\to} \Lim \C F_i, \C G) \TOT \forall_{g \in G}\big(g \circ h \in 
\underset{\to} \Lim F_i\big)$.
If $g \in G$, we show that the dependent operation 
$\Theta_g : \bigcurlywedge_{i \in I}F_i $, defined by $\Theta_g (i) := g \circ \varepsilon_i$, for every $i \in I$,
is well-defined, since $\varepsilon_i \in \Mor(\C F_i, \C G)$, and that $\Theta_g \in \prod_{i \in I}^{\mt}F_i$. 
To prove the latter, if $i \lt k$, we show that $\Theta_g (i) = \Theta_g (k) 
\circ \lambda_{ik}^{\lt}$. By the commutativity of the above left diagram we have that
$\Theta_g (k) \circ \lambda_{ik}^{\lt} := (g \circ \varepsilon_k) \circ \lambda_{ik}^{\lt} :=
g \circ (\varepsilon_k \circ \lambda_{ik}^{\lt}) = g \circ \varepsilon_i := \Theta_g (i)$, 
Hence $f_{\Theta_g} \in \underset{\to} \Lim F_i$. Since
$(g \circ h)\big(\eql_0(i, x)\big) := g(\varepsilon_i(x)) := (g \circ \varepsilon_i)(x) :=
[\Theta_g (i)](x) := f_{\Theta_g}\big((\eql_0(i, x)\big)$, 
we get $g \circ h := \eql_0 f_{\Theta_g} \in \underset{\to} \Lim F_i$. The uniqueness of $h$, and the commutativity
of the diagram in property (iii) follow immediately. 
\end{proof}

The uniqueness of $\underset{\to} \Lim \lambda_0 (i)$, up to Bishop isomorphism, is shown easily from 
its universal property.
Note that if $i, j \in I$, $x \in \lambda_0 (i)$ and $y \in \lambda_0 (j)$, we have that
\begin{align*}
 \eql_i (x) =_{\underset{\to} \Lim \lambda_0 (i)} \eql_j (y) & : \TOT 
 \eql_0 (i, x) =_{\underset{\to} \Lim \lambda_0 (i)} \eql_0 (j, y)\\
 & \TOT (i, x) =_{\sum_{i \in I}^{\lt} \lambda_0(i)} (j, y)\\
 & : \TOT \exists_{k \in I}\big(i \lt k \ \& \ j \lt k \ \& \ \lambda_{ik}^{\lt}(x) =_{\lambda_0 (k)} 
\lambda_{jk}^{\lt}(y)\big).
\end{align*}

\begin{defi}\label{def: representative}
Let $S^{\lt} := (\lambda_0, \lambda_1^{\lt} ; \phi_0^{\Lambda^{\lt}}, \phi_1^{\Lambda^{\lt}})$ be 
a direct spectrum over $(I, \lt)$. 
If $i \in I$, an element $x$ of $\lambda_0 (i)$ is a
\textit{representative} of $\eql_0(z) \in \underset{\to} \Lim \lambda_0 (i)$,
if ${\eql}_i (x) =_{\mathsmaller{\underset{\to} \Lim \lambda_0 (i)}} \eql_0(z)$.
\end{defi}

Although an element $\eql_0(z) \in \underset{\to} \Lim \lambda_0 (i)$ may not have a representative
in every $\lambda_0 (i)$, it surely has one at some $\lambda_0 (i)$. Actually, the following holds. 

\begin{prop}\label{prp: representative1}
For every $n \geq 1$ and every $\eql_0(z_1), \ldots, \eql_0(z_n) \in \underset{\to}
\Lim \lambda_0 (i)$
there are $i \in I$ and $x_1, \ldots, x_n \in \lambda_0 (i)$ such that $x_l$ represents $\eql_0 (z_l)$, for
every $l \in \{1, \ldots, n\}$.
\end{prop}

\begin{proof}
The proof is by induction on $\Nat^+$. We present only the case $n := 2$. Let $z := (j, y), z{'} := 
(j{'}, y{'}) \in \sum_{i \in I}^{\lt} \lambda_0(i)$, and $k \in I$ with $j \lt k$ and $j{'} 
\lt k$. By definition we have that $\lambda_{jk}^{\lt}(y) \in \lambda_0 (k)$ and 
$\lambda_{j{'}k}^{\lt}(y{'}) \in \lambda_0 (k)$. We show that $\lambda_{jk}^{\lt}(y)$ represents $\eql_0(z)$
and $\lambda_{j{'}k}^{\lt}(y{'})$ represents $\eql_0(z{'})$. By the equivalences right before 
Definition~\ref{def: representative} for the first representation 
we need to show that
\[ {\eql}_k \big(\lambda_{jk}^{\lt}(y)\big) =_{\underset{\to} \Lim \lambda_0 (i)} {\eql}_j (y) 
\TOT \exists_{k{'} \in I}\big(k \lt k{'} \ \& \ j \lt k{'} \ \& \ 
\lambda_{kk{'}}^{\lt}(\lambda_{jk}^{\lt}(y)) =_{\lambda_0 (k{'})} \lambda_{jk{'}}^{\lt}(y)\big). \]
By the composition of the transport maps it suffices to take any $k{'} \in I$ with $k \lt k{'} \ \& \ j \lt k{'}$,
and for the second representation it suffices to take any $k{''} \in I$ with
$k \lt k{''} \ \& \ j{'} \lt k{''}$.
\end{proof}

\begin{thm}\label{thm: directspectrummap1}
Let $S(\Lambda^{\lt}) := (\lambda_0, \lambda_1^{\lt}, \phi_0^{\Lambda^{\lt}}, \phi_1^{\Lambda^{\lt}})
\in \Spec(I, \lt_I)$ with Bishop spaces $(\C F_i)_{i \in I}$ and Bishop morphisms $(\lambda_{ij}^{\lt})_{(i,j) \in D^{\lt}(I)}$, 
$S(M^{\lt}) := (\mu_0, \mu_1^{\lt}, \phi_0^{M^{\lt}}, \phi_1^{M^{\lt}}) \in \Spec(I, \lt_I)$ 
with Bishop spaces $(\C G_i)_{i \in I}$
and Bishop morphisms $(\mu_{ij}^{\lt})_{(i,j) \in D^{\lt}(I)}$, and $\Psi \colon S(\Lambda^{\lt}) \To S(M^{\lt})$. \\[1mm]
\normalfont (i)
\itshape There is a unique function $\Psi_{\to} \colon \underset{\to} \Lim \lambda_0 (i) \to 
\underset{\to} \Lim \mu_0 (i)$ such that, for every $i \in I$, the following diagram commutes
\begin{center}
\begin{tikzpicture}

\node (E) at (0,0) {$\underset{\to} \Lim \lambda_0 (i)$};
\node[right=of E] (F) {$\underset{\to} \Lim \mu_0 (i)$.};
\node[above=of F] (A) {$\mu_0(i)$};
\node [above=of E] (D) {$\lambda_0(i)$};

\draw[dashed,->] (E)--(F) node [midway,below]{$\Psi_{\to}$};
\draw[->] (D)--(A) node [midway,above] {$\Psi_i$};
\draw[->] (D)--(E) node [midway,left] {$\eql_i^{\Lambda^{\lt}}$};
\draw[->] (A)--(F) node [midway,right] {$\eql_i^{M^{\lt}}$};

\end{tikzpicture}
\end{center}
\normalfont (ii)
\itshape If $\Psi$ is continuous, then 
$\Psi_{\to} \in \Mor(\underset{\to} \Lim \C F_i, \underset{\to} \Lim \C G_i)$.\\[1mm]
\normalfont (iii)
\itshape If $\Psi_i$ is an embedding for every $i \in I$,
then $\Psi_{\to}$ is an embedding.
\end{thm}

\begin{proof}
(i) The following well-defined operation $\Psi_{\to} \colon \underset{\to} \Lim \lambda_0 (i) \sto 
\underset{\to} \Lim \mu_0 (i)$, given by 
\[ \Psi_{\to} \big(\eql_0^{\Lambda^{\lt}}(i, x)\big) := \eql_0^{M^{\lt}}(i, \Psi_i (x));  \ \ \ 
\ \eql_0^{\Lambda^{\lt}}(i, x) \in 
\underset{\to} \Lim \lambda_0 (i) \]
is a function,  
since, if  
$
 \eql_0^{\Lambda^{\lt}} (i, x) =_{\mathsmaller{\underset{\to} \Lim \lambda_0 (i)}} \eql_0^{\Lambda^{\lt}}(j, y)   \TOT 
 (i, x)=_{\mathsmaller{\sum_{i \in I}^{\lt} \lambda_0(i)}} (j, y)$, which is equivalent to
 $\exists_{k \in I}\big(i \lt k \ \& \ j \lt k \ \& \ \lambda_{ik}^{\lt}(x) 
 =_{\mathsmaller{\lambda_0 (k)}} \lambda_{jk}^{\lt}(y)\big)$, 
we show that
\begin{align*}
 \Psi_{\to} \big(\eql_0^{\Lambda^{\lt}}(i, x)\big) =_{\mathsmaller{\underset{\to} \Lim \mu_0 (i)}} 
 \Psi_{\to} \big(\eql_0^{\Lambda^{\lt}}(j, y)\big) & : \TOT \eql_0^{M^{\lt}}\big(i, \Psi_i (x)\big)
 =_{\mathsmaller{\underset{\to} \Lim \mu_0 (i)}}  \eql_0^{M^{\lt}}\big(j, \Psi_j (y)\big)\\
 & \TOT (i, \Psi_i (x)) =_{\mathsmaller{\sum_{i \in I}^{\lt}\mu_0 (i)}} (j, \Psi_j (y))\\
 & : \TOT \exists_{i \in I}\big(i, j \lt k \ \& \ \mu_{ik}^{\lt}(\Psi_i (x)) 
 =_{\mathsmaller{\mu_0 (k)}} \mu_{jk}^{\lt}(\Psi_j (y)\big).
\end{align*}
By the commutativity of the following diagrams,  and since $\Psi_k$ is a function, 
\begin{center}
\begin{tikzpicture}

\node (E) at (0,0) {$\mu_0(i)$};
\node[right=of E] (F) {$\mu_0(k)$};
\node[above=of F] (A) {$\lambda_0(k)$};
\node [above=of E] (D) {$\lambda_0(i)$};
\node [right=of F] (G) {$\mu_0 (j)$};
\node [above=of G] (H) {$\lambda_0(j)$};
\node [right=of G] (K) {$\mu_0 (k)$,};
\node [above=of K] (L) {$\lambda_0(k)$};

\draw[->] (E)--(F) node [midway,below]{$\mu_{ik}^{\lt}$};
\draw[->] (D)--(A) node [midway,above] {$\lambda_{ik}^{\lt}$};
\draw[->] (D)--(E) node [midway,left] {$\Psi_i$};
\draw[->] (A)--(F) node [midway,right] {$\Psi_k$};
\draw[->] (G)--(K) node [midway,below] {$\mu_{jk}^{\lt}$};
\draw[->] (H)--(L) node [midway,above] {$\lambda_{jk}^{\lt}$};
\draw[->] (H)--(G) node [midway,left] {$\Psi_j$};
\draw[->] (L)--(K) node [midway,right] {$\Psi_k$};

\end{tikzpicture}
\end{center}
we get
$\mu_{ik}^{\lt}\big(\Psi_i(x)\big) =_{\mathsmaller{\mu_0 (k)}} \ \Psi_k \big(\lambda_{ik}^{\lt}(x)\big)
 =_{\mathsmaller{\mu_0 (k)}} \ \Psi_k \big(\lambda_{jk}^{\lt}(y)\big)
 =_{\mathsmaller{\mu_0 (k)}} \ \mu_{jk}^{\lt}\big(\Psi_j (y)\big)
$.\\
(ii) By the $\bigvee$-lifting of morphisms it suffices to show that 
$\forall_{H \in \prod_{i \in I}^{\mt}G_i}\big((\eql_0^{M^{\lt}}g_H)
\circ \Psi_{\to} \in \underset{\to} \Lim F_i \big).$
By Definition~\ref{def: Limtop} we have that
\[ \big((\eql_0^{M^{\lt}}g_H) \circ  \Psi_{\to}\big)\big(\eql_0^{\Lambda^{\lt}}(i, x)\big) 
:= \big(\eql_0^{\Lambda^{\lt}}g_H\big)\big(\eql_0^{M^{\lt}}(i, \Psi_i (x)) \big)
:= g_H(i, \Psi_i(x)) \]
\[ := H_i (\Psi_i(x))
=: (H_i \circ \Psi_i)(x)
 := H_i^* (x)
:= f_{H^*}(i, x)
:= \big(\eql_0^{\Lambda^{\lt}}f_{H^*}\big)\big(\eql_0^{\Lambda^{\lt}}(i, x)\big), \]
where $H^* \in \prod_{i \in I}^{\mt}F_i$ is defined in Lemma~\ref{lem: preorderdependent}, and
$(\eql_0^{M^{\lt}}g_{H^*}) \circ \Psi_{\to} := \eql_0^{\Lambda^{\lt}}f_{H^*} \in \underset{\to} \Lim F_i$. \\
(iii) If $\Psi_{\to} \big(\eql_0^{\Lambda^{\lt}}(i, x) =_{\mathsmaller{\underset{\to} \Lim \mu_0 (i)}} 
 \Psi_{\to} \big(\eql_0^{\Lambda^{\lt}}(j, y)\big)$ i.e., 
$\mu_{ik}^{\lt}(\Psi_i (x)) =_{\mathsmaller{\mu_0(k)}} \mu_{jk}^{\lt}(\Psi_j(y))\big)$,
for some $k \in I$ with $i, j \lt k$, by the proof of case (ii) we get  
$\Psi_k \big(\lambda_{ik}^{\lt}(x)\big) =_{\mathsmaller{\mu_0 (k)}}  
\Psi_k \big(\lambda_{jk}^{\lt}(y)\big)$, and since
$\Psi_k$ is an embedding, we conclude that $\lambda_{ik}^{\lt}(x) =_{\mathsmaller{\lambda_0 (k)}} 
\lambda_{jk}^{\lt}(y)$ i.e., $(i, x) =_{\mathsmaller{\sum_{i \in I}^{\lt}\lambda_0(i)}} (j, y)$.
\end{proof}


\begin{prop}\label{prp: transitivity}
Let $S(\Lambda^{\lt}) := (\lambda_0, \lambda_1^{\lt}, \phi_0^{\Lambda^{\lt}}, \phi_1^{\Lambda^{\lt}}) \in 
\Spec(I, \lt_I)$ with Bishop spaces $(F_i)_{i \in I}$ and morphisms $(\lambda_{ij}^{\lt})_{(i,j) \in D^{\lt}(I)}$,
$S(M^{\lt}) := (\mu_0, \mu_1^{\lt}, \phi_0^{M^{\lt}}, \phi_1^{M^{\lt}}) \in \Spec(I, \lt_I)$ with 
Bishop spaces $(G_i)_{i \in I}$
and morphisms $(\mu_{ij}^{\lt})_{(i,j) \in D^{\lt}(I)}$, and let
$S(N^{\lt}) := (\nu_0, \nu_1^{\lt}, \phi_0^{N^{\lt}}, \phi_1^{N^{\lt}}) \in \Spec(I, \lt_I)$ with 
Bishop spaces $(H_i)_{i \in I}$
and morphisms $(\nu_{ij}^{\lt})_{(i,j) \in D^{\lt}(I)}$. If 
$\Psi \colon S(\Lambda^{\lt}) \To S(M^{\lt})$ and $\Xi \colon S(M^{\lt}) \To S(N^{\lt})$, then 
$(\Xi \circ \Psi)_{\to} :=  \Xi_{\to} \circ \Psi_{\to}$
\begin{center}
\begin{tikzpicture}

\node (E) at (0,0) {$\lambda_0(i)$};
\node[right=of E] (H) {};
\node[right=of H] (F) {$\mu_0(i)$};
\node[below=of E] (A) {$\underset{\to} \Lim \lambda_0 (i)$};
\node[below=of F] (B) {$\underset{\to} \Lim \mu_0 (i)$};
\node[right=of F] (K) {};
\node[right=of K] (G) {$\nu_0(i)$};
\node [below=of G] (C) {$\underset{\to} \Lim \nu_0 (i)$.};

\draw[->] (E)--(F) node [midway,above] {$ \Psi_i$};
\draw[->] (F)--(G) node [midway,above] {$\Xi_i$};
\draw[->] (B)--(C) node [midway,below] {$\Xi_{\to} \ \ $};
\draw[->] (A)--(B) node [midway,below] {$ \ \ \Psi_{\to}$};
\draw[->] (E)--(A) node [midway,left] {$\eql_i^{\Lambda^\lt}$};
\draw[->] (F) to node [midway,left] {$\eql_i^{M^\lt}$} (B);
\draw[->] (G)--(C) node [midway,right] {$\eql_i^{N^\lt}$};
\draw[->,bend right] (A) to node [midway,below] {$(\Xi \circ \Psi)_{\to}$} (C) ;
\draw[->,bend left] (E) to node [midway,above] {$(\Xi \circ \Psi)_i$} (G) ;

\end{tikzpicture}
\end{center}

\end{prop}

\begin{defi}\label{def: directspectrummap2}
Let  $S(\Lambda^{\lt}) := (\lambda_0, \lambda_1^{\lt}, \phi_0^{\Lambda^{\lt}}, \phi_1^{\Lambda^{\lt}}) 
\in \Spec(I, \lt_I)$
and $(J, e, \cof_J) \subseteq^{\cof} I$, a cofinal subset of $I$ with modulus of cofinality $e \colon J \eto I$.
The relative spectrum of 
$S(\Lambda^{\lt})$ to $J$\index{relative direct spectrum} is the $e$-subfamily 
$S(\Lambda^{\lt}) \circ e := \big(\lambda_0 \circ e, \lambda_1 \circ e, \phi_0^{\Lambda^{\lt}} \circ e, 
\phi_1^{\Lambda^{\lt}} \circ e \big)$ of $S(\Lambda^{\lt})$, where $\Phi^{\Lambda^{\lt}} \circ e 
:= \big(\phi_0^{\Lambda^{\lt}} \circ e, \phi_1^{\Lambda^{\lt}} \circ e \big)$ is the $e$-subfamily 
of $\Phi^{\Lambda^{\lt}}$ (see Definition~\ref{def: newfamsofsets1}).
\end{defi}

\begin{lem}\label{lem: cofinallemma}
Let $S(\Lambda^{\lt}) := (\lambda_0, \lambda_1^{\lt}, \phi_0^{\Lambda^{\lt}}, \phi_1^{\Lambda^{\lt}}) 
\in \Spec(I, \lt_I)$  with Bishop spaces $(\C F_i)_{i \in I}$ and Bishop morphisms $(\lambda_{ij}^{\lt})_{(i,j) \in D^{\lt}(I)}$,
$(J, e, \cof_J) \subseteq^{\cof} I$, and $S(\Lambda^{\lt}) \circ e := \big(\lambda_0 \circ e,
\lambda_1 \circ e, \phi_0^{\Lambda^{\lt}} \circ e, \phi_1^{\Lambda^{\lt}} \circ e \big)$ the relative spectrum of 
$S(\Lambda^{\lt})$ to $J$.\\[1mm]
\normalfont (i)
\itshape  If $\Theta \in \prod_{i \in I}^{\mt} F_i$, then $\Theta^J \in \prod_{j \in J}^{\mt} F_j$, where
for every $j \in J$ we define $\Theta^J_j := \Theta_{e(j)}.$\\[1mm]
\normalfont (ii)
\itshape If $H^J \in \prod_{j \in J}^{\mt} F_j$, then $H \in \prod_{i \in I}^{\mt} F_i$, 
where, for every $i \in I$, let $H_i := H^J_{\cof_J(i)} \circ \lambda_{i e(\cof_J(i))}^{\lt}$
\begin{center}
\begin{tikzpicture}

\node (E) at (0,0) {$\Real$.};
\node[above=of E] (F) {$\lambda_0 (e(\cof_J(i)))$};
\node[left=of F] (B) {};
\node[left=of B] (A) {$\lambda_0 (i)$};

\draw[->] (F)--(E) node [midway,right] {$H^J_{\cof_J(i)}$};
\draw[->] (A)--(E) node [midway,left] {$H_i \ \ \ $};
\draw[->] (A)--(F) node [midway,above] {$\lambda_{i e(\cof_J(i))}^{\lt}$};

\end{tikzpicture}
\end{center}
\end{lem}

\begin{proof}
(i) It suffices to show that if $j \lt j{'} :\TOT e(j) \lt e(j{'})$, then $\Theta^J_j = \Theta^J_{j{'}} \circ 
\lambda_{j j{'}}^{\lt}$. Since $\Theta \in \prod_{i \in I}^{\mt} F_i$ we have that
$\Theta^J_j := \Theta_{e(j)} = \Theta_{e(j{'})} \circ \lambda_{e(j) e(j{'})}^{\lt} =: 
\Theta^J_{j{'}} \circ \lambda_{j j{'}}^{\lt}$.\\[1mm]
(ii) By definition $H^J_{\cof_J(i)} \in F_{\cof_J(i)} := F_{e(\cof_J(i))}$, and since $i \lt e(\cof_J(i))$, 
we get $H_i \in \Mor(\C F_i, \C R) = \C F_i$ i.e., $H : \bigcurlywedge_{i \in I}F_i$. Next we show that if
$i \lt i{'}$, then $H_i = H_{i{'}} \circ \lambda_{i i{'}}^{\lt}$. By $(\Cf_3)$ and $(\Cf_2)$ we have that
\begin{equation}\label{eq: eq1}
i \lt i{'} \lt e(\cof_J(i{'})),
\end{equation}
and $i \lt i{'} \To \cof_J(i) \lt \cof_J(i{'}) :\TOT e(\cof_J(i)) \lt e(\cof_J(i{'})),$
hence we also get 
\begin{equation}\label{eq: eq2}
i \lt e(\cof_J(i)) \lt e(\cof_J(i{'})).
\end{equation}
Since  $H^J \in \prod_{j \in J}^{\mt} F_j$, we have that
\begin{align*}
H_{i{'}} \circ \lambda_{i i{'}}^{\lt} & := \big[H^J_{\cof_J(i{'})} \circ \lambda_{i{'} e(\cof_J(i{'}))}^{\lt}\big] \circ 
\lambda_{ii{'}}^{\lt}\\
& := H^J_{\cof_J(i{'})} \circ \big[\lambda_{i{'} e(\cof_J(i{'}))}^{\lt} \circ \lambda_{ii{'}}^{\lt}\big]\\
& \stackrel{(\ref{eq: eq1})} =  H^J_{\cof_J(i{'})} \circ \lambda_{i e(\cof_J(i{'}))}^{\lt}\\
& \stackrel{(\ref{eq: eq2})} = H^J_{\cof_J(i{'})} \circ \big[\lambda_{e(\cof_J(i)) e(\cof_J(i{'}))}^{\lt} \circ
\lambda_{i e(\cof_J(i))}^{\lt} \big]\\
& := \big[H^J_{\cof_J(i{'})} \circ \lambda_{e(\cof_J(i)) e(\cof_J(i{'}))}^{\lt}\big] \circ
\lambda_{i e(\cof_J(i))}^{\lt}\\
& := \big[H^J_{\cof_J(i{'})} \circ \lambda_{\cof_J(i) \cof_J(i{'})}^{\lt}\big] \circ
\lambda_{i e(\cof_J(i))}^{\lt}\\
& := H^J_{\cof_J(i)} \circ \lambda_{i e(\cof_J(i))}^{\lt}\\
& =: H_i.  \qedhere
\end{align*}
\end{proof}

\begin{thm}\label{thm: cofinal2}
Let $S(\Lambda^{\lt}) := (\lambda_0, \lambda_1^{\lt}, \phi_0^{\Lambda^{\lt}}, \phi_1^{\Lambda^{\lt}})
\in \Spec(I, \lt_I)$,
$(J, e, \cof_J) \subseteq^{\cof} I$, and $S(\Lambda^{\lt}) \circ e := \big(\lambda_0 \circ e, \lambda_1 
\circ e, \phi_0^{\Lambda^{\lt}} \circ e, \phi_1^{\Lambda^{\lt}} \circ e \big)$ the relative spectrum of 
$S(\Lambda^{\lt})$ to $J$.
Then 
\[ \underset{\to} \Lim \C F_j \simeq \underset{\to} \Lim \C F_i. \]
\end{thm}

\begin{proof}
We define the operation $\phi : \underset{\to} \Lim \lambda_0 (j) \sto \underset{\to} 
\Lim \lambda_0 (i)$ by
$\phi \big(\eql_0^{\Lambda^{\lt} \circ e}(j, y)\big) := \eql_0^{\Lambda^{\lt}}(e(j), y)$
\begin{center}
\begin{tikzpicture}

\node (E) at (0,0) {$\underset{\to} \Lim \lambda_0 (j)$};
\node[right=of E] (A) {};
\node[right=of A] (F) {$\underset{\to} \Lim \lambda_0 (i)$,};
\node [above=of A] (D) {$\lambda_0 (j)$};

\draw[->] (E)--(F) node [midway,below] {$\phi$};
\draw[->] (D)--(E) node [midway,left,near start] {$\eql_j^{\Lambda^{\lt} \circ e} \ $};
\draw[->] (D)--(F) node [midway,right,near start] {$ \ \eql_{e(j)}^{\Lambda^{\lt}}$};

\end{tikzpicture}
\end{center}
for every $\eql_0^{\Lambda^{\lt} \circ e}(j, y) \in \underset{\to} \Lim \lambda_0 (j)$, 
where, if $j \in J$ and $y \in \lambda_0 (j)$, we have that
\[ \eql_0^{\Lambda^{\lt} \circ e}(j, y) := \bigg\{(j{'}, y{'}) \in \sum_{j \in J}^{\lt}\lambda_0 (j) 
\mid (j{'}, y{'})
=_{\mathsmaller{\sum_{j \in J}^{\lt}\lambda_0 (j)}} (j, y) \bigg\}, \]
\[ \eql_0^{\Lambda^{\lt}}(e(j), y) := \bigg\{(i, x) \in \sum_{i \in I}^{\lt}\lambda_0 (i) \mid (i, x)
=_{\mathsmaller{\sum_{i \in I}^{\lt}\lambda_0 (i)}} (e(j), y) \bigg\}. \]
First we show that $\phi$ is a function.
 By definition we have that
\begin{align*}
\eql_0^{\Lambda^{\lt} \circ e}(j, y) =_{\mathsmaller{\underset{\to} \Lim \lambda_0 (j)}} 
\eql_0^{\Lambda^{\lt} \circ e}(j{'}, y{'}) & \TOT 
(j, y) =_{\mathsmaller{\sum_{j \in J}^{\lt}\lambda_0 (j)}} (j{'}, y{'}) \\
& \TOT \exists_{j{''} \in J}\big(j, j{'} \lt j{''} \ \& \ \lambda_{jj{''}}^{\lt} (y) 
=_{\mathsmaller{\lambda_0 (j{''})}} \lambda_{j{'}j{''}}^{\lt} (y{'})\big)  \ \ \ (1)
\end{align*}
\begin{align*}
\eql_0^{\Lambda^{\lt}}(e(j), y) =_{\mathsmaller{\underset{\to} \Lim \lambda_0 (i)}} 
\eql_0^{\Lambda^{\lt}}(e(j{'}), y{'}) & \TOT (e(j), y) 
=_{\mathsmaller{\sum_{i \in I}^{\lt}\lambda_0 (i)}} (e(j{'}), y{'}) \\
& \TOT \exists_{k \in I}\big(e(j), e(j{'}) \lt k \ \& \ \lambda_{e(j)k}^{\lt} (y) =_{\mathsmaller{\lambda_0 (k)}}
\lambda_{e(j{'})k}^{\lt} (y{'})\big). \  (2)
\end{align*}
Taking $k := e(j{''})$, we see that (1) implies (2), and hence $\phi$ is a function. To show that $\phi$ is an embedding,
we show that (2) implies (1). Since $e(j), e(j{'}) \lt k \lt e(\cof_J(k))$, 
we get $j, j{'} \lt \cof_J(k) := j{''}$.
By the commutativity of the following diagrams 
\begin{center}
\begin{tikzpicture}

\node (E) at (0,0) {$\lambda_0(e(j))$};
\node[below=of E] (F) {$\lambda_0(k)$};
\node[right=of F] (H) {};
\node[right=of H] (G) {$\lambda_0(e(\cof_J(k)))$};
\node[right=of G] (A) {$\lambda_0(k)$};
\node[above=of A] (B) {$\lambda_0(e(j{'}))$};
\node[right=of A] (D) {};
\node[right=of D] (C) {$\lambda_0(e(\cof_J(k)))$};

\draw[->] (E)--(F) node [midway,left] {$\lambda_{e(j)k}^{\lt}$};
\draw[->] (E)--(G) node [midway,right, near start] {$\ \ \ \lambda_{e(j)e(\cof_J(k))}^{\lt}$};
\draw[->] (F)--(G) node [midway,below] {$\lambda_{ke(\cof_J(k))}^{\lt}$};
\draw[->] (B)--(A) node [midway,left] {$\lambda_{e(j{'})k}^{\lt}$};
\draw[->] (A)--(C) node [midway,below] {$\lambda_{ke(\cof_J(k))}^{\lt}$};
\draw[->] (B)--(C) node [midway,right, near start] {$\ \ \ \lambda_{e(j{'})e(\cof_J(k))}^{\lt}$};

\end{tikzpicture}
\end{center}
\begin{align*}
\lambda_{jj{''}}^{\lt} (y) & := \lambda_{e(j)e(\cof_J(k))}^{\lt}(y)\\
& = \big[\lambda_{ke(\cof_J(k))} \circ \lambda_{e(j)k}^{\lt}\big] (y) \\  
& = \lambda_{ke(\cof_J(k))} \big(\lambda_{e(j)k}^{\lt}(y)\big)\\ 
& = \lambda_{ke(\cof_J(k))} \big(\lambda_{e(j{'})k}^{\lt}(y{'})\big)\\
& =: \big[\lambda_{ke(\cof_J(k))} \circ \lambda_{e(j{'})k}^{\lt}\big](y{'})\\
& = \lambda_{e(j{'})e(\cof_J(k))}^{\lt}(y{'})\\
& := \lambda_{j{'} j{''}}(y{'}).
\end{align*}
By the $\bigvee$-lifting of morphisms we have that 
\[ \phi \in \Mor(\underset{\to} \Lim \C F_j, \underset{\to} \Lim \C F_i) : \TOT
\forall_{\mathsmaller{\Theta \in \prod_{i \in I}^{\mt}F_i}}\big(\eql_0 f_{\Theta}
\circ \phi \in \underset{\to} \Lim F_j\big). \]
If $\Theta \in \prod_{i \in I}^{\mt}F_i$, we have that 
\[  (\eql_0 f_{\Theta} \circ \phi) \big(\eql_0^{\Lambda^{\lt} \circ e}(j, y)\big) := 
 (\eql_0 f_{\Theta})\big(\eql_0^{\Lambda^{\lt}}(e(j), y)\big) \]
\[ := \Theta_{e(j)}(y) := \Theta^J_j(y) := (\eql_0 f_{\Theta^{J}})\big(\eql_0^{\Lambda^{\lt} \circ e}(j, y)\big), \]
where $\Theta^{J} \in \prod_{j \in J}^{\lt} F_j$ is defined in Lemma~\ref{lem: cofinallemma}(i). 
Hence, $\eql_0 f_{\Theta} \circ \phi = \eql_0 f_{\Theta^{J}} \in \underset{\to} \Lim F_j$.
Next we show that $\phi$ is a surjection. If $\eql_0^{\Lambda^{\lt}}(i, x) \in 
\underset{\to} \Lim \lambda_0 (i)$, we find $\eql_0^{\Lambda^{\lt} \circ e}(j, y) \in 
\underset{\to} \Lim \lambda_0 (j)$ such that 
$\phi \big(\eql_0^{\Lambda^{\lt} \circ e}(j, y)\big) := \eql_0^{\Lambda^{\lt}}(e(j), y) 
=_{\mathsmaller{\underset{\to} \Lim \lambda_0 (i)}} \eql_0^{\Lambda^{\lt}}(i, x)$ i.e., 
we find $k \in I$ such that
$i, e(j) \lt k$ and $\lambda_{ik}^{\lt}(x) =_{\mathsmaller{\lambda_0 (k)}} \lambda_{e(j)k}^{\lt}(y)$.
If $j := \cof_J(i)$, by $(\Cf_3)$ we have that $i \lt e(\cof_J(i))$,
and by the reflexivity of $\lt$ we have that $e(\cof_J(i)) \lt e(\cof_J(i)) =: k$. If 
$y := \lambda_{i e(\cof_J(i))}{^\lt}(x) \in \lambda_0(e(\cof_J(i))) := (\lambda_0 \circ e)(\cof_J(i))$, then
\[ \lambda_{e(\cof_J(i))e(\cof_J(i))}^{\lt}\big(\lambda_{ie(\cof_J(i))}^{\lt}(x)\big) 
=_{\mathsmaller{\lambda_0 (k)}} \lambda_{ie(\cof_J(i))}^{\lt}(x). \]
We could use the $\bigvee$-lifting of openness to show that $\phi$ is an open morphism, and hence a 
Bishop isomorphism, but it is better to define directly its inverse Bishop morphism using the previous
proof of the surjectivity
 of $\phi$. Let the operation $\theta \colon \underset{\to} \Lim \lambda_0(i) \sto \underset{\to} \Lim \lambda_0(j)$,
 defined by 
\[ \theta \big(\eql_0^{\Lambda^{\lt}}(i, x)\big) := \eql_0^{\Lambda^{\lt} \circ e}\big(\cof_J(i), 
\lambda_{ie(\cof_J(i))}(x)\big);  \ \ \ \
\eql_0^{\Lambda^{\lt}}(i, x) \in \underset{\to} \Lim \lambda_0(i). \]
First we show that $\theta$ is a function. We have that 
\[ \eql_0^{\Lambda^{\lt}}(i, x)  =_{\mathsmaller{\underset{\to} \Lim \lambda_0(i)}} 
\eql_0^{\Lambda^{\lt}}(i{'}, x{'}) \TOT
\exists_{k \in I}\big(i \lt k \ \& \ i{'} \lt k \ \& \ \lambda_{ik}^{\lt}(x) =_{\mathsmaller{\lambda_0(k)}}
\lambda_{i{'}k}^{\lt}(x{'})\big), \]
\[ \eql_0^{\Lambda^{\lt} \circ e}\big(\cof_J(i), \lambda_{ie(\cof_J(i))}(x)\big) =_{\mathsmaller{\underset{\to}
\Lim \lambda_0(j)}} \eql_0^{\Lambda^{\lt} \circ e}\big(\cof_J(i{'}), \lambda_{i{'}e(\cof_J(i{'}))}(x{'})\big) \TOT \]
\[ \exists_{j{'} \in J}\bigg(\cof_J(i) \lt j{'} \ \& \ \cof_J(i{'}) \lt j{'} \ \&  \] 
\[ \lambda^{\lt}_{e(\cof_J(i))e(j{'})}\big(\lambda^{\lt}_{ie(\cof_J(i))}(x)\big) =_{\mathsmaller{\lambda_0(e(j{'})}}
\lambda^{\lt}_{e(\cof_J(i{'}))e(j{'})}\big(\lambda^{\lt}_{i{'}e(\cof_J(i{'}))}(x{'})\bigg). \]
If $j{'} := \cof_J(k)$, then by $(\Cf_2)$ we get $\cof_J(i) \lt j{'}$ and $\cof_J(i{'}) \lt j{'}$. 
Next we show that
\[ \lambda^{\lt}_{e(\cof_J(i))e(\cof_J(k))}\big(\lambda^{\lt}_{ie(\cof_J(i))}(x)\big) 
=_{\mathsmaller{\lambda_0(e(\cof_J(k))}}
\lambda^{\lt}_{e(\cof_J(i{'}))e(\cof_J(k))}\big(\lambda^{\lt}_{i{'}e(\cof_J(i{'}))}(x{'})\big). \]
By the following order relations, the two terms of the required equality are written as  
\begin{center}
\begin{tikzpicture}

\node (E) at (0,0) {$e(\cof_J(k))$};
\node[above=of E] (F) {$k$};
\node[right=of F] (A) {$e(\cof_J(i{'}))$};
\node[left=of F] (B) {$e(\cof_J(i))$};
\node[above=of B] (C) {$i$};
\node[above=of A] (D) {$i{'}$};
\

\draw[->] (F)--(E) node [midway,below] {};
\draw[->] (B)--(E) node [midway,left] {};
\draw[->] (A)--(E) node [midway,below] {};
\draw[->] (C)--(B) node [midway,right] {};
\draw[->] (D)--(A) node [midway,right] {};
\draw[->] (D)--(F) node [midway,right] {};
\draw[->] (C)--(F) node [midway,right] {};

\end{tikzpicture}
\end{center}
$\lambda^{\lt}_{ie(\cof_J(k))}(x) = \lambda^{\lt}_{ke(\cof_J(k))}\big(\lambda^{\lt}_{ik}(x)\big)$, 
and $\lambda^{\lt}_{i{'}e(\cof_J(k))}(x{'}) = \lambda^{\lt}_{ke(\cof_J(k))}\big(\lambda^{\lt}_{i{'}k}(x{'})\big)$.
By the equality 
$\lambda_{ik}^{\lt}(x) =_{\mathsmaller{\lambda_0(k)}} \lambda_{i{'}k}^{\lt}(x{'})$ we get the required equality.
Next we show that
\[ \theta \in \Mor\big(\underset{\to} \Lim \C F_i, \underset{\to} \Lim \C F_j\big) \TOT \forall_{H^J 
\in \prod_{j \in J}^{\mt}F_j}\bigg(\eql_0 f_{H^J} \circ \theta \in \bigvee_{\Theta \in 
\prod_{i \in I}^{\mt}F_i}\eql_0 f_{\Theta}\bigg). \]
If we fix $H^J \in \prod_{j \in J}^{\mt} F_j$, and if $H \in \prod_{i \in I}^{\mt}F_i$, defined 
in Lemma~\ref{lem: cofinallemma}(ii), then 
\begin{align*}
\big(\eql_0 f_{H^J} \circ \theta \big)\big(\eql_0^{\Lambda^{\lt}}(i,x)\big) & 
:= \eql_0 f_{H^J}\bigg(\eql_0^{\Lambda^{\lt} \circ e}\big(\cof_J(i), \lambda_{ie(\cof_J(i))}(x)\big)\bigg)\\
& := f_{H^J}\big(\cof_J(i), \lambda_{ie(\cof_J(i))}(x)\big)\\
& := H^J_{\cof_J(i)}\big(\lambda_{ie(\cof_J(i))}(x)\big)\\
& := \big[H^J_{\cof_J(i)} \circ \lambda_{ie(\cof_J(i))}\big](x)\\
& := H_i(x)\\
& := f_H(i,x)\\
& := \eql_0 f_H \big(\eql_0^{\Lambda^{\lt}}(i,x)\big),
\end{align*}
hence $\eql_0 f_{H^J} \circ \theta := \eql_0 f_H \in \underset{\to} \Lim F_i$. Next we show that
$\phi$ and $\theta$ are inverse to each other.
\begin{align*}
\phi\big(\theta\big(\eql_0^{\Lambda^{\lt}}(i,x)\big)\big) & := \phi\big(\eql_0^{\Lambda^{\lt} 
\circ e}\big(\cof_J(i), \lambda_{ie(\cof_J(i))}(x)\big) \\
& := \eql_0^{\Lambda^{\lt}}\big(e(\cof_J(i)), \lambda_{ie(\cof_J(i))}(x)\big),
\end{align*}
which is equal to $\eql_0^{\Lambda^{\lt}}(i,x)$ if and only if there is $k \in I$ with $i \lt k$ and 
$e(\cof_J(i)) \lt k$ and 
\[ \lambda_{ik}^{\lt}(x) =_{\mathsmaller{\lambda_0(k)}} 
\lambda^{\lt}_{e(\cof_J(i))k}\big(\lambda_{ie(\cof_J(i))}(x)\big), \]
which holds for every such $k \in I$. As by $(\Cf_3)$ we have that $i \lt e(\cof_J(i))$, the existence 
of such a $k \in I$ follows trivially. Similarly,
\begin{align*}
\theta\big(\phi\big(\eql_0^{\Lambda^{\lt} \circ e}(j,y)\big)\big) & := 
\theta\big(\eql_0^{\Lambda^{\lt}}\big(e(j), y)\big)\\
& := \eql_0^{\Lambda^{\lt} \circ e}\big(\cof_J(e(j)), \lambda_{e(j)e(\cof_J(e(j)))}(y)\big),
\end{align*}
which is equal to $\eql_0^{\Lambda^{\lt} \circ e}(j,y)$ if and only if there is $j{'} \in J$ 
with $j \lt j{'}$, $(\cof_J(e(j))) \lt j{'}$ and 
\[ \lambda_{e(j)e(j{'})}^{\lt}(y) =_{\mathsmaller{\lambda_0(e(j{'}))}} 
\lambda^{\lt}_{e(\cof_J(e(j)))e(j{'})}\big(\lambda_{e(j)e(\cof_J(e(j)))}(y)\big), \]
which holds for every such $j{'} \in J$. As by $(\Cf_1)$ we have that $j =_J \cof_J(e(j))$, the existence of such a 
$j{'} \in J$ follows trivially. 
\end{proof}

Notice that there is an alternative way for proving Theorem~\ref{thm: cofinal2} by 
appealing to the universal property (Proposition~\ref{prp: universaldirect}).
Next, we use for simplicity the same symbol for different orderings.

\begin{prop}\label{prp: prodspec}
If $(I, \lt), (J, \lt)$ are directed sets, $i \in I$ and $j \in J$, let
\[ (i, j) \lt (i{'}, j{'}) : \TOT i \lt i{'} \ \& \ j \lt j{'}. \]
If $(K, i_K, \cof_K) \subseteq^{\cof} I$ and $(L, i_L, \cof_L) \subseteq^{\cof} J$, let 
$i_{K \times L} : K \times L \hookrightarrow I \times J$ and $\cof_{K \times L} : I \times J \to K \times L$, 
defined, for every $k \in K$ and $l \in L$, by
\[ i_{K \times L}(k, l) := \big(i_K(k), i_L(l)\big) \ \ \ \& \ \ \ \cof_{K \times L}(i, j) := \big(\cof_K(i),
\cof_L(j)\big). \]
Let $\Lambda^{\lt} := (\lambda_0, \lambda_1^{\lt}) \in \Fam(I, \lt)$  
and $M^{\lt} := (\mu_0, \mu_1^{\lt}) \in \Fam(J, \lt)$. Let also 
$S(\Lambda^{\lt}) := (\lambda_0, \lambda_1^{\lt}, \phi_0^{\Lambda^{\lt}},  \phi_1^{\Lambda^{\lt}}) \in \Spec(I, \lt)$ 
with Bishop spaces $(\C F_i)_{i \in I}$ and morphisms $(\lambda_{ii{'}})_{(i, i{'})
\in \lt(I)}$,
and $S(M^{\lt}) := (\mu_0, \mu_1^{\lt}, \phi_0^{M^{\lt}},  \phi_1^{M^{\lt}}) \in \Spec(J, \lt)$ 
with Bishop spaces $(\C G_j)_{j \in J}$ and morphisms $(\mu_{jj{'}})_{(j, j{'}) \in \lt(J)}^{\lt}$.\\[1mm]
\normalfont (i)
\itshape $(I \times J, \lt)$ is a directed set, and $(K \times L, i_{K \times L}, \cof_{K \times L}) 
\subseteq^{\cof} I \times J$.\\[1mm]
\normalfont (ii)
\itshape The pair $\Lambda^{\lt} \times M^{\lt} := (\lambda_0 \times \mu_0 , (\lambda_1 \times \mu_1)^{\lt}) \in 
\Fam(I \times J, \lt)$, where  
\[ (\lambda_0 \times \mu_0)\big((i, j)\big) := \lambda_0 (i) \times \mu_0(j), \]
\[ (\lambda_1 \times \mu_1)^{\lt}\big((i, j), (i{'}, j{'})\big) := 
(\lambda_1 \times \mu_1)^{\lt}_{(i, j), (i{'}, j{'})}, \]
\[ (\lambda_1 \times \mu_1)^{\lt}_{(i, j), (i{'}, j{'})}\big((x, y)\big) := 
\big(\lambda_{ii{'}}^{\lt}(x), \mu_{jj{'}}^{\lt}(y)\big). \]
\normalfont (iii)
\itshape The structure 
$S(\Lambda^{\lt} \times M^{\lt}) := \big(\lambda_0 \times \mu_0, \lambda_1^{\lt} \times \mu_1^{\lt} ; 
\phi_0^{\Lambda^{\lt} \times M^{\lt}},  \phi_1^{\Lambda^{\lt} \times M^{\lt}}\big) \in \Spec(I \times J, \lt)$
with Bishop spaces $(F_i \times G_j)_{(i, j) \in I \times J}$ and Bishop morphisms 
$(\lambda_1 \times \mu_1)_{\big((i,j)(i{'},j{'})\big) \in D^{\lt}(I \times J)}$, where
\[ \phi_0^{\Lambda^{\lt} \times M^{\lt}}(i, j) := F_i \times G_j, \]
\[ \phi_1^{\Lambda^{\lt} \times M^{\lt}}\big((i, j), (i{'}, j{'})\big) := 
[(\lambda_1 \times \mu_1)_{(i,j)(i{'},j{'})}^{\lt}]^* : F_{i{'}} \times G_{j{'}} \to F_i \times G_j.\]
\end{prop}

\begin{proof}
(i) is immediate to show. For the proof of (ii) we have that 
$(\lambda_1 \times \mu_1)^{\lt}_{(i, j), (i, j)}\big((x, y)\big) := 
\big(\lambda_{ii}^{\lt}(x), \mu_{jj}^{\lt}(y)\big) := (x, y),$
and if $(i, j) \lt (i{'}, j{'}) \lt (i{''}, j{''})$, then the commutativity of the 
\begin{center}
\begin{tikzpicture}

\node (E) at (0,0) {$\lambda_0(i) \times \mu_0(j)$};
\node[below=of E] (F) {$\lambda_0(i{'}) \times \mu_0(j{'})$};
\node[right=of F] (H) {};
\node[right=of H] (G) {$\lambda_0(i{''}) \times \mu_0(j{''})$};

\draw[->] (E)--(F) node [midway,left] {$\mathsmaller{(\lambda_1 \times \mu_1)^{\lt}_{(i, j), (i{'}, j{'})}}$};
\draw[->] (E)--(G) node [midway,right, near start] {$\ \ \ \mathsmaller{(\lambda_1 \times \mu_1)^{\lt}_{(i, j), (i{''}, j{''})}}$};
\draw[->] (F)--(G) node [midway,below] {$\mathsmaller{(\lambda_1 \times \mu_1)^{\lt}_{(i{'}, j{'}), (i{''}, j{''})}}$};

\end{tikzpicture}
\end{center}
above diagram follows from the equalities $\lambda_{ii{''}}^{\lt} = \lambda_{i{'}i{''}}^{\lt} \circ \lambda_{ii{'}}^{\lt}$
and $\mu_{jj{''}}^{\lt} = \mu_{j{'}j{''}}^{\lt} \circ \mu_{jj{'}}^{\lt}$.\\
(iii) We show that $(\lambda_1 \times \mu_1)_{(i,j)(i{'},j{'})}^{\lt} \in \Mor(\C F_i \times \C G_j, 
\C F_{i{'}} \times \C G_{j{'}})$. By the $\bigvee$-lifting of morphisms it suffices to show that
$\forall_{f \in F_{i{'}}}\big((f \circ \pi_1) \circ  (\lambda_1 \times
\mu_1)_{(i,j)(i{'},j{'})}^{\lt} \in F_i \times G_j\big)$ and $\forall_{g \in G_{j{'}}}\big((g \circ \pi_2) \circ 
(\lambda_1 \times \mu_1)_{(i,j)(i{'},j{'})}^{\lt} \in F_i \times G_j \big).$
If $f \in F_{i{'}}$, then $(f \circ \pi_1) \circ  (\lambda_1 \times \mu_1)_{(i,j)(i{'},j{'})}^{\lt} := (f \circ 
\lambda_{ii{'}}^{\lt}) \circ \pi_1 \in F_i \times G_j$, as $f \circ \lambda_{ii{'}}^{\lt} \in F_i$ and
$[(f \circ \pi_1) \circ  (\lambda_1 \times \mu_1)_{(i,j)(i{'},j{'})}^{\lt}](x, y) := 
(f \circ \pi_1)\big(\lambda_{ii{'}}^{\lt}(x), \mu_{jj{'}}^{\lt}(y)\big) := f \big(\lambda_{ii{'}}^{\lt}(x)\big)
:= [(f \circ \lambda_{ii{'}}^{\lt}) \circ \pi_1](x, y)$.
If $g \in G_{j{'}}$, we get $(g\circ \pi_2) \circ  (\lambda_1 \times \mu_1)_{(i,j)(i{'},j{'})}^{\lt} := 
(g \circ \lambda_{jj{'}}^{\lt}) \circ \pi_2 \in F_i \times G_j$.
\end{proof}

\begin{lem}\label{lem: prodlemma}
Let $S(\Lambda^{\lt}) := (\lambda_0, \lambda_1^{\lt},
\phi_0^{\Lambda^{\lt}},  \phi_1^{\Lambda^{\lt}}) \in \Spec(I, \lt)$ 
with Bishop spaces $(\C F_i)_{i \in I}$ and Bishop morphisms
$(\lambda_{ii{'}})_{(i, i{'})\in D^{\lt}(I)}^{\lt}$,
$S(M^{\lt}) := (\mu_0, \mu_1^{\lt}, \phi_0^{M^{\lt}},  \phi_1^{M^{\lt}}) \in \Spec(J, \lt)$ 
with Bishop spaces $(\C G_j)_{j \in J}$ and Bishop morphisms $(\mu_{jj{'}})_{(j, j{'}) \in D^{\lt}(J)}$,
$\Theta \in \prod_{i \in I}^{\mt} F_i$ and $\Phi \in \prod_{j \in J}^{\mt} G_j$. Then 
\[ \Theta_1 \in \prod_{(i, j) \in I \times J}^{\mt} F_i \times G_j \ \ \ \& \ \ \ \ \Phi_2 \in
\prod_{(i, j) \in I \times J}^{\mt} F_i \times G_j, \]
\[ \Theta_1 (i, j) := \Theta_i \circ \pi_1 \in F_i \times G_j \ \ \ \& \ \ \
\Phi_2 (i, j) := \Phi_j \circ \pi_2 \in F_i \times G_j; \ \ \ \ (i, j) \in I \times J. \]
\end{lem}

\begin{proof}
We prove that $\Theta_1 \in \prod_{(i, j) \in I \times J}^{\mt} F_i \times G_j$, and for 
$\Phi_2$ we proceed similarly. If $(i, j) \lt (i{'}, j{'})$, we need to show that 
$\Theta_1 (i, j) = \Theta_1 (i{'}, j{'}) \circ (\lambda_1 \times \mu_1)^{\lt}_{(i, j), (i{'}, j{'})}$.
Since $\Theta \in \prod_{i \in I}^{\mt} F_i$, we have that 
$\Theta_i = \Theta_{i{'}} \circ \lambda_{i i{'}}^{\lt}$.
If $x \in \lambda_0(i)$ and $y \in \mu_0(j)$, we have that
\begin{align*}
\big[\Theta_1 (i{'}, j{'}) \circ (\lambda_1 \times \mu_1)^{\lt}_{(i, j), (i{'}, j{'})}](x, y) & := 
\big[\Theta_{i{'}} \circ \pi_1\big]\big(\lambda_{ii{'}}^{\lt}(x), \mu_{jj{'}}^{\lt}(y)\big)\\ 
& := \Theta_{i{'}}\big(\lambda_{ii{'}}^{\lt}(x)\big)\\
& =: \big[\big(\Theta_{i{'}} \circ \lambda_{ii{'}}^{\lt}\big) \circ \pi_1\big](x, y)\\
& := \big(\Theta_i \circ \pi_1\big)(x, y)\\
& := \big[\Theta_1(i, j)\big](x, y). \qedhere
\end{align*}
\end{proof}

\begin{prop}\label{prp: proddirect}
If $S(\Lambda^{\lt}) := (\lambda_0, \lambda_1^{\lt}, \phi_0^{\Lambda^{\lt}},  \phi_1^{\Lambda^{\lt}}) \in \Spec(I, \lt)$ 
with Bishop spaces $(\C F_i)_{i \in I}$ and Bishop morphisms $(\lambda_{ii{'}})_{(i, i{'})\in D^{\lt}(I)}^{\lt}$,
and $S(M^{\lt}) := (\mu_0, \mu_1^{\lt}, \phi_0^{M^{\lt}},  \phi_1^{M^{\lt}}) \in \Spec(J, \lt)$ 
with Bishop spaces $(\C G_j)_{j \in J}$ and Bishop morphisms $(\mu_{jj{'}}^{\lt})_{(j, j{'}) \in D^{\lt}(J)}$,
there is a bijection
\[\theta : \underset{\to} \Lim \big(\lambda_0(i) \times \mu_0 (j)\big) \to \underset{\to} \Lim 
\lambda_0(i) \times 
\underset{\to} \Lim \mu_0(j) \in \Mor \big(\underset{\to} \Lim (\C F_i \times \C G_j), \underset{\to} 
\Lim \C F_i \times \underset{\to} \Lim \C G_j\big). \]
\end{prop}

\begin{proof}
Let the operation  $\theta : \underset{\to} \Lim \big(\lambda_0(i) \times \mu_0 (j)\big) \sto 
\underset{\to} \Lim \lambda_0(i) \times \underset{\to} \Lim \mu_0(j)$ be defined by 
\[ \theta\big(\eql_0^{\Lambda^{\lt} \times M^{\lt}}\big((i, j), (x, y)\big)\big) := \big(\eql_0^{\Lambda^{\lt}}(i, x), 
\eql_0^{M^{\lt}}(j, y)\big). \]
First we show that $\theta$ is an embedding as follows:
\begin{align*}
& \  \eql_0^{\Lambda^{\lt} \times M^{\lt}}\big((i, j), (x, y)\big) = \eql_0^{\Lambda^{\lt} 
\times M^{\lt}}\big((i{'}, j{'}), (x{'}, y{'})\big) :\TOT \\
& : \TOT \exists_{(k, l) \in I \times J}\big((i, j), (i{'}, j{'}) \lt (k, l) \ \& \ 
(\lambda_1 \times \mu_1)_{(i,j)(k,l)}^{\lt}(x, y) = 
(\lambda_1 \times \mu_1)_{(i{'},j{'})(k, l)}^{\lt}(x{'}, y{'})\big)\\
& : \TOT \exists_{(k, l) \in I \times J}\big((i, j), (i{'}, j{'}) \lt (k, l) \ \& \ 
\big(\lambda_{ik}^{\lt}(x), \mu_{jl}^{\lt}(y)\big) = \big(\lambda_{i{'}k}^{\lt}(x{'}), 
\mu_{j{'}k}^{\lt}(y{'})\big)\\
& :\TOT \exists_{k \in I}\big(i, i{'} \lt k \ \& \ \lambda_{ik}^{\lt}(x) = \lambda_{i{'}k}^{\lt}(x{'})\big)
\ \& \
\exists_{l \in J}\big(j, j{'} \lt l \ \& \ \lambda_{jk}^{\lt}(y) = \lambda_{j{'}k}^{\lt}(y{'})\big)\\
& : \TOT \eql_0^{\Lambda^{\lt}}(i, x) = \eql_0^{\Lambda^{\lt}}(i{'}, x{'}) \ \ \& \ \ \eql_0^{M^{\lt}}(j, y) = 
\eql_0^{M^{\lt}}(j{'}, y{'})\\
& : \TOT \big(\eql_0^{\Lambda^{\lt}}(i, x), \eql_0^{M^{\lt}}(j, y)\big) = \big(\eql_0^{\Lambda^{\lt}}(i{'}, x{'}), 
\eql_0^{M^{\lt}}(j{'}, y{'})\big)\\
& : \TOT \theta\big(\eql_0^{\Lambda^{\lt} \times M^{\lt}}\big((i, j), (x, y)\big)\big) = 
\theta\big(\eql_0^{\Lambda^{\lt} \times M^{\lt}}\big((i{'}, j{'}), (x{'}, y{'})\big)\big).
\end{align*}
The fact that $\theta$ is a surjection 
is immediate to show. 
By definition of the direct limit and by the $\vee$-lifting of the product Bishop topology we have that
\[ \underset{\to} \Lim (\C F_i \times \C G_j) := \bigg(\underset{\to} \Lim \big(\lambda_0(i) \times 
\mu_0 (j)\big), \bigvee_{\Xi \in \prod_{(i, j) \in I \times J}^{\mt}F_i \times 
G_j}\eql_0 f_{\Xi}\bigg),\]
\[ \underset{\to} \Lim \C F_i \times \underset{\to} \Lim \C G_j := \bigg(\underset{\to} \Lim \lambda_0(i) 
\times \underset{\to} \Lim \mu_0(j), \bigvee_{\Theta \in \prod_{i \in I}^{\mt}F_i}^{H 
\in \prod_{j \in J}^{\mt}G_j}\eql_0 f_{\Theta} \circ \pi_1, \eql_0 f_{H} \circ \pi_2 \bigg). \]
To show that $\theta \in \Mor \big(\underset{\to} \Lim (\C F_i \times \C G_j), \underset{\to} \Lim \C F_i 
\times \underset{\to} \Lim \C G_j\big)$ it suffices to show that
\[ \forall_{\Theta \in \prod_{i \in I}^{\mt}F_i}\forall_{H \in \prod_{j 
\in J}^{\mt}G_j}\big(\eql_0 f_{\Theta} \circ \pi_1) \circ \theta \in \underset{\to} \Lim 
(F_i \times G_j) \ \& \ (\eql_0 f_{H} \circ \pi_2) \circ \theta \in \underset{\to} \Lim (F_i 
\times G_j)\big). \]
If $\Theta \in \prod_{i \in I}^{\mt}F_i$, we show that $(\eql_0 f_{\Theta} \circ \pi_1) 
\circ \theta \in \underset{\to} \Lim (F_i \times G_j)$  
From the equalities
\begin{align*}
[(\eql_0 f_{\Theta^{\mt}} \circ \pi_1) \circ \theta]\big(\eql_0^{\Lambda^{\lt} \times M^{\lt}}\big((i, j),
(x, y)\big)\big) & := (\eql_0 f_{\Theta^{\mt}} \circ \pi_1)\big(\eql_0^{\Lambda^{\lt}}(i, x), 
\eql_0^{M^{\lt}}(j, y)\big)\\
& := \eql_0 f_{\Theta^{\mt}}\big(\eql_0^{\Lambda^{\lt}}(i, x)\big)\\
& := \Theta_i(x)\\
& =: \big(\Theta_i \circ \pi_1\big)(x, y)\\
& := \big[\Theta_1(i, j)\big](x, y)\\
& := \eql_0 f_{\Theta_1}\big(\eql_0^{\Lambda^{\lt} \times M^{\lt}}\big((i, j), (x, y)\big)\big),
\end{align*}
where $\Theta_1 \in \prod_{i \in I}^{\mt} F_i \times G_j$ is defined in Lemma~\ref{lem: prodlemma},
we conclude that $(\eql_0 f_{\Theta^{\mt}} \circ \pi_1) \circ \theta := \eql_0 f_{\Theta_1} \in
\underset{\to} \Lim (F_i \times G_j)$. For the second case we work similarly.
\end{proof}

\section{Inverse limit of Bishop spaces}
\label{sec: inverselimit}

\begin{defi}\label{def: inverselimit}
If $S(\Lambda^{\mt}) := (\lambda_0, \lambda_1^{\mt}, \phi_0^{\Lambda^{\mt}}, \phi_1^{\Lambda^{\mt}})$ is 
in $\Spec(I, \mt)$ i.e., a contravariant $(I, \lt)$-spectrum, with Bishop spaces $(\C F_i)_{i \in I}$ and Bishop morphisms 
$(\lambda_{ji}^{\mt})_{(i,j) \in D^{\lt}(I)}$, the 
\textit{inverse limit} of $S(\Lambda^{\mt})$\index{inverse limit of a contravariant direct spectrum of 
Bishop spaces} is the Bishop space
\[ \underset{\ot} \Lim \C F_i := \big(\underset{\ot} \Lim \lambda_0 (i), \ \underset{\ot} \Lim F_i\big), \]
\[ \underset{\ot} \Lim \lambda_0 (i) := \prod_{i \in I}^{\mt}\lambda_0 (i) \ \ \ \& \ \ \ \
\underset{\ot} \Lim F_i := \bigvee_{i \in I}^{f \in F_i}f \circ \pi_i^{\Lambda^{\mt}}. \]
\end{defi}

We write $\pi_i$ instead of $\pi_i^{\Lambda^{\mt}}$ for the 
function $\pi_i^{\Lambda^{\mt}} : \prod_{i \in I}^{\mt}\lambda_0 (i) \to \lambda_0(i)$, which is 
defined, as its dual $\pi_i^{\Lambda^{\lt}}$ in the Proposition~\ref{prp: preorderfamilymap1}(iv), by
the rule $\Phi \mapsto \Phi_i$, for every $i \in I$.

\begin{prop}[Universal property of the inverse limit]\label{prp: universalinverse}
 If  $S(\Lambda^{\mt}) := (\lambda_0, \lambda_1^{\mt}, \phi_0^{\Lambda^{\mt}}, \phi_1^{\Lambda^{\mt}})$ 
is in $\Spec(I, \mt)$ with Bishop spaces $(\C F_i)_{i \in I}$ and 
Bishop morphisms $(\lambda_{ji}^{\mt})_{(i,j) \in \lt(I)}$, its inverse limit $\underset{\ot} \Lim \C F_i$
satisfies the universal property of inverse limits:\\[1mm]
\normalfont (i)
\itshape  For every $i \in I$, we have that $\pi_i \in \Mor(\underset{\ot} \Lim \C F_i, \C F_i)$.\\[1mm]
\normalfont (ii)
\itshape If $i \lt j$, the following left diagram commutes 
\begin{center}
\begin{tikzpicture}

\node (E) at (0,0) {$\prod_{i \in I}^{\mt}\lambda_0(i)$};
\node[below=of E] (F) {};
\node [right=of F] (B) {$\lambda_0(j)$};
\node [left=of F] (C) {$\lambda_0(i)$};
\node [right=of B] (D) {$\lambda_0(i)$};
\node[right=of D] (K) {};
\node[right=of K] (L) {$\lambda_0(j)$.};
\node[above=of K] (M) {$Y$};

\draw[->] (E)--(B) node [midway,right] {$ \ \  \pi_j$};
\draw[->] (E)--(C) node [midway,left] {$\pi_i \ $};
\draw[->] (B)--(C) node [midway,below] {$\lambda_{ji}^{\mt}$};

\draw[->] (M)--(L) node [midway,right] {$ \ \varpi_j$};
\draw[->] (M)--(D) node [midway,left] {$\varpi_i \ $};
\draw[->] (L)--(D) node [midway,below] {$\lambda_{ji}^{\mt}$};

\end{tikzpicture}
\end{center}
\normalfont (iii)
\itshape If $\C G := (Y, G)$ is a Bishop space and 
$\varpi_i : Y \to \lambda_0(i) \in \Mor(\C G, \C F_i)$, for every $i \in I$, are such that if 
$i \lt j$,
the above right diagram commutes, there is a 
unique function $h : Y \to \prod_{i \in I}\lambda_0(i) \in \Mor(\C G, \underset{\ot} \Lim \C F_i)$
such that the following diagrams commute
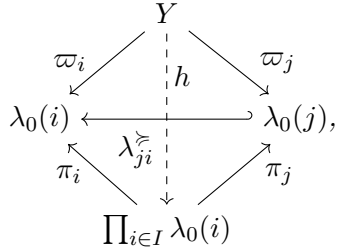
\begin{center}
\begin{tikzpicture}

\node (E) at (0,0) {$Y$};
\node[below=of E] (F) {};
\node [right=of F] (B) {$\lambda_0(j)$,};
\node [left=of F] (C) {$\lambda_0(i)$};
\node[below=of F] (G) {$\prod_{i \in I}\lambda_0(i)$};

\draw[->] (E)--(B) node [midway,right] {$ \ \  \varpi_j$};
\draw[->] (E)--(C) node [midway,left] {$\varpi_i \ $};
\draw[left hook->] (B)--(C) node [midway,below] {$\lambda_{ji}^{\mt} \ \ \ \ \ \  $};
\draw[->] (G)--(B) node [midway,right] {$ \ \  \pi_j$};
\draw[->] (G)--(C) node [midway,left] {$\pi_i \ $};
\draw[dashed,->] (E)--(G) node [midway,near start] {$\ \ \ h$};

\end{tikzpicture}
\end{center}
\end{prop}

\begin{proof}
The condition
$\pi_i \in \Mor(\underset{\ot} \Lim \C F_i, \C F_i) :\TOT \forall_{f \in F_i}\bigg(f \circ \pi_i \in
\bigvee_{i \in I}^{f \in F_i}f \circ \pi_i\bigg)$ is trivially satisfied, and (i) follows.
For (ii), the required equality $\lambda_{ji}^{\mt}\big(\pi_j(\Phi)\big) 
=_{\mathsmaller{\lambda_0(i)}} \pi_i(\Phi) :\TOT \lambda_{ji}^{\mt}(\Phi_j) 
= _{\mathsmaller{\lambda_0(i)}} \Phi_i$ holds by 
the definition of $\prod_{i \in I}^{\mt}\lambda_0(i)$.
To show (iii), let the operation $h : Y \sto \prod_{i \in I}^{\mt}\lambda_0(i)$, defined by 
$h(y) := \Phi_y$, where $\Phi_y(i) := \varpi_i(y)$, for every $y \in Y$ and $i \in I$.
First we show that $h$ is well-defined i.e., $h(y) \in \prod_{i \in I}^{\mt}\lambda_0(i)$. If $i \lt j$, by the 
supposed commutativity of the above right diagram we have that
$\lambda_{ji}^{\mt}\big(\Phi_y(j)\big) :=  \lambda_{ji}^{\mt}\big(\varpi_j(y)\big) = \varpi_i(y) 
:= \Phi_y(i)$.
Next we show that $h$ is a function.
If $y =_Y y{'}$, 
the last formula in the following equivalences
\[ \Phi_y =_{\mathsmaller{\prod_{i \in I}^{\mt}\lambda_0(i)}} \Phi_{y{'}} :\TOT \forall_{i \in I}\big(\Phi_y(i)
=_{\mathsmaller{\lambda_0(i)}} \Phi_{y{'}}(i)\big) : \TOT \forall_{i \in I}\big(\varpi_i(y)
=_{\mathsmaller{\lambda_0(i)}} \varpi_i(y{'})\big) \]
holds by the fact that $\varpi_i$ is a function, for every $i \in I$. By the 
$\bigvee$-lifting of morphisms we have that 
$h \in \Mor(\C G, \underset{\ot} \Lim \C F_i) \TOT \forall_{i \in I}\forall_{f \in F_i}\big((f \circ \pi_i)
\circ h \in G\big)$. 
If $i \in I$, $f \in F_i$, and $y \in Y$, then
\[ \big[(f \circ \pi_i) \circ h \big](y) := (f \circ \pi_i)(\Phi_y) := f\big(\Phi_y(i)\big) := 
f\big(\varpi_i(y)\big) =: (f \circ \varpi_i)(y), \]
hence $(f \circ \pi_i) \circ h := f \circ \varpi_i \in G$, since $\varpi_i \in \Mor(\C G, \C F_i)$. The required
commutativity of the last diagram above, and the uniqueness of $h$ follow immediately.
\end{proof}

The uniqueness of $\underset{\ot} \Lim \lambda_0 (i)$, up to Bishop isomorphism, follows easily from 
its universal property. Next follows the inverse analogue to the Theorem~\ref{thm: directspectrummap1}.

\begin{thm}\label{thm: inverselimitmap}
Let $S(\Lambda^{\mt}) := (\lambda_0, \lambda_1^{\mt}, \phi_0^{\Lambda^{\mt}}, \phi_1^{\Lambda^{\mt}})$ be in
$\Spec(I, \mt)$ with Bishop spaces $(\C F_i)_{i \in I}$ and Bishop morphisms 
$(\lambda_{ji}^{\mt})_{(i,j) \in D^{\lt}(I)}$, and $S(M^{\mt}) := (\mu_0, \mu_1, \phi_{0}^{M^{\mt}}, \phi_{1}^{M^{\mt}})$
in $\Spec(I, \mt)$ with Bishop spaces $(\C G_i)_{i \in I}$ and Bishop morphisms 
$(\mu_{ji}^{\mt})_{(i,j) \in D^{\lt}(I)}$, and $\Psi \colon S(\Lambda^{\mt}) \To S(M^{\mt})$.\\[1mm]
\normalfont (i)
\itshape There is a unique function $\Psi_{\ot} \colon \underset{\ot} \Lim \lambda_0 (i) \to 
\underset{\ot} \Lim \mu_0 (i)$ such that, for every $i \in I$, the following diagram commutes
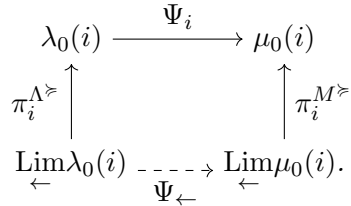
\begin{center}
\begin{tikzpicture}

\node (E) at (0,0) {$\underset{\ot} \Lim \lambda_0 (i)$};
\node[right=of E] (F) {$\underset{\ot} \Lim \mu_0 (i)$.};
\node[above=of F] (A) {$\mu_0(i)$};
\node [above=of E] (D) {$\lambda_0(i)$};

\draw[dashed,->] (E)--(F) node [midway,below]{$\Psi_{\ot}$};
\draw[->] (D)--(A) node [midway,above] {$\Psi_i$};
\draw[->] (E)--(D) node [midway,left] {$\pi^{\Lambda^{\mt}}_i$};
\draw[->] (F)--(A) node [midway,right] {$\pi^{M^{\mt}}_i$};

\end{tikzpicture}
\end{center}
\normalfont (ii)
\itshape If $\Psi$ is continuous, then 
$\Psi_{\ot} \in \Mor(\underset{\ot} \Lim \C F_i, \underset{\ot} \Lim \C G_i)$.\\[1mm]
\normalfont (iii)
\itshape If $\Psi_i$ is an embedding, for every $i \in I$,
then $\Psi_{\ot}$ is an embedding.
\end{thm}

\begin{proof}
(i) Let the assignment routine $\Psi_{\ot} \colon \underset{\ot} \Lim \lambda_0 (i) \sto 
\underset{\ot} \Lim \mu_0 (i)$, defined by 
\[ \Theta \mapsto \Psi_{\ot}(\Theta), \ \ \ \  \big[ \Psi_{\ot}(\Theta) \big]_i := \Psi_i(\Theta_i);  \ \ \ \ 
\Theta \in \underset{\ot} \Lim \lambda_0(i), \ i \in I. \]
First we show that $\Psi_{\ot}$ is well-defined i.e., $\Psi_{\ot}(\Theta) \in \prod_{i \in I}^{\mt}\mu_0 (i).$
If $i \lt j$, since $\Theta \in \prod_{i \in I}^{\mt}\lambda_0 (i)$, we have that 
$\Theta_i = \lambda_{ji}^{\mt}(\Theta_j)$, and since $\Psi \colon S(\Lambda^{\mt}) \To S(M^{\mt})$
\begin{center}
\begin{tikzpicture}

\node (E) at (0,0) {$\mu_0(i)$};
\node[right=of E] (F) {$\mu_0(j)$,};
\node[above=of E] (A) {$\lambda_0(i)$};
\node [right=of A] (D) {$\lambda_0(j)$};

\draw[->] (F)--(E) node [midway,below]{$\mu_{ji}^{\mt}$};
\draw[->] (D)--(A) node [midway,above] {$\lambda_{ji}^{\mt}$};
\draw[->] (A)--(E) node [midway,left] {$\Psi_i$};
\draw[->] (D)--(F) node [midway,right] {$\Psi_j$};

\end{tikzpicture}
\end{center}
$\big[\Psi_{\ot}(\Theta)\big]_i := \Psi_i(\Theta_i)
= \Psi_i\big(\lambda_{ji}^{\mt}(\Theta_j)\big)
=: \big(\Psi_i \circ \lambda_{ji}^{\mt}\big)(\Theta_j) 
= \big(\mu_{ji}^{\mt} \circ \Psi_j\big)(\Theta_j) 
 := \mu_{ji}^{\mt}\big(\Psi_j(\Theta_j)\big)
=: \mu_{ji}^{\mt}\big(\big[\Psi_{\ot}(\Theta)\big]_j\big).
$
$\Psi_{\ot}$ is a function: 
$\Theta =_{\mathsmaller{\underset{\ot} \Lim \lambda_0 (i)}} \Phi : \TOT
\forall_{i \in I}\big(\Theta_i =_{\mathsmaller{\lambda_0 (i)}} \Phi_i\big)
 \To \forall_{i \in I}\big(\Psi_i(\Theta_i) =_{\mathsmaller{\mu_0 (i)}} \Psi_i(\Phi_i)\big) 
: \TOT \forall_{i \in I}\big(\big[\Psi_{\ot}(\Theta)\big]_i =_{\mathsmaller{\mu_0 (i)}}
 \big[\Psi_{\ot}(\Phi)\big]_i\big)
 : \TOT \Psi_{\ot}(\Theta) =_{\mathsmaller{\underset{\ot} \Lim \mu_0 (i)}}  \Psi_{\ot}(\Phi)$.
The commutativity of the diagram and the uniqueness of $\Psi_{\ot}$ are immediate to show.\\
(ii) By the $\bigvee$-lifting of morphisms
$ \Psi_{\ot} \in \Mor(\underset{\ot} \Lim \C F_i, \underset{\ot} \Lim \C G_i) \TOT 
\forall_{i \in I}\forall_{g \in G_i}\big((g \circ \pi_i^{M^{\mt}}) \circ \Psi_{\ot}
\in \underset{\ot} \Lim F_i\big). $
If $i \in I$ and $g \in G_i$, then
$[(g \circ \pi_i^{M^{\mt}}) \circ \Psi_{\ot}](\Theta) :=
g \big([\Psi_{\ot}(\Theta)\big]_i\big)
:= g \big(\Psi_{i}(\Theta_i)\big)
=: \big(g \circ \Psi_{i}\big)(\Theta_i)
:= \big[\big(g \circ \Psi_{i}\big) \circ \pi_i^{\Lambda^{\mt}}\big](\Theta),
$
and $g \circ \Psi_{i} \in F_i$, by the continuity of $\Psi$, hence
$(g \circ \pi_i^{M^{\mt}}) \circ \Psi_{\ot} := 
\big(g \circ \Psi_{i}\big) \circ \pi_i^{\Lambda^{\mt}} \in \underset{\ot} \Lim F_i$.\\
(iii) By definition we have that
$
\Psi_{\ot}(\Theta) =_{\mathsmaller{\underset{\ot} \Lim \mu_0 (i)}}  \Psi_{\ot}(\Phi)
: \TOT \forall_{i \in I}\big(\Psi_i(\Theta_i) =_{\mathsmaller{\mu_0 (i)}} \Psi_i(\Phi_i)\big)
\To $
$\forall_{i \in I}\big(\Theta_i =_{\mathsmaller{\lambda_0 (i)}} \Phi_i\big)
: \TOT \Theta =_{\mathsmaller{\underset{\ot} \Lim \lambda_0 (i)}} \Phi.$
\end{proof}

\begin{prop}\label{prp: transitivity2}
If $S(\Lambda^{\mt}) := (\lambda_0, \lambda_1^{\mt}, \phi_0^{\Lambda^{\mt}}, \phi_1^{\Lambda^{\mt}})$,  
$S(M^{\mt}) := (\mu_0, \mu_1^{\mt}, \phi_0^{M^{\mt}}, \phi_1^{M^{\mt}})$ and
$S(N^{\mt} ):= (\nu_0, \nu_1^{\mt}, \phi_0^{N^{\mt}}, \phi_1^{N^{\mt}})$
are in $\Spec(I, \mt)$, and if $\Psi \colon S(\Lambda^{\mt}) \To S(M^{\mt})$ and 
$\Xi \colon S(M^{\mt}) \To S(N^{\mt})$, then 
$(\Xi \circ \Psi)_{\ot} :=  \Xi_{\ot} \circ \Psi_{\ot}$
\begin{center}
\begin{tikzpicture}

\node (E) at (0,0) {$\lambda_0(i)$};
\node[right=of E] (H) {};
\node[right=of H] (F) {$\mu_0(i)$};
\node[below=of E] (A) {$\underset{\ot} \Lim \lambda_0 (i)$};
\node[below=of F] (B) {$\underset{\ot} \Lim \mu_0 (i)$};
\node[right=of F] (K) {};
\node[right=of K] (G) {$\nu_0(i)$};
\node [below=of G] (C) {$\underset{\ot} \Lim \nu_0 (i)$.};

\draw[->] (E)--(F) node [midway,above] {$ \Psi_i$};
\draw[->] (F)--(G) node [midway,above] {$\Xi_i$};
\draw[->] (B)--(C) node [midway,below] {$\Xi_{\ot} \ \ $};
\draw[->] (A)--(B) node [midway,below] {$ \ \ \Psi_{\ot}$};
\draw[->] (A)--(E) node [midway,left] {$\pi^{\Lambda^{\lt}}_i$};
\draw[->] (B) to node [midway,left] {$\pi^{M^{\lt}}_i$} (F);
\draw[->] (C)--(G) node [midway,right] {$\pi^{N^{\lt}}_i$};
\draw[->,bend right] (A) to node [midway,below] {$(\Xi \circ \Psi)_{\ot}$} (C) ;
\draw[->,bend left] (E) to node [midway,above] {$(\Xi \circ \Psi)_i$} (G) ;

\end{tikzpicture}
\end{center}
\end{prop}

\begin{thm}\label{thm: cofinal3}
Let  $S(\Lambda^{\mt}) := (\lambda_0, \lambda_1^{\mt}, \phi_0^{\Lambda^{\mt}}, \phi_1^{\Lambda^{\mt}})$
in $\Spec(I, \mt)$, $(J, e, \cof_J)$ a cofinal subset of $I$, and $S(\Lambda^{\mt}) 
\circ e :=
\big(\lambda_0 \circ e, \lambda_1 \circ e, \phi_0^{\Lambda^{\mt} \circ e}, \phi_1^{\Lambda^{\mt} \circ e}\big)$ 
the relative spectrum of $S(\Lambda^{\mt})$ to $J$. Then $\underset{\ot} \Lim \C F_j \simeq \underset{\ot} \Lim \C F_i$. 
\end{thm}

\begin{proof}
If $\Theta \in \prod_{j \in J}^{\mt}\lambda_0(j)$, then, if $j \lt j{'}$, we have that
$\Theta_j = \lambda_{j{'}j}^{\mt}\big(\Theta_{j{'}}\big) := \lambda_{e(j{'})e(j)}^{\mt}\big(\Theta_{j{'}}\big)$.
If $i \in I$, then $\cof_J(i) \in J$ and $\Theta_{\cof_J(i)} \in \lambda_0(e(\cof_J(i)))$. Since
$i \lt e(\cof_J(i))$, we define the operation  
$\phi : \underset{\ot} \Lim \lambda_0 (j) \sto \underset{\ot} 
\Lim \lambda_0 (i)$, by the rule $\Theta \mapsto \phi(\Theta)$, for every $\Theta \in \underset{\ot} \Lim \lambda_0 (j)$,
where
\[ [\phi(\Theta)]_i := \lambda_{e(\cof_J(i)) i}^{\mt}\big(\Theta_{\cof_J(i)}\big) \in \lambda_0(i); \ \ \ \
i \in I. \]
First we show that $\phi$ is well-defined i.e., $\phi(\Theta) \in \prod_{i \in I}^{\mt}\lambda_0(i)$ i.e., 
for every $i, i{'} \in I$, 
$i \lt i{'} \To [\phi(\Theta)]_i = \lambda_{i{'}i}^{\mt}\big([\phi(\Theta)]_{i{'}}\big)$.
Working as in the proof of Lemma~\ref{lem: cofinallemma}(ii), we get
\begin{align*}
\lambda_{i{'}i}^{\mt}\big([\phi(\Theta)]_{i{'}}\big) & := 
\lambda_{i{'}i}^{\mt}\big(\lambda_{e(\cof_J(i{'})) i{'}}^{\mt}\big(\Theta_{\cof_J(i{'})}\big)\big)\\
& =: \big[\lambda_{i{'}i}^{\mt} \circ \lambda_{e(\cof_J(i{'})) i{'}}^{\mt}\big]\big(\Theta_{\cof_J(i{'})}\big)\\
& \stackrel{(\ref{eq: eq1})} = \lambda_{e(\cof_J(i{'})) i}^{\mt} \big(\Theta_{\cof_J(i{'})}\big)\\
& \stackrel{(\ref{eq: eq2})} = \big[\lambda_{e(\cof_J(i))i}^{\mt} \circ 
\lambda_{e(\cof_J(i{'})) e(\cof_J(i))}^{\mt}\big] \big(\Theta_{\cof_J(i{'})}\big)\\
& := \lambda_{e(\cof_J(i))i}^{\mt}\big(\lambda_{e(\cof_J(i{'})) 
e(\cof_J(i))}^{\mt}\big(\Theta_{\cof_J(i{'})}\big)\big)\\
& = \lambda_{e(\cof_J(i))i}^{\mt}\big(\Theta_{\cof_J(i)}\big)\\
& := [\phi(\Theta)]_i.
\end{align*}
To show that $\phi$ is a function we consider the following equivalences:
\begin{align*}
\phi(\Theta) =_{\mathsmaller{\underset{\ot} \Lim \lambda_0 (i)}} \phi(H) & :\TOT
\forall_{i \in I}\big(\big[\phi(\Theta)\big]_i =_{\mathsmaller{\lambda_0 (i)}} \big[\phi(H)\big]_i\big)\\
& : \TOT \forall_{i \in I}\big(\lambda_{e(\cof_J(i)) i}^{\mt}\big(\Theta_{\cof_J(i)}\big)
=_{\mathsmaller{\lambda_0 (i)}} \lambda_{e(\cof_J(i)) i}^{\mt}\big(H_{\cof_J(i)}\big)\big), \ \ \ (1)
\end{align*}
\[ \Theta =_{\mathsmaller{\underset{\ot} \Lim \lambda_0 (j)}} H : 
\TOT \forall_{j \in J}\big(\Theta_j =_{\mathsmaller{\lambda_0 (e(j))}} H_j\big) \ \ \ \ \ (2). \]
To show that (1) $\To$ (2) we use the fact that $e(\cof_J(j)) = j$, and since $j \lt j$, by the 
extensionality of $\lt$ we get $j \lt e(\cof_J(j))$. Since  
$\Theta_j = \lambda_{e(\cof_J(i)) i}^{\mt}\big(\Theta_{\cof_J(i)}\big)$, and 
$H_j = \lambda_{e(\cof_J(i)) i}^{\mt}\big(H_{\cof_J(i)}\big)$, we get (2).
By the $\bigvee$-lifting of morphisms
$\phi \in \Mor(\underset{\ot} \Lim \C F_j, \underset{\ot} \Lim \C F_i) \TOT \forall_{i \in I}\forall_{f \in F_i}
\big((f \circ \pi_i^{S^{\mt}}) \circ \phi \in \underset{\ot} \Lim F_j\big)$. 
If $\Theta \in \prod_{j \in J}^{\mt}\lambda_0(j)$, we have that
\begin{align*}
[(f \circ \pi_i^{S^(\Lambda{\mt})}) \circ \phi](\Theta) & := f\big(\pi_i^{S(\Lambda^{\mt})}(\phi(\Theta))\big)\\
& := f\big(\big[\phi(\Theta)\big]_i\big)\\
& := f\big(\lambda_{e(\cof_J(i)) i}^{\mt}\big(\Theta_{\cof_J(i)}\big)\big)\\
& =: \big(f \circ \lambda_{e(\cof_J(i)) i}^{\mt}\big)\big(\Theta_{\cof_J(i)}\big)\\
& =: \big[\big(f \circ \lambda_{e(\cof_J(i)) i}^{\mt}\big) \circ \pi_{\cof_J(i)}^{S(\Lambda^{\mt}) \circ e}\big](\Theta),
\end{align*}
hence 
$\big(f \circ \pi_i^{S(\Lambda^{\mt})}\big) \circ \phi :=
\big(f \circ \lambda_{e(\cof_J(i)) i}^{\mt}\big) \circ \pi_{\cof_J(i)}^{S(\Lambda^{\mt}) \circ e}
\in \underset{\ot} \Lim F_j$, 
as by definition $\lambda_{e(\cof_J(i)) i}^{\mt} \in \Mor(\C F_{e(\cof_J(i))}, \C F_i)$, and hence
\begin{center}
\begin{tikzpicture}

\node (E) at (0,0) {$\lambda_0(e(\cof_J(i)))$};
\node[right=of E] (H) {};
\node[right=of H] (F) {$\lambda_0(i)$};
\node[below=of F] (A) {$\Real$};

\draw[->] (E)--(F) node [midway,above] {$\lambda_{e(\cof_J(i)) i}^{\mt} $};
\draw[->] (F)--(A) node [midway,right] {$f$};
\draw[->] (E)--(A) node [midway,left] {$f \circ \lambda_{e(\cof_J(i)) i}^{\mt} \ \  $};

\end{tikzpicture}
\end{center}
$f \circ \lambda_{e(\cof_J(i)) i}^{\mt} \in F_{e(\cof_J(i))} := F_{\cof_J(i)}$. 
Let the operation $\theta \colon \underset{\ot} \Lim \lambda_0 (i) \sto \underset{\ot} \Lim \lambda_0(j)$,
defined by the rule $H \mapsto \theta(H) := H^J$, for every  
$H \in \prod_{i \in I}^{\mt}\lambda_0(i) $, where
$H_j^J := H_{e(j)} \in \lambda_0(e(j))$, for every $j \in J$. We show that $H^J \in \prod_{j \in J}\lambda_0(j)$. 
If $j \lt j{'}$, then 
\[ H_j^J = \lambda_{e(j{'})e(j)}^{\mt}\big(H_{j{'}}^J\big) :\TOT   
H_{e(j)} = \lambda_{e(j{'})e(j)}^{\mt}\big(H_{e(j{'})}\big), \]
which holds by the hypothesis $H \in \prod_{i \in I}^{\mt}\lambda_0(i)$. Moreover,
we have that $\phi(H^J) = H :\TOT \forall_{i \in I}\big(\big[\phi(H^J)\big]_i =_{\mathsmaller{\lambda_0(i)}} H_i\big)$.
If $i \in I$, and since $i \lt e(\cof_J(i))$, we have that 
\[ \big[\phi(H^J)\big]_i := \lambda_{e(\cof_J(i)) i}^{\mt}\big(H^J_{\cof_J(i)}\big)
:= \lambda_{e(\cof_J(i)) i}^{\mt}\big(H_{e(\cof_J(i))}\big)
= H_i. \]
It is immediate to show that $\theta$ is a function. Moreover, $\theta(\phi(\Theta) = \Theta$, as if $j \in J$, then
\[ \phi(\Theta)_j^J := \phi(\Theta)_{e(j)}
:= \lambda^{\mt}_{e(\cof_J(e(j))e(j)}\big(\Theta_{\cof_J(e(j))}\big)
= \Theta_j,
\]
as by hypothesis $\Theta_j = \lambda_{e(j{'})e(j)}^{\mt}(\Theta_{j{'}})$, with $j \lt j{'}$, and by $(\Cf_1)$ we 
have that
$j =_J (\cof_J(e(j))$, hence by the extensionality of $\lt$ we get $j \lt (\cof_J(e(j))$. Finally,
$\theta \in \Mor \big(\underset{\ot} \Lim \C F_i, \underset{\ot} \Lim \C F_j \big) \TOT \forall_{j
\in J}\forall_{f \in F_j}\big(\big(f \circ \pi_j^{S(\Lambda^{\mt})}\big) \circ \theta \in \underset{\ot} 
\Lim F_i\big)$, which follows from 
\begin{align*}
\big[\big(f \circ \pi_j^{S(\Lambda^{\mt})}\big) \circ \theta \big](H) & := \big(f \circ
\pi_j^{S(\Lambda^{\mt})}\big)(H^J)\\
& := f\big(H^J_j\big)\\
& := f \big(H_{e(j)}\big)\\
& =: \big(f \circ \pi_{e(j)}^{S(\Lambda^{\mt})}\big)(H). \qedhere
\end{align*}
\end{proof}

\begin{prop}\label{prp: prodinverse}
If $(I, \lt), (J, \lt)$ are directed sets, $S(\Lambda^{\mt}) := (\lambda_0, \lambda_1^{\mt},
\phi_0^{\Lambda^{\mt}},  \phi_1^{\Lambda^{\mt}})$ is in $\Spec(I, \mt)$ 
with Bishop spaces $(\C F_i)_{i \in I}$ and Bishop morphisms $(\lambda_{i{'}i}^{\mt})_{(i, i{'}) \in D^{\lt}(I)}$,
and $S(M^{\mt}) := (\mu_0, \mu_1^{\mt}, \phi_0^{M^{\mt}},  \phi_1^{M^{\mt}})$ is in $\Spec(J, \mt)$
with Bishop spaces $(\C G_j)_{j \in J}$ and Bishop morphisms 
$(\mu_{j{'}j}^{\mt})_{(j, j{'}) \in D^{\lt}(J)}$, there is a function
\[ \times : \prod_{i \in I}^{\mt}\lambda_0 (i) \times \prod_{j \in J}^{\mt}\mu_0 (j) \to 
\prod_{(i, j) \in I \times J}^{\mt}\lambda_0 (i) \times \mu_0 (j) \in \Mor\big(\underset{\ot} 
\Lim \C F_i \times \underset{\ot} \Lim \C G_j, \underset{\ot} \Lim (\C F_i \times \C G_j)\big). \]
\end{prop}

\begin{proof}
We proceed as in the proof of Proposition~\ref{prp: proddirect}. 
\end{proof}

\section{Duality between direct and inverse limits}
\label{sec: duality}

\begin{prop}\label{prp: plus}
Let $\C F := (X, F), \C G := (Y, G)$ and $\C H := (Z, H)$ be Bishop spaces, and let
$\lambda \in \Mor(\C G, \C H),
\mu \in \Mor(\C H, \C G)$. We define the mappings
\[ \lambda^{+} \colon \Mor(\C H, \C F) \to \Mor(\C G, \C F),  \ \ \ \  \lambda^{+} (\phi) := \phi 
\circ \lambda; \ \ \ \ \phi \in \Mor(\C H, \C F), \]
\[ \mu^{-} \colon \Mor(\C F, \C H) \to \Mor(\C F, \C G), \ \ \ \ \mu^{-} (\theta) := \mu \circ \theta; \ \ \ \ 
\theta \in \Mor(\C F, \C H), \]
\begin{center}
\begin{tikzpicture}

\node (E) at (0,0) {$Y$};
\node[above=of E] (F) {$Z$};
\node[right=of F] (A) {$X$};
\node[right=of A] (B) {$Z$};
\node[below=of B] (C) {$Y$.};

\draw[->] (E)--(F) node [midway,left] {$\lambda$};
\draw[->] (F)--(A) node [midway,above] {$\phi$};
\draw[->] (E)--(A) node [midway,right] {$\mathsmaller{\phi \circ \lambda}$};
\draw[->] (A)--(C) node [midway,left] {$\mathsmaller{\mu \circ \theta} \ $};
\draw[->] (A)--(B) node [midway,above] {$\theta$};
\draw[->] (B)--(C) node [midway,right] {$\mu$};

\end{tikzpicture}
\end{center}
\[ ^{+} \colon \Mor(\C G, \C H) \to \Mor(\C H \to \C F, \C G \to \C F) , \ \ \ \ \lambda \mapsto \lambda^+; \ \ \ \ 
\lambda \in \Mor(\C G, \C H), \]
\[ ^{-} \colon \Mor(\C H, \C G) \to \Mor(\C F \to \C H, \C F \to \C G), \ \ \ \ \mu \mapsto \mu^-; \ \ \ \ 
\mu \in \Mor(\C H, \C G). \]
Then $^{+} \in \Mor\big(\C G \to \C H, (\C H \to \C F) \to (\C G \to \C F)\big)$ and 
$^{-} \in \Mor\big(\C H \to\C G, (\C F \to \C H) \to (\C F \to \C G)\big)$.
\end{prop}

\begin{proof}
By definition and the $\bigvee$-lifting of the exponential topology we have that
\[ \C G \to \C H := \bigg(\Mor(\C G, \C H), \bigvee_{y \in Y}^{h \in H}\phi_{y, h}\bigg), \ \ \ \
\C H \to \C F := \bigg(\Mor(\C H, \C F), \bigvee_{z \in Z}^{f\in F}\phi_{z, f}\bigg), \]
\[ \C G \to \C F := \bigg(\Mor(\C G, \C F), \bigvee_{y \in Y}^{f \in F}\phi_{y, f}\bigg), \]
\[ (\C H \to \C F) \to (\C G \to \C F) := \bigg(\Mor((\C H \to \C F), (\C G \to \C F)), 
\bigvee_{\varphi \in \Mor(\C H, \C F)}^{e \in G \to F}\phi_{\varphi, e}\bigg), \]
\[ \bigvee_{\varphi \in \Mor(\C H, \C F)}^{e \in G \to F}\phi_{\varphi, e} = 
\bigvee_{\varphi \in \Mor(\C H, \C F)}^{y \in Y, f \in F}\phi_{\varphi, \phi_{y, f}}. \]
By the  $\bigvee$-lifting of morphisms we have that 
\[ ^{+} \in \Mor\big(\C G \to \C H, (\C H \to \C F) \to (\C G \to \C F)\big) \TOT 
\forall_{\varphi \in \Mor(\C H, \C F)}\forall_{y \in Y}\forall_{f \in F}\big(\phi_{\varphi, 
\phi_{y, f}} \circ \ ^{+} \in G \to H\big). \]
If $\lambda \in \Mor(\C G, \C H)$, we have that
$ [\phi_{\varphi, \phi_{y, f}} \circ \ ^{+}](\lambda) := \phi_{\varphi, \phi_{y, f}}(\lambda^{+})
 := (\phi_{y, f} \circ \lambda^+)(\varphi) := (\phi_{y, f}(\varphi \circ \lambda)
 := [f \circ (\varphi \circ \lambda)](y) := [(f \circ \varphi) \circ \lambda](y)
=: \phi_{y, f \circ \varphi}(\lambda)$
i.e., $\phi_{\varphi, \phi_{y, f}} \circ \ ^{+} :=  \phi_{y, f \circ \varphi} \in G \to H$,
since $\varphi\in \Mor(\C H, \C F)$ and hence $f \circ \varphi \in H$. For the mapping $^{-}$ we work similarly.
\end{proof}

With the use of the exponential Bishop topology we can get a contravatiant spectrum
from a covariant one, and vice versa.

\begin{prop}\label{prp: dirtoproj}
\normalfont (A)
\itshape Let $S(\Lambda^{\lt}) := (\lambda_0, \lambda_1^{\lt}, \phi_0^{\Lambda^{\lt}}, \phi_1^{\Lambda^{\lt}}) \in \Spec(I, \lt)$ and $\C F := (X, F)$ a Bishop space. \\[1mm]
\normalfont (i)
\itshape If $S(\Lambda^{\lt}) \to \C F := (\mu_0, \mu_1^{\mt}, \phi_0^{M^{\mt}}, \phi_1^{M^{\mt}})$, where 
$M^{\mt} := (\mu_0, \mu_1^{\mt})$ 
is a contravariant direct family of sets over $(I, \lt)$ with $\mu_0 (i) := \Mor(\C F_i, \C F)$ and 
\[ \mu_1^{\mt}(i, j) :=  (\lambda_{ij}^{\lt})^{+} \colon \Mor(\C F_j, \C F) \to \Mor(\C F_i, \C F), \]
and if $\phi_0^{M^{\mt}}(i) := F_i \to F$ and $\phi_1^{M^{\mt}}(i, j) := \big(F_i \to F, F_j \to F, 
[(\lambda_{ij}^{\lt})^{+}]^*\big)$, then $S^{\lt} \to \C F$ is a contravariant $(I, \lt)$-spectrum
with Bishop spaces $(\C F_i \to \C F)_{i \in I}$ and Bishop morphisms 
$\big((\lambda_{ij}^{\lt})^{+}\big)_{(i, j) \in D^{\lt}(I)}$.\\[1mm]
\normalfont (ii)
\itshape If $\C F \to S(N^{\lt}) := (\nu_0, \nu_1^{\lt}, \phi_0^{N^{\lt}}, \phi_1^{N^{\lt}})$, where 
$N^{\lt} := (\nu_0, \nu_1^{\lt})$ is a direct family of sets over $(I, \lt)$ with $\nu_0 (i) := \Mor(\C F, \C F_i)$ and 
\[ \nu_1^{\lt}(i, j) := (\lambda_{ij}^{\lt})^{-} \colon \Mor(\C F, \C F_i) \to \Mor(\C F, \C F_j), \]
and if $\phi_0^{N^{\lt}}(i) := F \to F_i$ and $\phi_1^{N^{\lt}}(i, j) := \big(F \to F_j, F \to F_i,
[(\lambda_{ij}^{\lt})^{-}]^*\big)$, then $\C F \to S^{\lt}$ is a covariant $(I, \lt)$-spectrum with
Bishop spaces $(\C F \to \C F_i)_{i \in I}$ and Bishop morphisms 
$\big((\lambda_{ij}^{\lt})^{-}\big)_{(i, j) \in D^{\lt}(I)}$.\\[1mm]
\normalfont (B)
\itshape Let $S(\Lambda^{\mt}) := (\lambda_0, \lambda_1^{\mt}, \phi_0^{\Lambda^{\mt}}, \phi_1^{\Lambda^{\mt}})$
be a contravariant $(I, \lt)$-spectrum, and $\C F := (X, F)$ a Bishop space. \\[1mm]
\normalfont (i)
\itshape If $S(\Lambda^{\mt}) \to \C F := (\mu_0, \mu_1^{\lt}, \phi_0^{M^{\lt}}, \phi_1^{M^{\lt}})$, where 
$M^{\lt} := (\mu_0, \mu_1^{\lt})$ 
is a direct family of sets over $(I, \lt)$ with $\mu_0 (i) := \Mor(\C F_i, \C F)$ and 
\[ \mu_1^{\lt}(i, j) := (\lambda_{ji}^{\mt})^{+} \colon \Mor(\C F_i, \C F) \to \Mor(\C F_j, \C F), \]
and if $\phi_0^{M^{\lt}}(i) := F_i \to F$ and $\phi_1^{M^{\lt}}(i, j) := \big(F_j \to F, F_i \to F, 
[(\lambda_{ji}^{\mt})^{+}]^*\big)$, then $S^{\mt} \to \C F$ is an $(I, \lt)$-spectrum with
Bishop spaces $(\C F_i \to \C F)_{i \in I}$ and Bishop morphisms 
$\big((\lambda_{ji}^{\mt})^{+}\big)_{(i, j) \in D^{\lt}(I)}$.\\[1mm]
\normalfont (ii)
\itshape If $\C F \to S(N^{\mt}) := (\nu_0, \nu_1^{\mt}, \phi_0^{N^{\mt}}, \phi_1^{N^{\mt}})$, where 
$N^{\mt} := (\nu_0, \nu_1^{\mt})$ 
is a contravariant direct family of sets over $(I, \lt)$ with $\nu_0 (i) := \Mor(\C F, \C F_i)$ and 
\[ \nu_1^{\mt}(i, j) := (\lambda_{ji}^{\mt})^{-} \colon \Mor(\C F, \C F_j) \to \Mor(\C F, \C F_i), \]
and if $\phi_0^{N^{\mt}}(i) := F \to F_i$ and $\phi_1^{N^{\mt}}(i, j) := \big(F \to F_i, F \to F_j, 
[(\lambda_{ji}^{\mt})^{-}]^*\big)$, then $\C F \to S^{\lt}$ is a contravariant $(I, \lt)$-spectrum with
Bishop spaces $(\C F \to \C F_i)_{i \in I}$ and Bishop morphisms 
$\big((\lambda_{ij}^{\lt})^{-}\big)_{(i, j) \in D^{\lt}(I)}$.
\end{prop}

\begin{proof}
We prove only the case (A)(i) and for the other cases we work similarly. 
It suffices to show that if $i \lt j \lt k$, then the following diagram commutes
\begin{center}
\begin{tikzpicture}

\node (E) at (0,0) {$\Mor(F_j, \C F)$};
\node[right=of E] (F) {$\Mor(\C F_k, \C F).$};
\node [above=of E] (D) {$\Mor(\C F_i, \C F)$};

\draw[->] (F)--(E) node [midway,below] {$\big(\lambda_{jk}^{\lt}\big)^+$};
\draw[->] (E)--(D) node [midway,left] {$\big(\lambda_{ij}^{\lt}\big)^+$};
\draw[->] (F)--(D) node [midway,right] {$\ \  \big(\lambda_{ik}^{\lt}\big)^+$};

\end{tikzpicture}
\end{center}
If $\phi \in \Mor(\C F_k, \C F)$, then
$\big(\lambda_{ij}^{\lt}\big)^+\big[\big(\lambda_{jk}^{\lt}\big)^+(\phi)\big]  := 
\big(\lambda_{ij}^{\lt}\big)^+[\phi \circ \lambda_{jk}^{\lt}]
:= (\phi \circ \lambda_{jk}^{\lt}) \circ \lambda_{ij}^{\lt}
:= \phi \circ (\lambda_{jk}^{\lt} \circ \lambda_{ij}^{\lt})
= \phi \circ \lambda_{ik}^{\lt}
:= \big(\lambda_{ik}^{\lt}\big)^+ (\phi)$.
\end{proof}

Similarly to the $\bigvee$-lifting of the product topology, if 
$S(\Lambda^{\mt}) := (\lambda_0, \lambda_1^{\mt} ,\phi_0^{\Lambda^{\mt}},
\phi_1^{\Lambda^{\mt}})$ is a contravariant direct spectrum over $(I, \lt)$ with Bishop spaces
$\big(\C F_i = \big(\lambda_0(i), \bigvee F_{0i}\big)\big)_{i \in I}$, then 
\[ \prod_{ i \in I}^{\mt} F_i = \bigvee_{i \in I}^{f \in F_{0i}}\big(f \circ \pi_i^{\Lambda^{\mt}}\big). \]

\begin{thm}[Duality principle]\label{thm: duality1}
Let $S(\Lambda^{\lt}) := (\lambda_0, \lambda_1^{\lt}, \phi_0^{\Lambda^{\lt}},
\phi_1^{\Lambda^{\lt}}) \in \Spec(I, \lt)$ with Bishop spaces $(\C F_i)_{i \in I}$
and Bishop morphisms $(\lambda_{ij}^{\lt})_{(i, j) \in D^{\lt}(I)}$. If $\C F := (X, F)$ is a Bishop space and 
$S(\Lambda^{\lt}) \to \C F := (\mu_0, \mu_1^{\mt}, \phi_0^{M^{\mt}}, \phi_1^{M^{\mt}})$ is the contravariant 
direct spectrum over $(I, \lt)$ defined in Proposition~\ref{prp: dirtoproj} \normalfont (A)(i),
\itshape then 
\[ \underset{\ot} \Lim (\C F_i \to \C F) \simeq [(\underset{\to} \Lim \C F_i) \to \C F]. \]
\end{thm}

\begin{proof}
First we determine the topologies involved in the required  Bishop isomorphism.
By definition and by the above remark on the $\bigvee$-lifting of the $\prod^{\mt}$-topology we have that
\[ \underset{\ot} \Lim (\C F_i \to \C F) := \bigg(\prod_{i \in I}^{\mt}\mu_0 (i), 
\bigvee_{i \in I}^{g \in F_i \to F}g \circ \pi_i^{S(\Lambda^{\lt}) \to \C F}\bigg), \]
\[ F_i \to F := \bigvee_{x \in \lambda_0 (i)}^{f \in F}\phi_{x, f}, \]
\[ \bigvee_{i \in I}^{g \in F_i \to F}g \circ \pi_i^{S^{\lt} \to \C F} = 
\bigvee_{i \in I}^{x \in \lambda_0 (i), f \in F}\phi_{x, f} \circ \pi_i^{S(\Lambda^{\lt}) \to \C F}, \]
\[ \underset{\to} \Lim \C F_i := \bigg(\underset{\to} \Lim \lambda_0 (i), \bigvee_{\mathsmaller{\Theta
\in \prod_{i \in I}^{\lt}F_i}}\eql_0 f_{\Theta}\bigg), \]
\[ \big(\underset{\to} \Lim \C F_i\big) \to \C F := \bigg(\Mor(\underset{\to} 
\Lim \C F_i, \C F), \bigvee_{\mathsmaller{\eql_0^{\Lambda^{\lt}}(i, x) \in 
\underset{\to} \Lim \lambda_0 (i)}}^{f \in F}
\phi_{\mathsmaller{\eql_0^{\Lambda^{\lt}}(i, x)}, f}\bigg), \]
\[ \phi_{\mathsmaller{\eql_0^{\Lambda^{\lt}}(i, x)}, f}(h) := (f \circ h)\big(\eql_0^{\Lambda^{\lt}}(i, x)\big) \]
\begin{center}
\begin{tikzpicture}

\node (E) at (0,0) {$\Real$.};
\node[above=of E] (F) {$X$};
\node[left=of F] (B) {};
\node[left=of B] (A) {$\underset{\to} \Lim \lambda_0 (i)$};

\draw[->] (F)--(E) node [midway,right] {$f \in F$};
\draw[->] (A)--(E) node [midway,left] {$f \circ h \ \ \ $};
\draw[->] (A)--(F) node [midway,above] {$\mathsmaller{h \in \Mor(\underset{\to} \Lim \C F_i, \C F)}$};

\end{tikzpicture}
\end{center}
If $H \in \prod_{i \in I}^{\mt}\Mor(\C F_i, \C F)$, let the operation $\theta(H) :
\underset{\to} \Lim \lambda_0(i) \sto X$, defined by
\[ \theta(H)\big(\eql_0^{\Lambda^{\lt}}(i, x)\big) := H_i(x); \ \ \ \ \eql_0^{\Lambda^{\lt}}(i, x) \in 
\underset{\to} \Lim \lambda_0(i). \]
We show that $\theta(H)$ is a function. If
\[ \eql_0^{\Lambda^{\lt}}(i, x) =_{\mathsmaller{\underset{\to} \Lim \lambda_0 (i)}} \eql_0^{\Lambda^{\lt}}(j, y) \TOT
\exists_{k \in I}\big(i, j \lt k \ \& \ \lambda_{ik}^{\lt}(x) =_{\mathsmaller{\lambda_0 (k)}} 
\lambda_{jk}^{\lt}(y)\big),\]
we show that
$ \theta(H)\big(\eql_0^{\Lambda^{\lt}}(i, x)\big) :=  H_i(x) =_X H_j(y) =:
\theta(H)\big(\eql_0^{\Lambda^{\lt}}(j, y)\big)$. 
By the equalities
$H_i = \big(\lambda_{ik}^{\lt}\big)^+(H_k) = H_k \circ \lambda_{ik}^{\lt}$ and  
$H_j = \big(\lambda_{jk}^{\lt}\big)^+(H_k) = H_k \circ \lambda_{jk}^{\lt}$ we get
\[ H_i(x) = \big(H_k \circ \lambda_{ik}^{\lt}\big)(x) := H_k\big(\lambda_{ik}^{\lt}(x)\big) =_X 
H_k\big(\lambda_{jk}^{\lt}(y)\big) := \big(H_k \circ \lambda_{jk}^{\lt}\big)(y) := H_j(y). \]
Next we show that 
$\theta(H) \in \Mor(\underset{\to} \Lim \C F_i, \C F) :\TOT \forall_{f \in F}\big(f 
\circ \theta(H) \in \underset{\to} \Lim  F_i\big)$. 
If $f \in F$, then the dependent operation
$\Theta :\bigcurlywedge_{i \in I}F_i$, defined by $\Theta_i := f \circ H_i$, for every $i \in I$ 
\begin{center}
\begin{tikzpicture}

\node (E) at (0,0) {$\Real$};
\node[above=of E] (F) {$X$};
\node[left=of F] (B) {};
\node[left=of B] (A) {$\lambda_0 (i)$};

\draw[->] (F)--(E) node [midway,right] {$f \in F$};
\draw[->] (A)--(E) node [midway,left] {$f \circ H_i \ \ \ $};
\draw[->] (A)--(F) node [midway,above] {$\mathsmaller{H_i \in \Mor(\C F_i, \C F)}$};

\end{tikzpicture}
\end{center}
is in $\prod_{i \in I}^{\lt}F_i$ i.e., if $i \lt j$, then $\Theta_i = \big(\lambda_{ij}^{\lt}\big)^*(\Theta_j)
= \Theta_j \circ \lambda_{ij}^{\lt}$, since 
$\Theta_i := f \circ H_i = f \circ \big(H_j \circ \lambda_{ij}^{\lt}\big) = (f \circ H_j) \circ 
\lambda_{ij}^{\lt} := \Theta_j \circ \lambda_{ij}^{\lt}$.
Hence $f \circ \theta(H) := \eql_0 f_{\Theta} \in \underset{\to} \Lim  F_i$, since 
\[
[f \circ \theta(H)]\big(\eql_0^{\Lambda^{\lt}}(i, x)\big) := f\big(H_i(x)\big)
:= (f \circ H_i)(x)
:= f_{\Theta}(i, x)
:= \eql_0 f_{\Theta}\big(\eql_0^{\Lambda^{\lt}}(i, x)\big).
\]
Consequently, the operation $\theta \colon \prod_{i \in I}^{\mt}\Mor(\C F_i, \C F)
\sto \Mor(\underset{\to} \Lim \C F_i, \C F)$, defined by the rule $H \mapsto \theta(H)$,
is well-defined. Next we show that $\theta$ is an embedding. 
\begin{align*}
\theta(H) = \theta(K) & : \TOT \forall_{\mathsmaller{\eql_0^{\Lambda^{\lt}}(i, x) \in 
\underset{\to} \Lim \lambda_0 (i)}}\big(\theta(H)(\eql_0^{\Lambda^{\lt}}(i, x)) =  
\theta(K)({\eql_0}_{S^{\lt}}(i, x))\big)\\
& : \TOT \forall_{i \in I}\big(H_i(x) =_X K_i(x)\big)\\
& : \TOT H = K.
\end{align*}
Next we show that $\theta \in \Mor\big(\underset{\ot} \Lim (\C F_i \to \C F), 
(\underset{\to} \Lim \C F_i) \to \C F\big)$ i.e.,
\[ \forall_{\mathsmaller{\eql_0^{\Lambda^{\lt}}(i, x) \in \underset{\to} \Lim 
\lambda_0 (i)}}\forall_{f \in F}\bigg(\phi_{\mathsmaller{\eql_0^{\Lambda^{\lt}}(i, x)}, f} \circ 
\theta \in  \bigvee_{i \in I, x \in \lambda_0 (i)}^{f \in F}\phi_{x, f} \circ \pi_i^{S(\Lambda^{\lt})
\to \C F} \bigg). \]
By the equalities
 \[ [\phi_{\mathsmaller{\eql_0^{\Lambda^{\lt}(i, x)}, f}} \circ \theta](H) := 
\phi_{\mathsmaller{\eql_0^{\Lambda^{\lt}}(i, x), f}}(\theta(H))
:= [f \circ \theta(H)]\big(\eql_0^{\Lambda^{\lt}}(i, x)\big) 
:= f \big(H_i(x)\big), \]
\[ [\phi_{x, f} \circ \pi_i^{S^{\lt} \to \C F}](H) := \phi_{x, f}\big(H_i\big)
:= f\big(H_i(x)\big),  \]
we get $\phi_{\mathsmaller{\eql_0^{\Lambda^{\lt}}(i, x)}, f} \circ \theta = \phi_{x, f} \circ
\pi_i^{S(\Lambda^{\lt}) \to \C F}$. Let $\phi \colon \Mor(\underset{\to} \Lim \C F_i, \C F) \sto
\prod_{i \in I}^{\mt}\Mor(\C F_i, \C F)$ be defined by $h \mapsto \phi(h) := H^h$, where 
$ H^{h} : \bigcurlywedge_{i \in I}\Mor(\C F_i, \C F)$ is defined by  
$H^{h}_i := h \circ \eql_i$, for every $i \in I$
\begin{center}
\begin{tikzpicture}

\node (E) at (0,0) {$X$.};
\node[above=of E] (F) {$\underset{\to} \Lim \lambda_0 (i)$};
\node[left=of F] (B) {};
\node[left=of B] (A) {$\lambda_0 (i)$};

\draw[->] (F)--(E) node [midway,right] {$h$};
\draw[->] (A)--(E) node [midway,left] {$H_i^h \ \ \ $};
\draw[->] (A)--(F) node [midway,above] {$\eql_{i}$};

\end{tikzpicture}
\end{center}
By Proposition~\ref{prp: universaldirect}(i) $H_i \in \Mor(\C F_i, \C F) $,
as a composition of Bishop morphisms. To show that $H^{h} \in 
\prod_{i \in I}\Mor(\C F_i, \C F)$, let $i \lt j$, and by Proposition~\ref{prp: universaldirect}(ii)
we get
$H^{h}_i := h \circ \eql_i = h \circ \big(\eql_j \circ \lambda_{ij}^{\lt}\big) :=
(h \circ \eql_j) \circ \lambda_{ij}^{\lt} := H_j \circ \lambda_{ij}^{\lt}$.
Moreover, $\theta(H_{h}) := h$, since
$ \theta(H_{h})\big(\eql_0^{\Lambda^{\lt}}(i, x)\big) := H_{i}(x) 
:= (h \circ \eql_{i}(x) := h \big(\eql_0^{\Lambda^{\lt}}(i, x)\big)$.
Clearly, $\phi$ is a function. Moreover $H^{\theta(H)} := H$, as, for every $i \in I$ we have that
$ \big(H_i^{\theta(H)}\big)(x) := (\theta(H) \circ \eql_i)(x) 
:= \theta(H)\big(\eql_0^{\Lambda^{\lt}}(i,x)\big) := H_i(x)$.
Finally we show that $\phi \in \Mor\big((\underset{\to} \Lim \C F_i) \to \C F, \underset{\ot}
\Lim (\C F_i \to \C F)\big)$
if and only if 
\[\forall_{i \in I}\forall_{x \in \lambda_0(i)}\forall_{f \in F}\bigg(\phi_{x,f} \circ \phi \in
\bigvee_{\eql_0^{\Lambda^{\lt}}(i,x) \in \underset{\to}  \Lim \lambda_0(i)}^{f \in F}
\phi_{\eql_0^{\Lambda^{\lt}}(i,x), f}\bigg). \]
If $h \in \Mor(\underset{\to} \Lim \C F_i, \C F)$, then
\[ \big[\phi_{x,f} \circ \pi_i^{S(\Lambda^{\lt}) \to F}) \circ \phi\big](h) := (\phi_{x,f} \circ \pi_i^{S(\Lambda^{\lt}) 
\to F})(H^h)
 := \phi_{x,f}\big(H^h_i\big) \]
\[ =: \phi_{x, f}(h \circ \eql_i)
:= f\big[(h \circ \eql_i)(x)\big]
 := (f \circ h)\big(\eql_0^{\Lambda^{\lt}}(i,x)\big)
  := \phi_{\eql_0^{\Lambda^{\lt}}(i,x),f}(h). \qedhere
\]
\end{proof}

With respect to the possible dual to the previous theorem i.e., the isomorphism 
$\underset{\to} \Lim (\C F_i \to \C F) \simeq [(\underset{\ot} \Lim \C F_i) \to \C F]$, what we can 
show is the following proposition.

\begin{prop}\label{prp: conversedual}
Let
$S(\Lambda^{\mt}) := (\lambda_0, \lambda_1^{\mt}, \phi_0^{\Lambda^{\mt}},
\phi_1^{\Lambda^{\mt}})$ be in $\Spec(I, \mt)$ with Bishop spaces 
$(\C F_i)_{i \in I}$ and Bishop morphisms
$(\lambda_{ji}^{\succ})_{(i,j) \in D^{\lt}(I)}$. If $\C F := (X, F)$ is a Bishop space and 
$S(\Lambda^{\mt}) \to \C F := (\mu_0, \mu_1^{\lt}, \phi_0^{M^{\lt}}, \phi_1^{M^{\lt}})$ is the $(I, \lt)$-directed
spectrum defined in Proposition~\ref{prp: dirtoproj} \normalfont (B)(i), 
\itshape
there is  a function 
$\ \widehat{} \ \colon \underset{\to} \Lim [\Mor(\C F_i, \C F)] \to \Mor(\underset{\ot} \Lim \C F_i, \C F)$
such that the following hold:\\[1mm]
\normalfont (i)
\itshape $\ \widehat{} \ \in \Mor\big(\underset{\to} \Lim (\C F_i \to \C F), 
(\underset{\ot} \Lim \C F_i) \to \C F\big)$.\\[1mm]
\normalfont (ii)
\itshape  If for every $j \in J$ and every $y \in \lambda_0 (j)$ there is 
$\Theta_y \in \prod_{i \in I}^{\mt}\lambda_0 (i)$ 
such that $\Theta_y(j) =_{\mathsmaller{\lambda_0 (j)}} y$, 
then $\ \widehat{} \ $ is an embedding of $\underset{\to} \Lim [\Mor(\C F_i, \C F)]$ into 
$\Mor(\underset{\ot} \Lim \C F_i, \C F)$.
\end{prop}

\begin{proof}
We proceed similarly to the proof of Theorem~\ref{thm: duality1}. 
\end{proof}

\begin{thm}\label{thm: duality2}
Let $S(\Lambda^{\mt}) := (\lambda_0, \lambda_1^{\mt}, \phi_0^{\Lambda^{\mt}},
\phi_1^{\Lambda^{\mt}})$ be in $\Spec(I, \mt)$ with Bishop spaces 
$(\C F_i)_{i \in I}$ and Bishop morphisms $(\lambda_{ji}^{\prec})_{(i,j) \in D^{\lt}(I)}$.
If $\C F := (X, F)$ is a Bishop space and 
$\C F \to S(\Lambda^{\mt}) := (\nu_0, \nu_1^{\mt}, \phi_0^{^{\mt}}, \phi_1^{N^{\mt}})$ is the contravariant
direct spectrum over $(I, \lt)$, defined in Proposition~\ref{prp: dirtoproj} \normalfont (B)(ii), 
\itshape then 
\[ \underset{\ot} \Lim (\C F \to \C F_i) \simeq [\C F \to \underset{\ot} \Lim \C F_i]. \]
\end{thm}

\begin{proof}

First we determine the topologies involved in the required Bishop isomorphism:
\[ \underset{\ot} \Lim (\C F \to \C F_i) := \bigg(\prod_{i \in I}^{\mt}\Mor(\C F, \C F_i), 
\bigvee_{i \in I}^{g \in F \to F_i}g \circ \pi_i^{\C F \to S(\Lambda^{\mt})}\bigg), \]
\[ \bigvee_{i \in I}^{g \in F \to F_i}g \circ \pi_i^{\C F \to S^{\mt}} = \bigvee_{i \in I, x \in 
\lambda_0 (i)}^{f \in F_i}\phi_{x, f} \circ \pi_i^{\C F \to S(\Lambda^{\mt})}, \]
\[ \underset{\ot} \Lim \C F_i := \bigg(\prod_{i \in I}^{\mt}\lambda_0 (i), 
\bigvee_{i \in I}^{f \in F_i}f \circ \pi_i^{S(\Lambda^{\mt})}\bigg), \]
\[ \C F \to \underset{\ot} \Lim \C F_i := \bigg(\Mor(\C F, \underset{\ot} \Lim \C F_i), 
\bigvee_{x \in X}^{g \in \underset{\ot} \Lim \C F_i}\phi_{x, g}\bigg), \]
 \[ \bigvee_{x \in X}^{g \in \underset{\ot} \Lim \C F_i}\phi_{x, g} = 
\bigvee_{x \in X, i \in I}^{f \in F_i}\phi_{x, f \circ \pi_i^{S(\Lambda^{\mt})}}. \]
If $H \in \prod_{i \in I}^{\mt}\Mor(\C F, \C F_i)$, and if $i \lt j$, then
$H_i = \nu_{ji}^{\mt}(H_j) = \big(\lambda_{ji}^{\mt}\big)^-(H_j) = \lambda_{ji}^{\mt} \circ H_j $
\begin{center}
\begin{tikzpicture}

\node (E) at (0,0) {$\lambda_0(i)$.};
\node[above=of E] (F) {$\lambda_0 (j)$};
\node[left=of F] (A) {$X$};

\draw[->] (F)--(E) node [midway,right] {$\lambda_{ji}^{\mt}$};
\draw[->] (A)--(E) node [midway,left] {$H_i \  $};
\draw[->] (A)--(F) node [midway,above] {$H_j$};

\end{tikzpicture}
\end{center}
Let the operation $e(H) : X \sto \prod_{i \in I}^{\mt} \lambda_0 (i)$, defined by
$x \mapsto [e(H)](x)$, where $\big[[e(H)](x)\big]_i := H_i(x)$, for every $i \in I$.
First we show that $[e(H)](x) \in \prod_{i \in I}^{\mt} \lambda_0 (i)$. If $i \lt j$, then 
$\big[[e(H)](x)\big]_i := H_i(x) = \big(\lambda_{ji}^{\mt} \circ H_j\big)(x) := 
\lambda_{ji}^{\mt}\big(H_j(x)\big) := \lambda_{ji}^{\mt}\big(\big[[e(H)](x)\big]_j\big)$.
Next we show that $e(H)$ is a function. If $x =_X x{'}$, then
$\forall_{i \in I}\big(H_i(x) =_{\mathsmaller{\lambda_0 (i)}} H_i(x{'})\big) : \TOT
\forall_{i \in I}\big(\big[[e(H)](x)\big]_i  =_{\mathsmaller{\lambda_0 (i)}} \big[[e(H)](x{'})\big]_i\big) 
 : \TOT [e(H)](x) =_{\mathsmaller{\prod_{i \in I}^{\mt}\lambda_0 (i)}} [e(H)](x{'})$.
By the $\bigvee$-lifting of morphisms  
$e(H) \in \Mor(\C F, \underset{\ot} \Lim \C F_i) \TOT 
\forall_{i \in I}\forall_{f \in F_i}\big(\big(f \circ \pi_i^{S(\Lambda^{\mt})}\big) \circ e(H) \in F\big).$ 
Since $ [(f \circ \pi_i^{S(\Lambda^{\mt})}) \circ e(H)](x) := 
\big(f \circ \pi_i^{S(\Lambda^{\mt})}\big)\big([e(H)](x)\big) 
:= f \big(H_i(x)\big) := (f \circ H_i)(x)$, 
we get $\big(f \circ \pi_i^{S(\Lambda^{\mt})}\big) \circ e(H) :=  f \circ H_i \in F$, since $f \in F_i$ 
and $H_i \in \Mor(\C F, \C F_i)$. Hence, the 
operation $e : \prod_{i \in I}^{\mt}\Mor(\C F, \C F_i) \sto \Mor(\C F, 
\underset{\ot} \Lim \C F_i)$, defined by the rule $H \mapsto e(H)$, 
is well-defined. Next we show that $e$ is an embedding. If $H, K \in \prod_{i \in I}^{\mt}\Mor(\C F, \C F_i)$,
then
\begin{align*}
e(H) = e(K) & : \TOT \forall_{x \in X}\big([e(H)](x) =_{\mathsmaller{\prod_{i \in I}^{\mt}\lambda_0 (i)}}
[e(K)](x)\big)\\
& : \TOT \forall_{x \in X}\forall_{i \in I}\big(H_i(x)  =_{\mathsmaller{\lambda_0 (i)}} K_i(x)\big)\\
& : \TOT \forall_{i \in I}\forall_{x \in X}\big(H_i(x)  =_{\mathsmaller{\lambda_0 (i)}} K_i(x)\big)\\
& :\TOT \forall_{i \in I}\big(H_i =_{\mathsmaller{\Mor(\C F, \C F_i)}} K_i\big)\\
& :\TOT H =_{\mathsmaller{\prod_{i \in I}^{\mt}\Mor(\C F, \C F_i)}} K.
\end{align*}
By the $\bigvee$-lifting of morphisms we show that 
\[ e \in \Mor(\underset{\ot} \Lim (\C F \to \C F_i), \C F \to \underset{\ot} \Lim \C F_i) 
\TOT \forall_{i \in I}\forall_{f \in F_i}\big(\phi_{x, f \circ \pi_i^{S(\Lambda^{\mt})}} \circ e 
\in \underset{\ot} \Lim (F \to F_i)\big) \]
\begin{align*}
\big(\phi_{x, f \circ \pi_i^{S(\Lambda^{\mt})}} \circ e \big)(H) & := 
\phi_{x, f \circ \pi_i^{S(\Lambda^{\mt})}}\big(e(H)\big)\\
& =: \big[(f \circ \pi_i^{S(\Lambda^{\mt})}) \circ e(H)\big](x)\\
& := (f \circ \pi_i^{S(\Lambda^{\mt})})\big([e(H)](x)\big)\\
& := f \big([e(H)(x)]_i\big)\\
& := f \big(H_i(x)\big)\\
& =: (f \circ  H_i)(x)\\
& =: \phi_{x, f}\big(H_i\big)\\
& =: \big[\phi_{x, f} \circ \pi_i^{\C F \to S(\Lambda^{\mt})}\big](H)
\end{align*}
we get $\phi_{x, f \circ \pi_i^{S(\Lambda^{\mt})}} \circ e := \phi_{x, f} \circ \pi_i^{\C F \to S(\Lambda^{\mt})} 
\in \underset{\ot} \Lim (F \to F_i)$. Let 
$\phi \colon \Mor(\C F, \underset{\ot} \Lim \C F_i) \sto \prod_{i \in I}^{\mt} \Mor(\C F, \C F_i)$, 
defined by the rule $\mu \mapsto H^{\mu}$, where 
for every $\mu : X \to \prod_{i \in I}^{\mt}\lambda_0 (i) \in \Mor(\C F, \underset{\ot} \Lim \C F_i)$ i.e.,
$\forall_{i \in I}\forall_{f \in F_i}\big(\big(f \circ \pi_i^{S(\Lambda^{\mt})}\big) \circ \mu \in F\big)$, let
\[ H^{\mu} : \bigcurlywedge_{i \in I}\Mor(\C F, \C F_i), \ \ \ [H_{\mu}]_i \colon X \to \lambda_0(i),  \ \ \ \ 
 H^{\mu}_i(x) := [\mu(x)]_i;  \ \ \ \ x \in X, \ i \in I. \]
First we show that $H^{\mu}_i \in \Mor(\C F, \C F_i) :\TOT
\forall_{f \in F_i}\big(f \circ H^{\mu}_i \in F\big)$. If $f \in F_i$, and $x \in X$, then 
$[f \circ H^{\mu}_i(x) := f\big(H^{\mu}_i\big) := f\big([\mu(x)]_i\big)
=: \big[\big(f \circ \pi_i^{S^(\Lambda{\mt})}\big) \circ \mu\big](x)$
i.e., $f \circ H^{\mu}_i := \big(f \circ \pi_i^{S(\Lambda^{\mt})}\big) \circ \mu \in F$,
as $\mu \in \Mor(\C F, \underset{\ot} \Lim \C F_i)$.
Since $\mu(x) \in \prod_{i \in I}^{\mt}\lambda_0 (i)$, 
$[\mu(x)]_i = \lambda_{ji}^{\mt}\big([\mu(x)]_j\big),$
for every $i, j \in I$ such that $i \lt j$. To show that $H^{\mu} \in \prod_{i \in I}^{\mt}\Mor(\C F, \C F_i)$,
let $i \lt j$. Then
\begin{align*}
H^{\mu}_i = \lambda_{ji}^{\mt} \circ H^{\mu}_j & \TOT \forall_{x \in X}\big(H^{\mu}_i(x) 
=_{\mathsmaller{\lambda_0(i)}} \big[\lambda_{ji}^{\mt} \circ H^{\mu}_j\big](x)\big)\\
& : \TOT \forall_{x \in X}\big([\mu(x)]_i 
=_{\mathsmaller{\lambda_0(i)}} \big[\lambda_{ji}^{\mt}\big([\mu(x)]_j\big)\big),
\end{align*}
which holds by the previous remark on $\mu(x)$. It is immediate to show that $\phi$ is a function. 
To show that 
$\phi \in \Mor \big([ \C F \to \underset{\ot} \Lim \C F_i ], \underset{\ot} \Lim (\C F \to \C F_i)\big)$, 
we show that
\[ \forall_{i \in I}\forall_{f \in F_i}\forall_{x \in \lambda_0(i)}\bigg(\big[\phi_{x,f} \circ 
\pi_i^{\C F \to S(\Lambda^{\mt})}\big] \circ \phi \in \bigvee_{x \in X, i \in I}^{f \in F_i}\phi_{x, f 
\circ \pi_i^{S(\Lambda^{\mt})}}\bigg), \]
\begin{align*}
\big[\big[\phi_{x,f} \circ \pi_i^{\C F \to S(\Lambda^{\mt})}\big] \circ \phi \big](\mu) & := 
\big[\phi_{x,f} \circ \pi_i^{\C F \to S(\Lambda^{\mt})}\big](H^{\mu})\\
& := \phi_{x,f}\big(H^{\mu}_i\big)\\
& =: (f \circ H_i^{\mu})(x) \\
& := f \big(\mu(x)_i\big)\\
& =: (f \circ \pi_i^{S(\Lambda^{\mt})} \circ \mu)(x)\\
& =: \big[\phi_{x, f \circ \pi_i^{S(\Lambda^{\mt})}}\big](\mu).
\end{align*}
Moreover, $\phi(e(H)) := H$, as $H^{e(H)}_i(x) := [e(H)(x)]_i := H_i(x)$, and $e(\phi(\mu)) = \mu$, as
\begin{align*}
e(H^{\mu}) = \mu & :\TOT \forall_{x \in X}\big([e(H^{\mu})](x) 
=_{\mathsmaller{\prod_{i \in I}^{\mt}\lambda_0 (i)}}
\mu(x)\big)\\
& :\TOT \forall_{x \in X}\forall_{i \in I}\big(H^{\mu}_i(x) =_{\mathsmaller{\lambda_0 (i)}}
[\mu(x)]_i\big)\\
& :\TOT \forall_{x \in X}\forall_{i \in I}\big([\mu(x)]_i =_{\mathsmaller{\lambda_0 (i)}}
[\mu(x)]_i\big). \qedhere
\end{align*}
\end{proof}

With respect to the possible dual to the previous theorem i.e., the isomorphism
$\underset{\to} \Lim (\C F \to \C F_i) \simeq (\C F \to \underset{\to} \Lim \C F_i)$, 
what we can show is the following proposition.

\begin{prop}\label{prp: conversedual2}
Let $S(\Lambda^{\lt}) := (\lambda_0, \lambda_1^{\lt}, \phi_0^{\Lambda^{\lt}},
\phi_1^{\Lambda^{\lt}}) \in \Spec(I, \lt)$ with Bishop spaces $(\C F_i)_{i \in I}$
and Bishop morphisms $(\lambda_{ij}^{\lt})_{(i,j) \in D^{\lt}(I)}$. If $\C F := (X, F)$ is a Bishop space and 
$\C F \to S(\Lambda^{\lt}) := (\nu_0, \nu_1^{\lt}, \phi_0^{N^{\lt}}, \phi_1^{N^{\lt}})$ is the 
$(I, \lt)$-direct spectrum defined in Proposition~\ref{prp: dirtoproj} \normalfont (A)(ii),
\itshape 
there is  a map $\ \widehat{} \ : \underset{\to} \Lim [\Mor(\C F, \C F_i)] \to 
\Mor(\C F, \underset{\to} \Lim \C F_i)$
with $\ \widehat{} \ \in \Mor\big((\underset{\to} \Lim (\C F \to \C F_i), \C F \to \underset{\to} 
\Lim \C F_i\big)$.
\end{prop}

\begin{proof}
We proceed similarly to the proof of Theorem~\ref{thm: duality2}. 
\end{proof}

\section{Concluding remarks}
\label{sec: concl}

In this paper we tried to show how some fundamental notions of Bishop Set Theory, a ``completion'' of Bishop's theory 
of sets that underlies Bishop-style constructive mathematics $\BISH$,
can be applied to the constructive topology of 
Bishop spaces, and especially in the theory of limits of Bishop spaces. 
For that we defined the notion of a direct family of sets, and by associating to such a family of sets
a family of Bishop spaces the notions of a direct and 
contravariant direct spectrum of Bishop spaces emerged. The definition of sum Bishop topology
(Definition~\ref{def: topondirectedsigma}) on the corresponding direct sum shows the 
harmonious relation between the notion of Bishop set
and that of Bishop space.  
In the subsequent sections we proved constructively the translation of all fundamental results in 
the theory of limits of topological spaces into the theory of Bishop spaces.

All notions of families of sets studied here have their generalised counterpart i.e., 
we can define generalised  $I$-families of sets, or generalised families of sets over a directed 
set $(I, \lt)$, where more than one transport maps from $\lambda_0(i)$ to $\lambda_0(j)$ are permitted
(see~\cite{Pe20}, section 3.9). 
The corresponding notions of generalised spectra of Bishop spaces and their
limits can be studied, and all major results of the
previous sections are expected to be extended to the case of generalised spectra of Bishop spaces.
In~\cite{Pe20}, section 6.9, we introduce the notion of a direct spectrum of Bishop subspaces. As in the case of 
set-indexed families of subsets, the main properties of the direct spectra of Bishop subspaces are determined internally.

We hope to examine interesting applications of the theory of spectra of Bishop spaces 
in future work. 

\section*{Acknowledgements}
A large part of this paper was completed during our research visit to the Department of Mathematical Sciences 
of the University of Cincinnati that was funded by the Eupean-Union project ``Computing with Infinite Data''.
We would like to thank Ning Zhong for hosting us in UC.

I would also like to thank the reviewers of this paper for their very helpful comments and suggestions.


\end{document}